\documentclass[11pt]{amsart}
\numberwithin{equation}{section}
\usepackage[english]{babel}
\usepackage[T1]{fontenc}
\usepackage[latin1]{inputenc}
\usepackage{indentfirst}
\usepackage{enumitem}
\usepackage{amsmath,amssymb, amsbsy}
\usepackage{comment}
\usepackage{amsfonts}
\usepackage{hyperref}
\usepackage{cleveref}
\usepackage{esint}
\usepackage{latexsym}
\usepackage{amsthm}
\usepackage[dvips]{graphicx}
\usepackage{xcolor}
\usepackage{tikz}
\usepackage{pgfplots}
\usetikzlibrary{intersections}
\usepackage{tikz-3dplot}
\usepackage[outline]{contour}
\DeclareGraphicsExtensions{.pdf,.png,.jpg,.eps}

\usepackage{bbm}

\usepackage[font=small,labelfont=bf]{caption}

\usepackage{amsthm}

\usepackage{hyperref}

\usepackage{xcolor}
\usepackage[latin1]{inputenc}
\usepackage[active]{srcltx}
\definecolor{Arancio}{cmyk}{0,0.61,0.87,0}

\definecolor{blus}{RGB}{0,102,204}

\newcommand{\capa}{{\rm Cap}}
\usepackage{doi}

\let\arXiv\arxiv

\newcommand{\brd}[1]{\mathbb{#1}}
\newcommand{\R}{\brd{R}}

\newcommand{\N}{\brd{N}}
\newcommand{\Z}{\brd{Z}}
\newcommand{\e}{\varepsilon}

\newcommand{\eps}{\varepsilon}
\newcommand{\be}{\begin{equation}}
\newcommand{\ee}{\end{equation}}

\newtheorem{teo}{Theorem}[section]
\newtheorem{Corollary}[teo]{Corollary}
\newtheorem{Lemma}[teo]{Lemma}
\newtheorem{Theorem}[teo]{Theorem}
\newtheorem{Proposition}[teo]{Proposition}
\theoremstyle{definition}
\newtheorem{Definition}[teo]{Definition}

\newtheorem{remark}[teo]{Remark}

\allowdisplaybreaks
\newcommand{\supp}{\operatorname{spt}}

\newcommand{\D}{\nabla}
\newcommand{\dive}{\operatorname{div}}
\DeclareRobustCommand{\rchi}{{\mathpalette\irchi\relax}}
\newcommand{\irchi}[2]{\raisebox{\depth}{$#1\chi$}}
\newcommand{\dist}{\operatorname{dist}}

\usepackage{amsmath}
\usepackage{tikz}
\usepackage{mathdots}
\usepackage{yhmath}
\usepackage{cancel}
\usepackage{color}
\usepackage{siunitx}
\usepackage{array}
\usepackage{multirow}
\usepackage{amssymb}
\usepackage{textcomp}
\usepackage{gensymb}
\usepackage{tabularx}
\usepackage{extarrows}
\usepackage{booktabs}
\usetikzlibrary{fadings}
\usetikzlibrary{patterns}
\usetikzlibrary{shadows.blur}
\usetikzlibrary{shapes}
\pgfplotsset{compat=1.18} 
\begin{document}


\title[Elliptic equations degenerating on lower dimensional manifolds]
{Schauder estimates for elliptic equations degenerating on lower dimensional manifolds}

\author{Gabriele Cora, Gabriele Fioravanti, Stefano Vita}

\address{G. Cora, D\'epartement de Math\'ematique, Universit\'e Libre de Bruxelles, Boulevard du Triomphe 155, 1050, Brussels, Belgium}
\email{gabriele.cora@ulb.be}

\address{G. Fioravanti, Dipartimento di Matematica "G. Peano", Universit\`a degli Studi di Torino, Via Carlo Alberto 10, 10124, Torino, Italy}
\email{gabriele.fioravanti@unito.it}

\address{S. Vita, Dipartimento di Matematica "F. Casorati", Universit\`a di Pavia, Via Ferrata 5, 27100, Pavia, Italy} \email{stefano.vita@unipv.it}




\begin{abstract}
In this paper we begin exploring a local regularity theory for elliptic equations having coefficients which are degenerate or singular on some lower dimensional manifold
\begin{equation*}
-\mathrm{div}(|y|^aA(x,y)\nabla u)=|y|^af+\mathrm{div}(|y|^aF)\qquad\mathrm{in \ } B_1\subset\mathbb R^d,
\end{equation*}
where $z=(x,y)\in\mathbb R^{d-n}\times\mathbb R^n$, $2\leq n\leq d$ are two integers and $a\in\mathbb R$. Such equations are a prototypical example of elliptic equations spoiling their uniform ellipticity on the (possibly very) thin characteristic manifold $\Sigma_0=\{|y|=0\}$ of dimension $0\leq d-n\leq d-2$, having
$$\lambda|y|^a|\xi|^2\leq |y|^aA(x,y)\xi\cdot\xi\leq\Lambda|y|^a|\xi|^2.$$
Whenever $a+n>0$, the weak solutions with a homogeneous conormal boundary condition at $\Sigma_0$ are provided to be $C^{0,\alpha}$ or even $C^{1,\alpha}$ regular up to $\Sigma_0$. Our approach relies on a regularization-approximation scheme which employs domain perforation, very fine blow-up procedures, and a new Liouville theorem in the perforated space. 
Our theory extends to the case of equations degenerating on suitably smooth curved manifolds.
\end{abstract}

\maketitle

\section{Introduction}
Aim of this work is the study of the local regularity properties, H\"older $C^{0,\alpha}$ and Schauder $C^{1,\alpha}$ estimates, of weak solutions to
\begin{equation}\label{generalPDE}
-\mathrm{div}(|y|^aA(x,y)\nabla u)=|y|^af+\mathrm{div}(|y|^aF)\qquad\mathrm{in \ } B_1\subset\R^d.
\end{equation}
Here $z=(x,y)\in\R^{d-n}\times\R^n$, $2\leq n\leq d$ are two integers and $a\in\R$. The second order equation in divergence form above is uniformly elliptic
far from a characteristic flat manifold of low dimension $0\leq d-n\leq d-2$
$$\Sigma_0=\{z=(x,y)\in\R^d \, \mid \, |y|=0\},$$
and the weight term is a power of the distance to $\Sigma_0$; that is, $|y|=\mathrm{dist}(z,\Sigma_0)$. In other words, there exist $0<\lambda\leq\Lambda$ such that the symmetric $(d\times d)$-dimensional matrix $A(z)=(a_{ij}(z))_{i,j=1,...,d}$ satisfies 
\[
\lambda|y|^a|\xi|^2\leq |y|^aA(z)\xi\cdot\xi\leq\Lambda|y|^a|\xi|^2\qquad\mathrm{for \ almost \ any \ }z\in B_1, \, \mathrm{for \ any \ }\xi\in\R^d.
\]
The weak solutions to the above problem are elements of the weighed Sobolev space $H^{1,a}(B_1):=H^1(B_1,|y|^adz)$ which satisfy
\begin{equation}\label{weakconormal}
\int_{B_1}|y|^aA\nabla u\cdot\nabla\phi=\int_{B_1}|y|^a(f\phi-F\cdot\nabla\phi)\qquad\mathrm{for \ every \ }\phi\in C^{\infty}_c(B_1).
\end{equation}
The equation is satisfied \emph{across} the thin manifold $\Sigma_0$, and this implies a formal homogeneous conormal condition at $\Sigma_0$
\begin{equation}\label{conormalweighted}
\lim_{|y|\to0}|y|^{a+n-1}(A\nabla u+F)\cdot\frac{y}{|y|}=0,
\end{equation}
see Remark \ref{rem:conormal} for the precise meaning of the above expression. Notice also that, when $A=\mathbb I$ and $F=0$, \eqref{conormalweighted} corresponds to a vanishing weighted radial derivative with respect to the degenerate variables.

The weighted $H^{1,a}$-capacity of the thin manifold $\Sigma_0$ is key to understanding what kind of solutions one may face. Depending on the value of the parameter $a+n\in\R$, the local weighed capacity of the characteristic manifold is infinite when $a+n\leq0$, positive and finite when $a+n\in(0,2)$ and zero when $a+n\geq2$. Consequently, weak solutions must vanish at $\Sigma_0$ in the \emph{supersingular} case $a+n\leq0$, and naturally satisfy \eqref{conormalweighted} in the \emph{superdegenerate} case $a+n\geq2$. Finally, in the \emph{mid-range} case $a+n\in(0,2)$, many different boundary conditions can be prescribed at $\Sigma_0$, and this corresponds to inhomogeneous Dirichlet and inhomogeneous conormal boundary problems (oblique derivative type conditions).

In this paper, we mostly work under the assumption $a+n>0$ - which makes the weight term locally integrable - and deal with functions which solve \eqref{generalPDE} across $\Sigma_0$ in the sense of \eqref{weakconormal}; that is, we study the homogeneous conormal boundary problem. The regularity for the inhomogeneous conormal problem in the mid-range $a+n\in(0,2)$ follows as a consequence of the theory for the homogeneous one and is treated in the subsequent \cite{CorFioVit25}. Concerning Dirichlet type boundary conditions, the inhomogeneous problem in the mid-range $a+n\in(0,2)$ is studied in the companion paper \cite{Fio24}, while the analysis of the homogeneous case, when $a+n<2$, is partially carried out in \cite{CorFioVit25} via a boundary Harnack type principle.

\subsection*{Background and motivations}
The local regularity theory for weighted degenerate elliptic equations 
starts with the seminal works by Fabes-Jerison-Kenig-Serapioni \cite{FabJerKen82,FabJerKen83,FabKenSer82}. Among their motivations for such a regularity theory, the connection with the study of fine qualitative properties of harmonic functions at the boundary of rough domains through quasiconformal mappings from a ball. In \cite{FabKenSer82}, the authors extend the De Giorgi-Nash-Moser theory to degenerate elliptic equations where the degeneracy/singularity is carried out by a weight term $\omega$ which arises either from quasiconformal mappings or belongs to the $A_2$-Muckenhoupt class. The latter indicates a combined local integrability property between the weight and its reciprocal. 
Their regularity theory includes Harnack inequalities and local H\"older continuity of solutions. These results rely primarily in the validity of a general functional framework in weighted Sobolev spaces comprehending Poincar\'e-Wirtinger and Sobolev inequalities. Later, the conditions which provide this setting have been summarized into four main properties of the weight, characterizing the so-called $2$-admissible weights, see for instance \cite{HeiKilMar06} and later discussions in Section \ref{sec:2admissible}. The general regularity theory for weighted equations also ends with \cite{FabKenSer82}, since H\"older continuity is the optimal regularity if no further geometric assumption on the zero/infinity set $\Sigma_0=\{\omega(z)=0 \, \vee \, \omega(z)=\infty\}$ is done.

\smallskip

In recent years, there have been significant contributions to the study of degenerate elliptic equations where the weight behaves like a power of the distance from a $(d-1)$-dimensional manifold.
The model equation is
\begin{equation}\label{eq:codim:1}
-\mathrm{div}(y^aA(x,y)\nabla u)=RHS \qquad\mathrm{in \ } B_1^+\subset\R^d_+.
\end{equation}
Here $z=(x,y)\in\R^{d-1}\times\R_+=\R^d_+$, $B_1^+=B_1\cap\{y>0\}$, $a\in\R$ and $\Sigma_0=\{y=0\}$ is an hyperplane, i.e. it has codimension $n=1$. This theory has profound connections with the edge calculus developed in \cite{Maz91,MazVer14}, see also \cite{Lin89,HanXie24a,HanXie24b} and the references therein. However, the recent interest in this problem primarily relies in the connection with fractional Laplacians due to an extension theory via Dirichlet-to-Neumann maps \cite{CafSil07}; that is, for $a\in(-1,1)$
\begin{equation*}
\begin{cases}
-\mathrm{div}(y^a\nabla u)=0 &\mathrm{in \ } \R^d_+\\
-\lim_{y\to0^+}y^a\partial_yu(x,y)=C_{a,d}(-\Delta)^{\frac{1-a}{2}}u(x,0) &\mathrm{in \ } \Sigma_0,
\end{cases}
\end{equation*}
see also \cite{ChaGon11} for the curved analogous in conformal geometry.

Concerning the general equation \eqref{eq:codim:1}, it is crucial to comment here its possible connection with Grushin operators. As it is highlighted in \cite{CafSil07}, a change of coordinates moves the operator into a Grushin type one. However, this is true only for a particular small range of the exponent $a$, and in any case both the transformation and the power term in the Grushin operator are not smooth. So, hypoellipticity and the H\"ormander regularity theory are not valid in general. In fact, the indicial root $y^{1-a}$ always solves the homogeneous ($RHS=0$) isotropic ($A=\mathbb I$) equation \eqref{eq:codim:1} whenever $a<1$. This prevents general \emph{harmonic functions for the degenerate operator} from being smooth.

However, the regularity of the solutions may strongly improve by imposing certain boundary conditions at $\Sigma_0$. In fact, solutions with a homogeneous conormal boundary condition at $\Sigma_0$ enjoy a Schauder theory in $C^{k,\alpha}$ spaces whenever $a>-1$, which in turns leads to smoothness if data are smooth. The full regularity theory in H\"older spaces is carried out in \cite{SirTerVit21a,SirTerVit21b,TerTorVit24a,DonJeoVit24}, see also \cite{AudFioVit24a,AudFioVit24b} for the parabolic case and \cite{DonPha21,DonPha23} for regularity estimates in Sobolev spaces. Notably, this theory finds nice applications in higher order boundary Harnack principles \cite{DesSav15} for uniformly elliptic equations \cite{TerTorVit24a,TerTorVit24b,Zha24}. The literature on the regularity in the mid-range case $a\in(-1,1)$ is extensive, primarly due again to the connection with fractional Laplacians. We cannot provide a comprehensive list of all contributions on this topic here; however, concerning Schauder type estimates, a notable reference can be found in \cite{CafSti16}.

\smallskip

Let us consider now the higher codimensional case presented in the present paper. The study of \eqref{generalPDE} can be seen at a first glance as a natural continuation of the codimension $1$ theory. However, there are many other motivations for developing such a regularity theory. We would like to focus our attention on some topics which strongly relates to the degenerate equation \eqref{generalPDE} and may have further developments due to our results and approach. Some of them are known topics, and have a literature: harmonic maps with prescribed singularities in general relativity, the Dirichlet problem and harmonic measure on lower dimensional boundaries. Some other connections are new: critical points for Caffarelli-Kohn-Nirenberg inequalities, higher codimensional extensions of fractional Laplacians, very thin free boundary problems. We will spend some other words on these topics in the last part of this introduction. Finally, we would like to mention some unique continuation results for Grushin-type operators which are degenerate on lower dimensional manifolds, see for instance \cite{AbaFerLuz24,Gar93,HuWuYanZho25} and many references therein.

\subsection*{Main results and strategies}


The main results we are presenting here concern the homogeneous conormal problem for \eqref{generalPDE} in the range $a+n>0$. The solutions we are referring to are elements of $H^{1,a}(B_1)$ which weakly solve \eqref{weakconormal}.
Our main objectives are certain H\"older $C^{0,\alpha}$ and Schauder $C^{1,\alpha}$ local regularity estimates up to $\Sigma_0$ for this class of weak solutions. In order to state precisely our main results, we need to be more specific regarding the uniform ellipticity properties of the variable coefficient matrix $A$. The latter is a symmetric $(d\times d)$-dimensional matrix satisfying the \emph{global} uniform ellipticity condition
\begin{equation}\label{eq:unif:ell}
    \lambda|\xi|^2\le A(z)\xi\cdot\xi\le\Lambda|\xi|^2,\qquad\text{for a.e. } z\in B_1 \text{ and for all } \xi\in \R^d,
\end{equation}
for some ellipticity constants $0<\lambda\leq\Lambda$. Moreover, $A$ satisfies the \emph{restricted-to-$\Sigma_0$} uniform ellipticity condition 
\begin{equation}\label{eq:unif:ell:loc}
    \lambda_*|\xi|^2\le A_3(x,0)\xi\cdot\xi\le\Lambda_*|\xi|^2,\qquad\text{for a.e. } (x,0)\in B_1\cap\Sigma_0 \text{ and for all } \xi\in \R^n,
\end{equation}
with ellipticity constants $0<\lambda\leq\lambda_*\leq\Lambda_*\leq\Lambda$, where $A_3$ is the $(n\times n)$-dimensional block located in the lower-right corner of the matrix $A$, see \eqref{Blocks}. Then, let us introduce the exponent
\begin{equation}\label{alphastar2}
    \alpha_*=\alpha_*(a,n,\lambda_*/\Lambda_*)= \frac{2-a-n + \sqrt{(2-a-n)^2 + 4\mu_* }}{2},
\end{equation}
with
\[
\mu_*=\mu_*(a,n,\lambda_*/\Lambda_*) =
\begin{cases}
\displaystyle \Big(\frac{\lambda_*}{\Lambda_*}\Big)^{\frac{|a|}{2}}(n-1) & \text{ if }n \geq 3\,,\\
 \displaystyle\Big({\frac{4}{\pi}}\arctan 
 \Big(\frac{\lambda_*}{\Lambda_*}\Big)^{\frac{|a|}{4}} 
 \Big)^2\, & \text{ if }n =2\,. 
\end{cases}
\]
Notice that both $\alpha_*$ and $\mu_*$ are monotone increasing with respect to $0<\lambda_*/\Lambda_*\leq1$. Notice also that in case $\lambda_*/\Lambda_*=1$ (which happens in particular if $A=\mathbb I$), $\mu_*=n-1$ which corresponds to the first positive eigenvalue of the Laplace-Beltrami operator on $\mathbb S^{n-1}$. Then, our first main result is the following.

\begin{teo}[H\"older $C^{0,\alpha}$ estimate]\label{T:0:alpha}
    Let $a+n>0$, $p>(d+a_+)/2$, $q>d+a_+$. Let $A$ be a uniformly elliptic matrix satisfying \eqref{eq:unif:ell} and \eqref{eq:unif:ell:loc} with constants $0<\lambda\leq\lambda_*\leq\Lambda_*\leq\Lambda$.
    Let $\alpha_*=\alpha_*(n,a,\lambda_*/\Lambda_*)$ be defined as in \eqref{alphastar2}. Let
\begin{equation}\label{eq:0:alpha:reg1}
       \alpha\in (0,1)\cap(0,\alpha_*)\cap(0,2-(d+a_+)/p]\cap(0,1-(d+a_+)/q].
\end{equation} 
Let $A$ be continuous with $$\|A\|_{C^{0,\omega}(B_1)}:=\|A\|_{L^\infty(B_1)}+\sup_{\substack{z,\zeta\in B_1 \\ z\neq \zeta}}\frac{|A(z)-A(\zeta)|}{\omega(|z-\zeta|)}\le L\qquad\mathrm{for \ some \ modulus \ of \ continuity \ }\omega,$$

$f\in L^{p,a}(B_1)$, $F\in L^{q,a}(B_1)^d$ and $u$ be a weak solution to \eqref{generalPDE} in $B_1$.

Then, $u\in C^{0,\alpha}_{\rm loc}(B_{1})$. Moreover, there exists a constant $C>0$ depending only on $d$, $n$, $a$, $\lambda$, $\Lambda$, $p$, $q$, $L$ and $\alpha$ such that
\[
\|u\|_{C^{0,\alpha}(B_{1/2})}\le C\big(
\|u\|_{L^{2,a}(B_{1})}
+\|f\|_{L^{p,a}(B_{1})}
+\|F\|_{L^{q,a}(B_{1})}
\big).
\]

\end{teo}

Let us notice here that H\"older regularity for some small implicit exponent follows by \cite{FabKenSer82} in the $A_2$-Muckenhoupt range $a+n\in(0,2n)$ or by the $2$-admissibility of the weight (see \cite{HeiKilMar06}) even in the full range $a+n>0$ again by \cite{FabKenSer82}, once the $2$-admissibility condition is established in Section \ref{sec:2admissible}. The latter is true even when the variable coefficient matrix is only assumed to be bounded measurable. However, the above result proves an estimate which comes with an explicit H\"older exponent. As we will see, this additional information is due to the homogeneity property of the weight term together with the peculiar geometry of its nodal set $\Sigma_0$.

Then, our second result, which is actually the main core of the paper, is the following Schauder estimate, which holds true any time the exponent $\alpha_*$ exceeds $1$, see Remark \ref{rem:alphastar}.

\begin{teo}[Schauder $C^{1,\alpha}$ estimate]\label{T:1:alpha}
    Let $a+n>0$, $p>d+a_+$. Let $A$ be a uniformly elliptic matrix satisfying \eqref{eq:unif:ell} and \eqref{eq:unif:ell:loc} with constants $0<\lambda\leq\lambda_*\leq\Lambda_*\leq\Lambda$.
    Assume that $\alpha_*=\alpha_*(n,a,\lambda_*/\Lambda_*)$ defined as in \eqref{alphastar2} is such that $\alpha_*>1$. Let
\begin{equation}\label{eq:1:alpha:reg:1}
    \alpha\in (0,1)\cap(0,\alpha_*-1)\cap(0,1-(d+a_+)/{p}].
\end{equation}
    Let $A$ be $\alpha$-H\"older continuous with $\|A\|_{C^{0,\alpha}(B_1)}\le L$, $f\in L^{p,a}(B_1)$, $F\in C^{0,\alpha}(B_1)$ and $u$ be a weak solution to \eqref{generalPDE} in $B_1$.
    
    Then, $u\in C^{1,\alpha}_{\rm loc}(B_{1})$ and satisfies 
\begin{equation}\label{eq:c1:boundary:condition}
    (A\D u + F)\cdot e_{y_i}=0,\quad\text{ on } \Sigma_0\cap B_{1},\quad\text{for every } i=1,\dots,n.
\end{equation}
Moreover, there exists a constant $C>0$ depending only on $d$, $n$, $a$, $\lambda$, $\Lambda$, $p$, $L$ and $\alpha$ such that
\begin{equation}\label{eq:stima:c1}
\|u\|_{C^{1,\alpha}(B_{1/2})}\le C\big(
\|u\|_{L^{2,a}(B_{1})}
+\|f\|_{L^{p,a}(B_{1})}
+\|F\|_{C^{0,\alpha}(B_1)}
\big).
\end{equation}
\end{teo}
Let us now comment on the above results and explain the techniques involved. The main idea is to take advantage of a $\varepsilon$-regularization-approximation scheme, which aims to address the loss of uniform ellipticity on the characteristic manifold by approximating with uniformly elliptic problems. In this way, the estimates become available for the $\varepsilon$-approximating problems, and the goal is to prove their uniformity (stability) as the regularization parameter $\varepsilon\to 0$. Returning to the codimension $1$ case, in \cite{SirTerVit21a} the regularization proposed is at the level of the weight term. Specifically, replacing $|y|^a$ with $\rho_\varepsilon^a(y)=(\varepsilon^2+|y|^2)^{\frac{a}{2}}$ one immediately obtains a uniformly elliptic problem, and stability can be established as $\varepsilon\to0$. This approach is not the most effective in this context. To approximate a given solution with a possible boundary condition at $\Sigma_0$, the same boundary condition must be prescribed at the level of the regularized equation. However, the local $H^1(\rho_\varepsilon^a)$-capacity of $\Sigma_0$ - which coincides with the classical local unweighted $H^1$-capacity - is always zero. As a result, no boundary condition can be imposed on the thin set. This fact becomes particularly evident in the mid-range $a+n\in(0,2)$, especially in case of inhomogeneous conditions.

Next, we propose a general strategy suitable for any boundary value problem and valid across different capacitary ranges. The regularization we introduce is performed at the domain level, using $\varepsilon$-perforations around $\Sigma_0$. In the isotropic case $A=\mathbb I$, the approximating problems are described as follows
\begin{equation}\label{approxPDE}
\begin{cases}
-\mathrm{div}(|y|^a\nabla u)=|y|^af+\mathrm{div}(|y|^aF) &\mathrm{in \ } B_1\setminus\Sigma_\varepsilon\\
(\nabla u+F)\cdot\nu=0 &\mathrm{on \ } B_1\cap\partial\Sigma_\varepsilon,
\end{cases}
\end{equation}
where $$\Sigma_\varepsilon=\{|y|\leq\varepsilon\}\qquad \mathrm{and}\qquad \partial\Sigma_\varepsilon=\{|y|=\varepsilon\}.$$
As one can see, the conormal boundary condition is now prescribed on a codimension $1$ boundary that shrinks toward $\Sigma_0$. Given a particular solution of the limiting problem, one can construct an approximating sequence of solutions for the regularized problems as $\varepsilon\to0$. If the regularity estimates for \eqref{approxPDE} are shown to be uniform in $\varepsilon$, the same regularity is transferred to the limit.

Adding variable coefficients makes things harder, as anisotropy directly affects the approximation. As we will see, the stability of the estimates relies on Liouville theorems on the complementary of a hole, after blow-up. For general variable coefficients, adjusting the shape of this hole to the anisotropy proves to be convenient, and this requires the domain to be perforated accordingly at a macroscopic scale; that is, 
$$\Sigma^A_\varepsilon=\{A_3^{-1}(x,y)y\cdot y \leq\varepsilon^2\} \qquad\mathrm{and} \qquad\partial\Sigma^A_\varepsilon=\{ A_3^{-1}(x,y)y\cdot y =\varepsilon^2\}.$$
However, a preliminary regularization of the coefficients through convolution with mollifiers is also needed at this point. In fact, the approximating problems
\begin{equation}\label{approxPDEA}
\begin{cases}
-\mathrm{div}(|y|^aA\nabla u)=|y|^af+\mathrm{div}(|y|^aF) &\mathrm{in \ } B_1\setminus\Sigma^A_\varepsilon\\
(A\nabla u+F)\cdot\nu=0 &\mathrm{on \ } B_1\cap\partial\Sigma^A_\varepsilon,
\end{cases}
\end{equation}
enjoy the desired H\"older $C^{0,\alpha}$ and Schauder $C^{1,\alpha}$ estimates only requiring additional regularity of coefficients, respectively $C^1$ and $C^{1,\alpha}$, since the latter ensures the right regularity of the shrinking boundaries (see Remark \ref{R:C1:boundary}). This is suboptimal when compared with the requirements stated in Theorems \ref{T:0:alpha} and \ref{T:1:alpha}. However, this additional regularity is necessary to ensure stability of the estimates with respect to domain perforation. On the other hand, stability with respect to mollification does not require any perforation and holds true under the optimal requirements on coefficients; that is, respectively the $C^0$ and $C^{0,\alpha}$ regularity imposed by the natural scaling of the equation.

The validity of Theorems \ref{T:0:alpha} and \ref{T:1:alpha} is based on the following main result.

\begin{teo}[Stable regularity estimates in perforated domains]\label{T:0:1:alpha:eps}
    Let $a+n>0$. Let $A$ be a uniformly elliptic matrix satisfying \eqref{eq:unif:ell} and \eqref{eq:unif:ell:loc} with constants $0<\lambda\leq\lambda_*\leq\Lambda_*\leq\Lambda$.
    Let $\alpha_*=\alpha_*(n,a,\lambda_*/\Lambda_*)$ be defined as in \eqref{alphastar2}. Then the following points hold true:
    \begin{itemize}
        \item[i)] Let $p>(d+a_+)/2$, $q>d+a_+$. Let $\alpha\in(0,1)$ satisfying \eqref{eq:0:alpha:reg1}. Let $A$ be $C^1$ with $\|A\|_{C^{1,\omega}(B_1)}\le L$ for some modulus of continuity $\omega$, $f\in L^{p,a}(B_1)$ and $F\in L^{q,a}(B_1)^d$. Then there exists $\eps_0 <1$ and a constant $C>0$ depending only on $d$, $n$, $a$, $\lambda$, $\Lambda$, $p$, $q$, $L$, $\varepsilon_0$ and $\alpha$ such that for every $\eps \in (0, \eps_0)$ and for every solution $u_\eps$ to \eqref{approxPDEA} it holds 
\[
\|u_\e\|_{C^{0,\alpha}(B_{1/2}\setminus\Sigma_\e^A)}\le C\big(
\|u_\e\|_{L^{2,a}(B_{1}\setminus\Sigma_\e^A)}
+\|f\|_{L^{p,a}(B_{1})}
+\|F\|_{L^{q,a}(B_{1})}
\big)\,.
\]  
        \item[ii)] Let $p>d+a_+$. Let assume that $\alpha_*>1$. Let $\alpha\in(0,1)$ satisfying \eqref{eq:1:alpha:reg:1}. Let $A$ be $C^{1,\alpha}$ with $\|A\|_{C^{1,\alpha}(B_1)}\le L$, $f\in L^{p,a}(B_1)$ and $F\in C^{0,\alpha}(B_1)$.
        Then there exists $\eps_0 <1$ and a constant $C>0$ depending only on $d$, $n$, $a$, $\lambda$, $\Lambda$, $p$, $L$, $\varepsilon_0$ and $\alpha$ such that
for every $\eps \in (0, \eps_0)$ and for every solution $u_\eps$ to \eqref{approxPDEA} it holds
\begin{equation*}
\|u_\e\|_{C^{1,\alpha}(B_{1/2}\setminus\Sigma_\e^A)}\le C\big(
\|u_\e\|_{L^{2,a}(B_{1}\setminus\Sigma_\e^A)}
+\|f\|_{L^{p,a}(B_{1})}
+\|F\|_{C^{0,\alpha}(B_1)}
\big).
\end{equation*}
Moreover, for every points $z \in \partial \Sigma_\e^A \cap B_{1/2}$ and for every $i=1,\dots,n$, it holds
\begin{equation}\label{eq:BC:stable:1}
    |(A\D u_\e +F)(z)\cdot e_{y_i}|\le C \e^\alpha \big(
\|u_\e\|_{L^{2,a}(B_{1}\setminus\Sigma_\e^A)}
+\|f\|_{L^{p,a}(B_{1})}
+\|F\|_{C^{0,\alpha}(B_1)}
\big).
\end{equation}

\end{itemize}

\end{teo}

First of all, we would like to point out that the above result remains valid even if the additional $C^1$ or $C^{1,\alpha}$ regularity is only assumed for the block $A_3$, see Remark \ref{R:A3regularity}. However, to simplify the notation, we assume the higher regularity for all entries of the matrix $A$.

Let us also observe that the quantitative condition \eqref{eq:BC:stable:1} is the key for obtaining the effective conormal boundary condition \eqref{eq:c1:boundary:condition} in Theorem \ref{T:1:alpha}, which is a much stronger information compared to the weighted conormal condition \eqref{conormalweighted} enjoyed by weak solutions.

The above result, as far as we know, is new even in the case of the Laplacian ($a=0$ and $A=\mathbb I$), and provides stability of classical H\"older estimates in perforated domains with Neumann boundary condition on the boundary of the hole. Since our estimates can be extended to second order elliptic equations with general lower order terms, Theorem \ref{T:0:1:alpha:eps} $i)$ has a remarkable link with the quantitative spectral stability for the Laplacian in domains with holes with prescribed Neumann boundary conditions, see for instance \cite{FelLivOgn25,Oza85,RauTay75} and many references therein. In particular, we can imply stable $\alpha$-H\"older bounds (for any $\alpha\in(0,1)$) for eigenfunctions in the Neumann-perforated domains, see Remark \ref{R:stability:eigenvalue}.

Notice that the stability of the Schauder estimate can not be valid for the Laplacian, since the effective conormal condition that would follow, i.e. $\nabla_yu=0$ on $\Sigma_0$, is not a general property of harmonic functions. The latter fact is not in contradiction with our result, since one needs $a<0$ in order to make $\alpha_*>1$, see Remark \ref{rem:alphastar}.

In the spirit of the work of Simon \cite{Sim97}, the stable estimates are obtained, in both the regularization procedures (Theorem \ref{T:0:1:alpha:eps} and Proposition \ref{P:a:priori}), by a contradiction argument involving a fine blow-up procedure and the classification of entire profiles. The latter is expressed in terms of rigidity results; that is, Liouville type theorems for both uniformly and degenerate elliptic problems on the blow-up domain, which can be the entire space, the half-space, or the space minus an unbounded cylinder around $\Sigma_0$. The latter result is of independent interest and can be stated as follows.

\begin{Theorem}[Liouville]\label{L:Liouville}
Let $a+n>0$, $\eps \geq 0$, and let $A$ be a constant coefficient $(d\times d)$-dimensional matrix which satisfies \eqref{eq:unif:ell}. Let us denote 
\begin{equation*}
\gamma_1^+=\gamma_1^+(a,n,A) = \frac{2-a-n + \sqrt{(2-a-n)^2 + 4\mu_1 }}{2}\,,
\end{equation*}
where $\mu_1=\mu_1(a,n,A)$ is a constant, which will be precisely defined in Lemma \ref{L:spectral}. 
Let $u$ be an entire solution (see Definitions \ref{def:weak:sol:hom:0} and \ref{def:weak:sol:hom:eps}) to 
\begin{equation}\label{eq:Lioueq}
\begin{cases}
-{\rm div}(|y|^a A \nabla u)= 0 & \mathrm{in \ }\R^d \setminus\Sigma^A_\e\\
|y|^a A\nabla u\cdot\nu=0     & \mathrm{ on \ } \partial\Sigma^A_\e,
\end{cases}
\end{equation}
and assume that there exist two constants $c>0$ and $ \gamma \in (0,\gamma_1^+)$ such that 
\begin{equation}\label{eq:Lgrowth}
|u(z)| \leq c(1 + |z|^\gamma)\,.
\end{equation}
If $\gamma <2$, then $u$ is linear. Moreover, if $\gamma < 1$, then $u$ is constant.
\end{Theorem}
The above result is derived by using spherical coordinates with respect to the degenerate variable $y$ and expressing the solution through a Fourier-type decomposition. This approach requires the anisotropy of the coefficients to align with the geometry of the hole, as radial-in-$y$ (axial) symmetry must emerge after a linear change of coordinates.

Let us elaborate on the regularity results in connection with the Liouville theorem discussed above. The exponent $\alpha_*$ is a threshold for local regularity. It represents the lower bound of the possible homogeneity degrees $\gamma_1^+$ that one may encounter at the microscopic scale of the blow-up. Then, we would like to highlight the unexpected similarity between $\alpha_*$ and the optimal H\"older continuity exponent of solutions to equations with bounded measurable (possibly discontinuous) coefficients, see \cite{PicSpa72}. In the present case, the (restricted-to-$\Sigma_0$) ellipticity ratio $\lambda_*/\Lambda_*$ also plays a role in the expression for $\alpha_*$, reducing its value compared to the isotropic case $A=\mathbb I$. It remains an open question whether this interference is intrinsic or not, as we currently lack examples of solutions with such reduced regularity.

\smallskip

Let us now comment on the possible higher regularity of solutions.

\begin{remark}\label{rem:alphastar}
    The exponent $\alpha_*$ is always strictly positive. In general, given an integer $k\geq1$, $\alpha_*>k$ if and only if $\mu_*>k(k+a+n-2)$. This implies in particular that $\alpha_*>1$ if and only if $a<\mu_*+1-n\leq0$, since $0<\mu_*\leq n-1$.
\end{remark}

As $a$ becomes more and more negative, depending also on an increasing codimension $n$, $\alpha_*$ can be made larger than any integer $k\geq2$. Therefore, $C^{k,\alpha}$ regularity of solutions is expected for any $k\geq2$, under suitable assumptions on the data and within certain ranges of $(a,n)$. 

Unfortunately, we currently lack a general strategy for higher order Schauder estimates, which remains an open problem. On one hand, the regularity estimates obtained in Theorems \ref{T:0:alpha} and \ref{T:1:alpha} cannot be directly iterated on derivatives, since the operator is translation invariant only in the tangential directions to $\Sigma_0$ (the $x$ variables). This implies that the operator commutes only with the derivatives in $x\in\R^{d-n}$ and not in $y\in\R^n$. However, if one knows additionally that the solution is axially symmetric with respect to $\Sigma_0$ (radial-in-$y$), i.e. $u(x,y)=u(x,|y|)$, the iteration procedure works and one can prove smoothness at least in the isotropic homogeneous case, see \cite[Theorem 2.5]{CorFioVit25}. 

On the other hand, one could attempt to construct a regularization-approximation scheme for any order $k$, by smoothing the data, perforating the domain and finally shrinking at the characteristic manifold. However, this strategy also fails, as the stability of the $C^{2,\alpha}$ estimate with respect to $\varepsilon$-perforation does not hold, see Remark \ref{R:C2:fails}.

\smallskip


As a corollary of our main theorems, we can deal with the case of curved regular characteristic thin manifolds. Let $2\le n<d$ and consider a $(d-n)$-dimensional $C^{1,\alpha}$ manifold $\Gamma$ locally embedded in $\R^d$ with $\alpha\in[0,1)$. Then Corollaries \ref{C:0:alpha:curve} and \ref{C:1:alpha:curve} provide the extension of the main Theorems \ref{T:0:alpha} and \ref{T:1:alpha} to weak solutions (see Definition \ref{D:weak:solution:curve}) to \begin{equation}\label{eq:weak:solution:curve}
    -\dive(\delta^a A \nabla u)= \delta^a f+\dive(\delta^aF), \quad \text{in }B_1,
\end{equation}
where the weight $\delta$ behaves as a particular distance function to $\Gamma$, chosen in accordance to the local parametrization of the thin manifold. See Definition \ref{D:admissible:weight} for detailed assumptions on the \emph{defining function} $\delta$.

\subsection*{Some further motivations and applications}
In the final part of this introduction, we aim to elaborate further on some noteworthy motivations and significant applications of our theory.

\subsubsection*{Harmonic maps with prescribed singularities in general relativity}
The study of axially symmetric stationary multi-black-hole configurations and the force between co-axially rotating black holes involves an analysis on the boundary regularity of the reduced singular harmonic maps, see \cite{Wei90,Wei92,LiTia91,LiTia92,LiTia93}. In \cite{Ngu11}, this analysis is carried out by considering those harmonic maps as solutions to some homogeneous divergence systems of partial differential equations with singular coefficients. Then, the model single equation of the system is given by
$$-\mathrm{div}(\omega(z)\nabla u) =0 \qquad\mathrm{in} \ B_1\setminus\Sigma_0,$$
where $\Sigma_0$ is a $(d-n)$-dimensional submanifold with $n\geq 2$, $\omega$ is a weight that satisfies $$\frac{1}{C} \mathrm{dist}(z,\Sigma_0)^{a} \leq \omega(z) \leq C\mathrm{dist}(z,\Sigma_0)^{a}$$ for some negative power $a<0$ and some constant $C>0$. It is worth noting here that the weight's singularity, caused by the negative exponent $a$, does not confine us to a particular capacitary range. Instead, all three ranges may be crossed since $a+n$ might be any real number.

\subsubsection*{The Dirichlet problem and harmonic measure on lower dimensional boundaries}
In a series of works, see \cite{DaiFenMay23,DavFenMay19,DavFenMay21,DavMay22} and many references therein, David-Feneuil-Mayboroda and collaborators start an elliptic theory on weighted non uniformly elliptic operators which naturally catch the capacity of very thin boundaries. The peculiar operator considered in these works has a weight behaving as
$\omega(z)=\mathrm{dist}(z,\Sigma_0)^{1-n}$ where $\Sigma_0$ is a $(d-n)$-dimensional manifold, i.e. the power of the weight belongs to the mid-range since $a+n=1$. Notice that with this choice of the exponent and when $n=1$ one recovers the classic Laplacian. The main target of the works quoted above is the study of the Dirichlet problem on the lower dimensional boundary, the associated harmonic measure, and its reciprocal absolute continuity with the $(d-n)$-dimensional Hausdorff measure \cite{DavFenMay19}, in the spirit of Dahlberg \cite{Dah77}. Although in general the interest in these works relies on rough boundaries, i.e. Lipschitz or less, if the manifold is more regular, $C^1$ or more, one may infer almost Lipschitz type regularity for the solutions of the inhomogeneous Dirichlet problem \cite{Fio24}, or even the sharp Lipschitz regularity for the homogeneous one when $n\geq4$, see \cite[Corollary 5.2 and Remark 5.3]{CorFioVit25}.

\subsubsection*{Critical points for Caffarelli-Kohn-Nirenberg inequalities}
In the particular case $n=d$, semilinear equations involving degenerate weights appear in connection with Caffarelli-Kohn-Nirenberg inequalities \cite{CafKohNir84}, i.e. for certain values of $\alpha,\beta,p$ equations like
\begin{equation*}
-\mathrm{div}(|z|^{-2\alpha}\nabla u)=|z|^{-p\beta}|u|^{p-2}u.
\end{equation*}
Here $-2\alpha+n>2$, and so the weight is superdegenerate and solutions naturally satisfy a homogeneous conormal boundary condition at $\Sigma_0=\{0\}$.
Even if the minimizers in the space $\R^d$ are classified and they are bubbles, the study of the above degenerate equation allows us to imply regularity properties for sign-changing solutions too and for minimizers in domains attaining the singularity at the boundary of a domain \cite{CheLin10}. 

\subsubsection*{Higher codimensional extensions of fractional Laplacians} 
Motivated by \cite{CafSil07}, one might wonder whether it is possible to extend functions $u$ defined in $\R^{d-n}$ having a well-defined $s$-fractional Laplacian ($s\in(0,1)$) to the whole of $\R^d$, introducing $n$ additional variables.
Under the assumption $d-n>2s$, which allows for the definition of suitable energy spaces, in \cite{CorFioVit25} we show that such an extension is given by the convolution $u * P$ with the Poisson-type kernel 
\[P(x,y)=P(|x|,|y|)=\frac{\Gamma(\frac{d-n + 2s}{2})}{\pi^{\frac{d-n}{2}}\Gamma(s)}\frac{|y|^{2s}}{(|x|^2+|y|^2)^{\frac{d-n+2s}{2}}},\qquad (x,y)\in\R^{d-n}\times\R^n.
\]
The extension is a radial-in-$y$ solution (axisymmetric with respect to $\Sigma_0=\R^{d-n}$) to 
$$\begin{cases}
-\mathrm{div}(|y|^a\nabla u)=0 &\mathrm{in \ }\R^d\setminus\Sigma_0\\
\displaystyle-\lim_{|y|\to0}|y|^{a+n-1}\nabla u\cdot\frac{y}{|y|}=d_{a,n}(-\Delta)^su(x,0) &\mathrm{on \ } \Sigma_0,
\end{cases}$$
where $a+n=2-2s\in(0,2)$, and $d_{a,n}$ is an explicit positive constant. As can be seen through the change of variables $|y| = r$, we exactly recover the classical Caffarelli-Silvestre extension.

\subsubsection*{Very thin free boundary problems}
Due to the natural property of these degenerate operators to capture the capacity of very thin sets in the mid-range $a+n\in(0,2)$, it is possible to formulate obstacle-type free boundary problems with an obstacle of arbitrary dimension $d-n \in \{0,\dots , d-2 \}$. In other words, it is possible to minimize the energy functional
\[
\int_{B_1} |y|^a |\nabla u|^2 \, dz,
\]
where $u$ satisfies a trace condition $u = g$ on $\partial B_1$ and the obstacle condition $u \geq \psi$ on the $(d-n)$-dimensional set $\Sigma_0$. We would like to cite \cite{FerJha21}, where the authors propose an alternative definition of a \emph{very thin obstacle problem} (obstacles of dimension $d-2$), related to the Caffarelli-Silvestre extension operator in the regime $a\in (-1,0)$.

\subsection*{Structure of the paper}
In Section \ref{sec:2} we introduce the weighted functional framework for the problem, and we provide many functional inequalities. In Section \ref{sec:3}, we extend the results of the previous section to a setting for perforated domains. In Section \ref{sec:4} we introduce the notion of solutions, and we provide some approximation lemmas for both the regularization procedures. Moreover, we provide stable $L^\infty$ bounds of solutions in perforated domains. Section \ref{sec:5} is dedicated to the proof of the main Liouville Theorem \ref{L:Liouville}. Section \ref{sec:6} is the main core of the paper: it contains the proofs of the main $C^{0,\alpha}$ and $C^{1,\alpha}$ stable estimates in perforated domains, i.e. Theorem \ref{T:0:1:alpha:eps}. In Section \ref{sec:7} we prove the a priori estimates of Proposition \ref{P:a:priori}, which imply stable estimates with respect to standard smoothing of data. This leads to the proof of the main Theorems \ref{T:0:alpha} and \ref{T:1:alpha}. Then, in Section \ref{sec:8} we generalize our results to the case of curved characteristic thin manifolds. In Appendix \ref{sec:A} we collect many important technical results concerning the geometry of the perforated domains.

\subsection*{Notation}
We establish the notation that will be used throughout the paper.

\begin{itemize}[left=0pt]

    \item[$\cdot$] Let $2\leq n\leq d$ be two integers and consider the coordinates $z=(x,y)\in\R^{d-n}\times\R^n$.

    \item [$\cdot$] We denote by $\{e_{z_\ell}\}$, where $\ell=1,\dots,d$, the canonical basis of $\R^d$. To distinguish between the variables $x$ and $y$, we will often denote the basis as $\{e_{x_j}\,, e_{y_i}\}$, where $j=1,\dots,d-n$ and $i=1,\dots,n$.

    \item [$\cdot$] For a vector $G = (G_x, G_y) \in \R^{d-n} \times \R^n$, we write $G = G_1 + G_2$, where $G_1 = (G_x, 0)$ and $G_2 = (0, G_y)$.
  \item [$\cdot$] We express $d$-dimensional symmetric uniformly elliptic matrix $A:\Omega \to \R^{d,d}$ in block form as 
    \begin{equation}\label{Blocks}
    A = \begin{pmatrix} 
    A_1 & A_2 \\
    A_2^\top & A_3
    \end{pmatrix}\,,
    \end{equation}
    where $A_1:\Omega \to \R^{d-n,d-n}$, $A_2:\Omega \to \R^{d-n,n}$ and $A_3:\Omega \to \R^{n,n}$.
    
    \item [$\cdot$] For $m \in \N$ we denote by $B_R^{m}(\zeta)$ the open $m$-dimensional ball of radius $R>0$ centered at $\zeta \in \R^m$. To ease the notation, we simply write $B_R^m= B^m_R(0)$ when $\zeta = 0$, and $B_R(\zeta)= B^d_R(\zeta)$, when the dimension $m=d$. In particular, $B_R = B^d_R(0)$.

    \item [$\cdot$] We write $0<\e\ll1$ to denote that there exists a small $\e_0>0$ such that $\e \in (0,\e_0)$.
\item [$\cdot$]  Throughout the paper, any positive constant whose value is not important is denoted by $c$. It may take different values at different places.

\end{itemize}

\section{Functional setting}\label{sec:2}
\subsection{Weighted Sobolev spaces}\label{S:sobolev:spaces}

Given a smooth domain $\Omega\subseteq\R^d$, we define the Sobolev space $H^{1,a}(\Omega)=H^1(\Omega,|y|^adz)$ as the completion of $C^{\infty}(\overline \Omega)$ with respect to the norm
\begin{equation}\label{Hnorm}
\|u \|_{H^{1,a}(\Omega)}   = \Big(\int_{\Omega}|y|^a \big(|u|^2 + |\nabla u|^2 \big) \, dz\Big)^{\frac12}\,.
\end{equation}
Here, $C^{\infty}(\overline \Omega)$ consists of the restrictions to $\overline\Omega$ of smooth functions in $\R^d$. For our purposes, $\Omega$ will be mostly a ball $B_R$ with radius $R>0$. Therefore, from now on we will focus on the space $H^{1,a}(B_R)$. Moreover, we say that $u \in H^{1, a}_{\rm loc}(\R^d)$ if $u \in H^{1,a}(B_R)$ for every $R>0$.  

Next, we define the homogeneous space $H^{1,a}_0(B_R)$ as the completion of $C^{\infty}_c(B_R)$ with respect to the norm given above. In light of the Poincar\'e inequality (see Proposition \ref{P:Poincare}), $H^{1,a}_0(B_R)$ can be equivalently defined as the completion of $C^{\infty}_c(B_R)$ with respect to the Dirichlet seminorm
\[
[u]_{H^{1,a}(B_R)}  = \|\nabla u\|_{L^{2,a}(B_R)} =\Big( \int_{B_R}|y|^a |\nabla u|^2 \, dz \Big)^{\frac12}\,.
\]

The weight term $|y|^a$ is locally integrable at $\Sigma_0$ whenever $a+n>0$. Conversely, when $a+n\leq0$ the weight is \emph{supersingular}; that is, it is not integrable at $\Sigma_0$. As we will see in Proposition \ref{P:densitysupersingular}, this lack of integrability forces elements in the Sobolev space to vanish on the characteristic manifold in order to maintain finite energy. 

\begin{remark}\label{R:muck}
We recall that a weight $\omega \in L^1_{\rm loc}(\R^d)$, $\omega \geq 0$ is said to belong to the $A_2$-Muckenhoupt class if and only if 
\[
M_\omega:=\sup_{z_0 \in \R^d,\, R>0}\Big(\frac{1}{|B_R(z_0)|} \int_{B_R(z_0)} \omega\,dz \Big) \Big(\frac{1}{|B_R(z_0)|} \int_{B_R(z_0)}\omega^{-1} \,dz \Big) < \infty\,.
\]
It is straightforward to verify that the weight $\omega = |y|^a$ belongs to the $A_2$-Muckenhoupt class if and only if $0<a+n<2n$. 
\end{remark}

\subsection{Hardy-Poincar\'e and Poincar\'e inequalities}

The following Hardy-Poincar\'e-type inequality is well known (see, for instance, \cite[\S 2.1.7, Corollary 2]{Maz11}). The proof of this result is postponed, as it will be included in the proof of the more general Proposition \ref{L:Hardy_eps}.

\begin{Proposition}[Hardy-Poincar\'e inequality]\label{P:Hardy}

Let $R>0$ and $\delta \in \R \setminus \{-n\}$. Then 
\begin{enumerate}
\item[$i)$] if $\delta +n > 0$, then for every $u \in C^\infty(\overline {B_R})$ it holds that
\[
\Big(\frac{\delta+n}{2}\Big)^2 \int_{B_R}|y|^\delta|u|^2 dz \leq \frac{ \delta +n}{2}\int_{\partial B_R}|y|^{\delta+1}|u|^2 ds  + \int_{B_R}|y|^{\delta+2}|\nabla u |^2 dz\,;
\]
\item[$ii)$] if $\delta + n < 0$, then for every $u \in C^\infty_c(\overline{B_R}\setminus \Sigma_0)$ it holds that
\[
\Big(\frac{\delta+n }{2}\Big)^2 \int_{B_R}|y|^{\delta}|u|^2 dz \leq \int_{B_R}|y|^{\delta+2}| \nabla u |^2 dz\,.
\]

\end{enumerate}
\end{Proposition}

As a corollary, we also recover the Poincar\'e inequality. The proof of this result is similarly postponed, as it will be included in the proof of the more general Proposition \ref{P:Poincare_eps}.
 
\begin{Proposition}[Poincar\'e inequality]\label{P:Poincare}
Let $R>0$ and $a\in \R$. Then
\begin{enumerate}
\item[$i)$] if $a+ n>0$ and $u\in C^\infty_c(B_R)$, it holds that
\[
    \Big(\frac{a+n}{2R}\Big)^2\int_{B_R}|y|^a|u|^2\, dz\leq\int_{B_R}|y|^a|\nabla u|^2dz\,,
\]
\item[$ii)$] if $a+ n <2$ and $u\in C^\infty_c(\overline{B_R}\setminus \Sigma_0)$, it holds that
\[
    \Big(\frac{a+n-2}{2R}\Big)^2\int_{B_R}|y|^a|u|^2\, dz \leq\int_{B_R}|y|^a|\nabla u|^2\, dz\,.
\]
\end{enumerate}
\end{Proposition}

\subsection{Weighted capacity}

In this section, we examine how the properties of the weighted Sobolev space $ H^{1,a}(B_R) $ depend on the natural weighted capacity of the characteristic manifold.

Given $R>0$ and a bounded domain $\Omega$ such that $\Sigma_0 \cap B_R \subset \Omega$, we define the local weighted capacity of $\Sigma_0$ in the box $\Omega$ as
\[
\begin{aligned}
\capa_{a}(\Sigma_0\cap B_R; \Omega) = \inf\Big\{ \, \int_{\Omega}|y|^a |\nabla u|^2 dz \ \mid \ u\in C_c^\infty(\Omega),\ u = 1 \text{ on }\Sigma_0\cap B_R \Big\}\,.
\end{aligned}
\]
Thanks to Proposition \ref{P:Poincare}, one could equivalently consider, in the definition above, the minimization of the full norm in \eqref{Hnorm}.
As next Lemma shows, we can identify three distinct capacitary ranges: when the weighted capacity of the characteristic manifold is infinite, i.e. $a+n\leq0$, we call the weight $|y|^a$ \emph{supersingular}, as we said before; when the weighted capacity is zero, i.e. $a+n\geq2$, we call the weight $|y|^a$ \emph{superdegenerate}; finally, when the weighted capacity is locally finite and positive, i.e. $0<a+n<2$, we call the weight $|y|^a$ \emph{mid-range}. Notice that, since $n\geq 2$, the $A_2$-Muckenhoupt range $0<a+n<2n$ intersects both the mid-range and the superdegenerate intervals. We also note that, within the mid-range interval, the capacity depends significantly on both the radius $ R $ and the diameter of the box $\Omega$.

\begin{Lemma}[Capacitary ranges]\label{L:capa}
Let $R >0$, and let $\Omega$ be a open set such that $\Sigma_0 \cap B_R \subset \Omega$.
\begin{enumerate}
\item[i)] if $a + n \leq 0$, then 
\[
\capa_{a}(\Sigma_0\cap B_R; \Omega) = \infty \,;
\]
\item[ii)] if $a+n \geq 2$,  then 
\[
\capa_{a}(\Sigma_0 \cap B_R; \Omega) = 0\,;
\]
\item[iii)] if $0 < a+n < 2$, then there exists a constant $c >0$ depending only on $a$, $d$, $n$ such that
\[
0 < c\ {\rm diam}(\Omega)^{a+n-2}R^{d-n} \leq  \capa_{a}(\Sigma_0 \cap B_R; \Omega) < \infty\,.
\]
\end{enumerate}
\end{Lemma}
\begin{proof}
The proof of $i)$ is straightforward since, when $a + n \leq 0$, it holds that 
\[
\int_{\Omega}|y|^a |\nabla u|^2 dz = \infty \quad \text{ for all } u\in C_c^\infty(\Omega),\ u = 1 \text{ on }\Sigma_0\cap B_R\,.
\]
Indeed, assume by contradiction that there exists $u \in C^\infty_c(\Omega)$ such that $u = 1$ on $\Sigma_0 \cap B_R$ and $\int_{\Omega}|y|^a|\nabla u|dz < \infty$. Since $a+n \leq 0$ implies that $|y|^a \not \in L^1_{\rm loc}(\Omega)$, we infer that $\nabla u =0$ on $\Sigma_0 \cap \Omega$. Therefore, $u = 1$ on $\Sigma_0 \cap \Omega$, which means that $u$ can not be compactly supported on $\Omega$, leading to a contradiction.  

\smallskip

Let now prove $ii)$. Call $\Omega_0 = \{x \in \R^{d-n} \mid (x, 0) \in \Omega\}$, and take $\varphi \in C^\infty_c(\Omega_0)$ such that $\varphi(x) = 1$ for every $x \in B_R^{d-n} \subset \Omega_0$. Moreover, let $\eta \in C^\infty(\R)$ be such that $0\leq \eta\leq 1$, $\eta(t) = 1$ for $t \in (2, \infty)$ and $\eta(t) = 0$ for $t \in (-\infty,1)$. For $h \in \N$, we introduce
\begin{equation}\label{Psih}
\Psi_h (z) = \varphi(x)\eta_h(y) \qquad \text{ where } \qquad \eta_h(y) = \eta\Big( \frac{-\log|y|}{h}\Big)\,.
\end{equation}
One can readily see that, for $h$ sufficiently large, $\Psi_h \in C^{\infty}_c(\Omega)$, $\Psi_h(x, y) = 1$ if $|x| \leq R$ and $|y| \leq e^{-2h}$ (and in particular $\Psi_h = 1$ on $\Sigma_0 \cap B_R$) and $\Psi_h(x,y) = 0$ if $|y| \geq e^{-h}$. Moreover, 
\[
|\nabla_y \Psi_h (z)| = |\varphi(x)|\Big| \eta'  \Big( \frac{-\log|y|}{h}\Big) \frac{y}{h|y|^2}\Big| \leq \frac{\|\varphi\|_{L^\infty(\R^{d-n})}\|\eta'\|_{L^\infty(\R)}}{h|y|}\rchi_{\{e^{-2h}\leq |y|\leq e^{-h}\}}\,.
\]
Using cylindrical coordinates we obtain
\[
\begin{aligned}
 \int_{\Omega} |y|^a |\nabla \Psi_h|^2 dz = \int_{\Omega_0}|\nabla_x \varphi|^2 dx \int_{\R^n}|y|^a|\eta_h|^2 dy &+  \int_{\Omega_0}|\varphi|^2 dx \int_{\R^n}|y|^a|\nabla \eta_h|^2 dy\\
&\leq c\int_{0}^{e^{-h}}r^{a+n-1}dr + \frac{c}{h^2}\int_{e^{-2h}}^{e^{-h}}r^{a+n-3}dr\,.
\end{aligned}
\]
Thus, since $a+n \geq 2$, it holds that $\|\Psi_h\|_{H^{1,a}(\Omega)} \to 0$ as $h \to \infty$, which proves $ii)$. 

 \smallskip

Finally, we prove $iii)$. Let $u \in C^\infty_c(\Omega)$ be such that $u = 1 $ on $\Sigma_0 \cap B_R$, that is, $u(x, 0) = 1$ whenever $|x| \leq R$. Since $|y|^a \in L^1_{\rm loc}(\Omega)$ it is clear that $\|u\|_{H^{1,a}} < \infty$, and thus $\capa(\Sigma_0 \cap B_R) < \infty$. 

Next, let $\rho = {\rm diam} (\Omega)$, so that $\Omega \subset B_\rho$. Fix $x \in B_R^{d-n} $. By the fundamental theorem of calculus and H\"older's inequality we have that, for every $\sigma \in \mathbb{S}^{n-1}$, it holds
\[
\begin{aligned}
1 = |u(x,\rho \sigma) - 1|^2 = |u(x,\rho \sigma) - u(x, 0)|^2 = \Big|\int_0^\rho \nabla u(x, s\sigma)\cdot \sigma d s \Big|^2 \leq \Big(\int_0^\rho |\nabla u|(x, s\sigma) d s \Big)^2& \\
\leq \int_0^\rho s^{n+a-1}|\nabla u|^2(x, s\sigma) d s \int_0^\rho s^{1-a-n}d s \leq c \rho^{2-a-n}\int_0^\infty s^{a+n-1}|\nabla u|^2(x, s\sigma) d s&\,,
\end{aligned}
\]
where $c >0$ depends only on $a$, $n$. As a consequence we have that
\[
n \omega_n = \int_{\mathbb{S}^{n-1}}d \sigma \leq c\rho^{2-a-n}\int_{\mathbb{S}^{n-1}}\int_0^\infty s^{n+a-1}|\nabla u|^2(x, s\sigma) d sd\sigma = c\rho^{2-a-n}\int_{\R^n}|y|^a|\nabla u|^2 dy\,,
\]
and finally, integrating over $ x \in B^{d-n}_{R}$ we find
\[
n \omega_n|B^{d-n}_{R}|\leq c\rho^{2-a-n}\int_{B^{d-n}_R \times \R^n}|y|^a|\nabla u|^2 dz \leq  c\rho^{2-a-n}\int_{\R^d}|y|^a|\nabla u|^2 dz\,.
\]
The conclusion readily follows. 
\end{proof}

As we have previously remarked, in the supersingular setting $a+n\leq0$, the weight term is not locally integrable. This phenomenon leads to the following result.

\begin{Proposition}[Density of smooth functions in the supersingular setting]\label{P:densitysupersingular}
    Let $a+n\leq0$. The space $H^{1,a}(B_R)$ can be equivalently defined as the completion of $C^{\infty}_c(\overline{B_R}\setminus\Sigma_0)$ with respect to the norm given in \eqref{Hnorm}.
\end{Proposition}
\begin{proof}
It suffices to show that for every $u \in C^\infty(\overline{B_R}) \cap H^{1,a}(B_R)$, there exists a sequence $\{u_\eps\}$ in $C^{\infty}_c(\overline{B_R}\setminus\Sigma_0)$ such that $\|u_\eps - u\|_{H^{1,a}(B_R)}\to 0$ as $\eps \to 0$.

Let $u \in C^\infty(\overline{B_R})$ be such that $\|u\|_{H^{1,a}(B_R)}<\infty$. Since $|y|^{a} \not \in L^1_{\rm loc}(B_R)$, it follows that $u = 0$ on $B_R \cap \Sigma_0$. 

Now, let $\eta \in C^\infty(\R)$ be such that $\eta(t) = 0$ for $|t|<1$, $\eta(t) = t$ for $|t|>2$, and $|\eta(t)| \leq |t|$ for any $t \in \R$. Define
\[
u_\eps(z) = \eps \eta(\eps^{-1}u(z))\,.
\]
Clearly, $u_\eps = 0 $ whenever $|u|<\eps$, which implies that $u_\eps \in C^\infty_c(\overline{B_R} \setminus \Sigma_0)$. Moreover, $u_\eps = u$ whenever $|u|> 2\eps$, and $\nabla u_\eps = \eta'(\eps^{-1}u)\nabla u$. Thus
\[
\int_{B_R}|y|^a|u_\eps - u|^2 dz \leq 4\int_{B_R\cap\{|u| < 2\eps\}\setminus\{u=0\}}|y|^a|u|^2 dz  \to 0 \quad \text{as }\eps \to 0\,, 
\]
since $|B_R\cap\{|u| < 2\eps\}\setminus\{u=0\}|\to 0$. We can estimate the gradient part as
\[
\int_{B_R}|y|^a|\nabla(u_\eps - u)|^2 dz \leq c\int_{B_R\cap\{|u| < 2\eps\}\setminus\{u=0,|\nabla u|=0\}}|y|^a|\nabla u|^2 dz \to 0 \quad \text{as }\eps \to 0\,, 
\]
since $|B_R\cap\{|u| < 2\eps\}\setminus\{u=0,|\nabla u|=0\}|\to 0$, given that $B_R\cap\{u=0,|\nabla u|\neq0\}$ is a set of Hausdorff dimension at most $d-1$. 

To prove this, assume $u\in C^1$ and define $\Omega_k:=\overline{B_R}\cap\{u=0,|\nabla u|\geq\frac{1}{k}\}$, which is a compact set. For each $z\in\Omega_k$, there exists $r_z>0$ such that $B_{r_z}(z)\cap\Omega_k$ is a $(d-1)$-dimensional regular hypersurface by the implicit function theorem. Then, selecting a finite subcovering of $\Omega_k\subset\bigcup_{i=1}^{k_0}B_{r_{z_i}}(z_i)\cap\Omega_k$, we have that for any arbitrarily small $\delta>0$ it holds
\[
\mathcal H^{d-1+\delta}(\Omega_k)=0.
\]
Thus, since $\overline{B_R}\cap\{u=0,|\nabla u|\neq0\}=\bigcup_{k\in\mathbb N}\Omega_k$, we have
$$\mathcal H^{d-1+\delta}(\overline{B_R}\cap\{u=0,|\nabla u|\neq0\})=0.$$
The arbitrariness of $\delta>0$ implies that $\overline{B_R}\cap\{u=0,|\nabla u|\neq0\}$ is $(d-1)$-dimensional. Notice that the $(d-1)$-Hausdorff measure could be infinite.
\end{proof}

Similarly to the supersingular setting, i.e. Proposition \ref{P:densitysupersingular}, whenever the weight is superdegenerate one has the following density result. 

\begin{Proposition}[Density of smooth functions in the superdegenerate setting]\label{P:densitysuperdegenerate}
    Let $a+n\geq2$. Then the space $H^{1,a}(B_R)$ can be equivalently defined as the completion of $C^{\infty}_c(\overline{B_R}\setminus\Sigma_0)$ with respect to the norm in \eqref{Hnorm}.
\end{Proposition}
\begin{proof}
It suffices to prove that for every $u \in C^\infty(\overline{B_R}) \cap H^{1,a}(B_R)$ there exists a sequence $u_h \in C^{\infty}_c(\overline{B_R}\setminus\Sigma_0)$ such that $\|u_h - u\|_{H^{1,a}(B_R)}\to 0$ as $h \to \infty$.

Let $\Psi_h \in C^\infty(\overline{B_R})$ be the family of functions introduced in \eqref{Psih}. We recall that $\Psi_h = 1$ in $B_R \cap \{|y| \leq e^{-2h}\}$ and $\Psi_h= 0$ in $B_R\cap\{|y| \geq e^{-h}\}$. Moreover, if $a+n \geq 2$ then   $\|\Psi_h\|_{H^{1,a}(B_R)} \to 0$ as $h \to \infty$.

Let us define 
\[
    u_h = u (1-\Psi_h).
\]
Since obviously $u_h \in C^{\infty}_c(\overline{B_R}\setminus\Sigma_0)$ and 
\[
\|u_h - u\|_{H^{1,a}(B_R)} \leq c \|\Psi_h\|_{H^{1,a}(B_R)}\,,
\]
for a constant $ c>0$ which depends on $\|u\|_{L^\infty(B_R)}$, $\|\nabla u\|_{L^\infty(B_R)}$, the proof is complete. 
\end{proof}

\subsection{H=W property}
Sobolev spaces can be defined in terms of weak derivatives in $L^p $ spaces. In the unweighted case, this definition is equivalent to the one based on the density of smooth functions (which we used to introduce $H^{1,a}(B_R)$). However, establishing this equivalence in the context of weighted Sobolev spaces is not as straightforward.

Following the approach of \cite{Zhi98}, we introduce the set 
\[
\tilde W = \{ u \in W^{1,1}_{\rm loc}(B_R) \mid \|u\|_{H^{1,a}(B_R)} < \infty \}.
\]
We then define
\[
W^{1,a}(B_R) := \text{the completion of $\tilde W$ with respect to $\|\cdot\|_{H^{1,a}(B_R)}$}.
\]
We have the following characterization. Note that the two cases intersect.
\begin{Lemma}\label{L:Wchar}
Let $a \in \R$, $R>0$. It holds
\begin{enumerate}
\item[$i)$] if $a + n <2n$, then 
\[
W^{1,a}(B_R) = \{ u \in W^{1,1}_{\rm loc}(B_R) \mid \|u\|_{H^{1,a}(B_R)}< \infty\}\,,
\]
\item[$ii)$] if $a + n \geq 2$, then 
\[
W^{1,a}(B_R) = \{ u \in W^{1,1}_{\rm loc}(B_R \setminus \Sigma_0) \mid \|u\|_{H^{1,a}(B_R)}< \infty\}\,.
\]
\end{enumerate}
\end{Lemma}
\begin{proof}
Let us prove $i)$. The inclusion $\tilde W  \subset W^{1,a}(B_R)$ is straightforward. 

To prove the reverse inclusion, let us take a Cauchy sequence $u_k \subset \tilde W$. We note that $u_k$ and $\nabla u_k$ are Cauchy sequences in $L^{2,a}(B_R)$ and $L^{2,a}(B_R)^d$, respectively. Due to the completeness of $L^{2,a}(B_R)$, there exist $u \in L^{2,a}(B_R)$ and $V \in L^{2,a}(B_R)^d$ such that
\begin{equation}\label{eq:Wchar_1}
    u_k \to u   \text{ in }L^{2,a}(B_R)\,, \qquad \nabla u_k \to V  \text{ in }L^{2,a}(B_R)^d\,.
\end{equation}
In particular, we note that
\[
\int_{B_R}|y|^a |u|^2 dz + \int_{B_R}|y|^a |V|^2 dz < \infty.  
\]
By H\"older inequality, we have that for every compact set $K \subset B_R$ it holds that
\[
\int_K |u| dz \leq \Big( \int_K |y|^a |u|^2 dz\Big)^{\frac12} \Big( \int_K |y|^{-a} dz\Big)^{\frac12} < \infty\,.
\]
Here we used that $|y|^{-a} \in L^1_{\rm loc}(B_R)$ due to the assumption $a + n < 2n$. 
Performing the same computation for $V$, we obtain that $u, V \in L^{1}_{\rm loc}(B_R)$. 

To conclude, it remains to show that $V$ is the weak gradient of $u$. Fix $\varphi \in C^{\infty}_c(B_R)$. Using \eqref{eq:Wchar_1} we get

\[
\begin{aligned}
    &\Big|\int_{B_R}(u_k - u)\nabla\varphi dz \Big| \leq \Big( \int_{B_R} |y|^a |u_k - u|^2 dz\Big)^{\frac12} \Big( \int_{B_R} |y|^{-a} |\nabla \varphi|^2 dz\Big)^{\frac12}  \to 0\,, \\
    &\Big|\int_{B_R}(\nabla u_k - V)\varphi dz \Big| \leq \Big( \int_{B_R} |y|^a |\nabla u_k - V|^2 dz\Big)^{\frac12} \Big( \int_{B_R} |y|^{-a} |\varphi|^2 dz\Big)^{\frac12}  \to 0\,.
\end{aligned}
\]
Hence, since $u_k \in W^{1,1}_{\rm loc}(B_R)$, we have
\[
\int_{B_R} (u \nabla \varphi + \varphi V )dz = \int_{B_R} (u_k \nabla \varphi + \varphi \nabla u_k )dz + o(1) = o(1)\,.
\]
Therefore, $V = \nabla u $, and $i)$ is proved. 

\smallskip

Let us now prove $ii)$. To show that 
\[
W^{1,a}(B_R) \subset \{ u \in W^{1,1}_{\rm loc}(B_R\setminus \Sigma_0) \mid \|u\|_{H^{1,a}(B_R)}< \infty\} 
\]
we can proceed in a manner similar to the proof of $i)$, with a few minor modifications. Note that if $a+ n \geq 2$, then in general $|y|^{-a} \not \in L^1_{\rm }(B_R)$; however, it is always true that $|y|^{-a} \in L^1_{\rm loc}(B_R\setminus \Sigma_0)$.

Next, let us prove that 
\[
\{ u \in W^{1,1}_{\rm loc}(B_R\setminus \Sigma_0) \mid \|u\|_{H^{1,a}(B_R)}< \infty\} \subset W^{1,a}(B_R)\,.
\]
Let $u \in W^{1,1}_{\rm loc}(B_R\setminus \Sigma_0)$ with $\|u\|_{H^{1,a}(B_R)}< \infty $ be fixed. Moreover, let $\eps >0$. For $k \in \N$, consider the sequence $u_k$ defined as
\[
u_k = 
\begin{cases}
k & \text{ if }u >k\\
u & \text{ if }-k \leq u \leq k \\
-k & \text{ if }u <-k
\end{cases}
\]
It is standard to see that $u_k \in L^\infty(B_R) \cap W^{1,1}_{\rm loc}(B_R\setminus \Sigma_0)$ and  $\|u - u_k\|_{H^{1,a}(B_R)} \to 0$ as $k \to \infty$. In particular, there exists $\hat u := u_{k_0}$ for some sufficiently large $k_0$ such that $\|u - \hat u\|_{H^{1,a}(B_R)} < {\eps}/{2}$.

Let $\Psi_h \in C^\infty(\overline{B_R})$ be the family of functions introduced in \eqref{Psih}. We recall that $0 \leq \Psi_h \leq 1$, $\Psi_h = 1$ in $B_R \cap \{|y| \leq e^{-2h}\}$ and $\Psi_h= 0$ in $B_R\cap\{|y| \geq e^{-h}\}$. Moreover, since $a+n \geq 2$ then $\|\Psi_h\|_{H^{1,a}(B_R)} \to 0$ as $h \to \infty$.

Define
\[
u_h = \hat u(1-\Psi_h).
\]
We point out that ${\rm spt }(u_h) \subset\{|y|\geq  e^{-h}\}$. As a consequence, $u_h, |\nabla u_h| \in L^{1}_{\rm loc}(B_R)$, even when $|y|^{-a} \not \in L^{1}_{\rm loc}(B_R)$. Moreover, for every $\varphi \in C^{\infty}_c(B_R)$ it holds
\[
\int_{B_R}(u_h \nabla \varphi + \varphi \nabla u_h) dz = \int_{B_R}\hat u \nabla[(1-\Psi_h)\varphi] + [(1 - \Psi_h) \varphi]\nabla\hat u\,dz=0\,,
\]
where we used that $(1-\Psi_h)\varphi \in C^{\infty}_c(B_R \setminus \Sigma_0)$ and $\hat u \in W^{1,1}_{\rm loc}(B_R \setminus \Sigma_0)$. 
Hence, $u_h \in W^{1,1}_{\rm loc}(B_R)$.

By Lebesgue dominated convergence theorem (recall that $|\Psi_h| \leq 1 $ and $\Psi_h \to 0$ a.e. on $B_R$ as $h \to \infty$) and since $\|\Psi_h\|_{H^{1,a}(B_R)} \to 0$ we obtain, as $h \to \infty$,
\[
\begin{aligned}
\|\hat u - u_h\|_{H^{1,a}(B_R)}^ 2 &\leq 2\int_{B_R}|y|^a (|\nabla \hat u|^2 + |\hat u|^2)|\Psi_h|^2 dz + 2\int_{B_R}|y|^a|\hat u|^2 |\nabla \Psi_h|^2 dz \\
&\leq o(1) + 2\|\hat u\|_{L^\infty(B_R)}\|\psi_h\|_{H^{1,a}(B_R)} = o(1)\,.
\end{aligned}
\]
In particular, fix $h_0 \in \N$ such that $\|\hat u - u_{h_0}\|_{H^{1,a}(B_R)} < \frac{\eps}{2}$.

Since $u_h \in W^{1,1}_{\rm loc}(B_R)$ and
\[
\|u - u_{h_0}\|_{H^{1,a}(B_R)} \leq \|u - \hat u\|_{H^{1,a}(B_R)} +\|\hat u - u_{h_0}\|_{H^{1,a}(B_R)} < \eps\,,
\]
thanks to the arbitrariness of $\eps$ we conclude that $u$ belongs $W^{1,a} (B_R)$, thus completing the proof. 
\end{proof}

It is now straightforward to show that the $H=W$ property holds, at least when $a+ n >0$. 

\begin{Lemma}[H=W]\label{L:H=W}
If $a + n >0$, then 
\[
H^{1,a}(B_R) = W^{1,a}(B_R).
\]
\end{Lemma}
\begin{proof}
If $0 < a+ n< 2n$ the weight $|y|^a$ belongs to the Muckenhoupt class $A_2$, see Remark \ref{R:muck}. Therefore, the result is classical; see for instance \cite[Theorem 2.5]{Kil94} or \cite[$\S$4]{Zhi98}.

If $a + n \geq 2$, we first note that every $u \in W^{1,a}(B_R)$ can be approximated by functions $u_h \in W^{1,1}_{\rm loc}(B_R)$ which are supported away from $\Sigma_0$; see the proof of Lemma \ref{L:Wchar}, $ii)$. In turn, the functions $u_h$ can be approximated by functions in $C^{\infty}_c (\overline{B_R} \setminus \Sigma_0)$ using a standard mollification technique.
\end{proof}

\subsection{Trace inequalities and compact embeddings}

Next we show a trace-type inequality. To this aim, we introduce the projection
\[
\Pi: \R^d \to \R^n \qquad \Pi(z) = y\,,
\]
and we notice that, using polar coordinates, for any $z = r \sigma$, $r=|z| >0$, $\sigma=\frac{z}{|z|} \in \mathbb{S}^{d-1}$, it holds
\[
y = r \Pi \sigma\,.
\]
\begin{Lemma}[Trace inequality]\label{L:tracePo}
Let $R>0$ and $a+n> 0$. Then for any $u \in C^\infty(\overline{B_R})$ and $0 < r\leq R$ it holds
\[
c\int_{\partial B_r}|y|^a |u|^2 ds \leq r\int_{B_R}|y|^a |\nabla u|^2 dz+ r^{-1}\int_{B_R}|y|^a |u|^2 dz\,,
\]
for a constant $c>0$ depending only on $d$ and $a$.
\end{Lemma}
\begin{proof}
Let $0 <\rho < r \leq R$. By fundamental theorem of calculus and H\"older's inequality we have that for any $\sigma \in \mathbb{S}^{d-1}$ it holds 
\[
|u(r \sigma) - u(\rho \sigma)|^2 \leq \Big(\int_{\rho}^r |\nabla u|(\tau\sigma) d\tau \Big)^2 \leq \int_{\rho}^r\tau^{1-a-d}d\tau\int_{\rho}^r\tau^{a+d-1} |\nabla u|^2(\tau\sigma) d\tau\,.
\]
Therefore
\[
\begin{aligned}
    \int_{\mathbb{S}^{d-1}}|\Pi \sigma|^a |u(r \sigma) - u(\rho \sigma)|^2 d\sigma \leq \int_{\rho}^r \tau^{1-a-d}d\tau\int_{B_r \setminus B_\rho}|y|^a|\nabla u|^2 dz \leq \int_{\rho}^r \tau^{1-a-d}d\tau\int_{B_R}|y|^a|\nabla u|^2 dz\,,
\end{aligned}
\]
and thus
\[
\begin{aligned}
r^{1-a-d}&\int_{\partial B_r}|y|^a |u|^2ds = \int_{\mathbb{S}^{d-1}}|\Pi \sigma|^a |u(r \sigma)|^2 d\sigma \\
&\leq 2\int_{\mathbb{S}^{d-1}}|\Pi \sigma|^a |u(r \sigma) - u(\rho \sigma)|^2 d\sigma+ 2 \int_{\mathbb{S}^{d-1}}|\Pi \sigma|^a |u(\rho \sigma)|^2d\sigma\\
&\leq 2(r-\rho)\max\{r^{1-a-d}, \rho^{1-a-d}\}\int_{B_R}|y|^a|\nabla u|^2dz + 2 \int_{\mathbb{S}^{d-1}}|\Pi \sigma|^a |u(\rho \sigma)|^2d\sigma\,.
\end{aligned}
\]
Next, we multiply both sides by $\rho^{a+ d -1}$ and integrate over $\rho \in [0,r]$, using that $a+ d \geq a+n >0$. Since
\[
\int_0^r\rho^{a+ d -1}\int_{\mathbb{S}^{d-1}}|\Pi \sigma|^a |u(\rho \sigma)|^2d\sigma d\rho = \int_{B_r}|y|^a |u|^2 dz\,,
\]
the proof is complete. 
\end{proof}

\begin{Lemma}[Compact embedding $H^{1,a}\hookrightarrow L^{2,a}$]\label{L:compactness}
Let $R>0$ and $a+ n >0$. Let $\{u_k\}, u$ in $H^{1,a}(B_R)$ be such that $u_k \rightharpoonup u$ in $H^{1,a}(B_R)$.  Then, up to subsequences, $ u_k \to u $ in $ L^{2,a}(B_R) $.
\end{Lemma}
\begin{proof}
Whitout loss of generality, we can assume that $u_k \rightharpoonup 0$ in $H^{1,a}(B_R)$. Morever, since $u_k$ is weakly convergent in the Hilbert space $H^{1,a}(B_R)$, it is bounded.

Let $\delta > 0$ be fixed, and consider $\varphi_\delta \in C^\infty(\overline{B_R})$ such that $0 \leq \varphi_\delta \leq 1$ with $\varphi = 1$ when $|y|\leq \delta$ and $\varphi = 0$ when $|y| \geq 2 \delta$. In particular, on the support of $\varphi_\delta$ and for every $\sigma >0$ it holds $1 \leq 2^\sigma\delta^\sigma |y|^{-\sigma}$.

We compute 
\[
\int_{B_R}|y|^a|u_k|^2 dz \leq 2\int_{B_R}|y|^a(1-\varphi_\delta)^2|u_k|^2  + 2\int_{B_R}|y|^a \varphi_\delta^2|u_k|^2.
\]
Fix $0<\sigma<2$ such that $ a - \sigma +n>0$. Using Proposition \ref{P:Hardy}, $i)$ with $\delta = a-\sigma$, we find
\[
\int_{B_R}|y|^a\varphi_\delta^2|u_k|^2 \leq \delta^\sigma \int_{B_R}|y|^{a-\sigma}|u_k|^2 \leq c \delta^\sigma R^{2-\sigma} \|u_k\|_{H^{1,a}(B_R)}\,,
\]
where in the last inequality we also used Lemma \ref{L:tracePo}.

In particular, since $\{u_k\}$ is a bounded sequence in $H^{1,a}(B_R)$, we infer that 
\[
\int_{B_R}|y|^a|u_k|^2 dz \leq 2\int_{B_R}|y|^a (1-\varphi_\delta)^2|u_k|^2  + c \delta ^\sigma\,,
\]
for a constant $c >0$ not depending on $k$. 

Next, we notice that $(1-\varphi_\delta)u_k$ is a sequence supported in $\overline{B_R \setminus \Sigma_\delta}$, where the weight $|y|^a$ is bounded and bounded away from zero. Thus, via a standard argument, we can see that $(1-\varphi_\delta)u_k \rightharpoonup 0$ in the unweighted Sobolev space $ H^{1}(B_R) $, and therefore
\[
\int_{B_R} |y|^a(1-\varphi_\delta)^2 |u_k|^2 \leq c \int_{B_R} (1-\varphi_\delta)^2 |u_k|^2 \to 0 \quad \text{as } k \to \infty,
\]
by the classical Rellich-Kondrakov theorem.
Summing up, we have proved that
\[
\lim_{k \to \infty} \int_{B_R} |y|^a |u_k|^2 \, dz \leq c \delta^\sigma,
\]
and the conclusion easily follows due to the arbitrariness of $\delta$.
\end{proof}

As a consequence of the previous two results we have the following trace theorem. 

\begin{Lemma}[Trace operator]\label{L:H^1_0:trace}
Let $R>0$ and $a+n>0$. Then there exists a unique bounded linear operator
\[
T_R : H^{1,a}(B_R) \to L^{2,a}(\partial B_R)\,,
\]
such that for all $u \in C^\infty(\overline{B_R})$
\[
T_R u = u _{|\partial B_R}.
\]
Moreover, the following characterization holds
\[
H^{1,a}_0(B_R) = \{ u \in H^{1,a}(B_R) \mid T_R u = 0 \}\,.
\]
\end{Lemma}
\begin{proof}
Since $C^{\infty}(\overline{B_R})$ is dense in $H^{1,a}(B_R)$, the first part of the lemma is a direct consequence of Lemma \ref{L:tracePo} and a density argument.  

As for the second part, it is obvious by the above definition that every $u \in H^{1,a}_0(B_R)$ satisfies $T_R(u) = 0$. Let then $u \in H^{1,a}(B_R)$ be such that $T_R u =  0$. Since $T_R u = 0$, $u$ can be trivially extended to zero outside of $B_R$. 

Now let us consider, for $\lambda \in (1,2)$ the functions 
\[
u_\lambda (z) = u(\lambda z)\,,
\]
which, by construction, are such that $u_\lambda (z) = 0$ if $\frac{R}{\lambda} \leq |z|\leq R$. 
As a consequence, one can easily see that $\|u_\lambda\|_{H^{1,a}(B_R)} \leq c \|u\|_{H^{1,a}(B_R)}$ for a constant $c>0$ not depending on $\lambda$. Thus $u_\lambda \in W^{1,a}(B_R)$ by Lemma \ref{L:Wchar}, and in fact $u_\lambda \in H^{1,a}(B_R)$ by Lemma \ref{L:H=W}. Moreover, since $u_\lambda (z) = 0$ if $\frac{R}{\lambda} \leq |z|\leq R$, 
each $u_\lambda$ can be approximated by functions in $C^\infty_c(B_R)$, that is, $u_\lambda \in H^{1,a}_0(B_R)$.

Since $H^{1,a}_0(B_R)$ is complete, to conclude it suffices to show that $\|u_\lambda - u\|_{H^{1,a}(B_R)}\to 0 $ as $\lambda \to 1^+$.
Since $u_\lambda$ is bounded in $H^{1,a}_0(B_R)$, it converges weakly to $v \in H^{1,a}_0(B_R)$ and, thanks to Lemma \ref{L:compactness}, it also holds $u_\lambda \to v$ a.e. in $B_R$. In fact, since $u_\lambda \to u$ a.e. in $B_R$ by definition, $v = u$. Finally, we observe that $\|u_\lambda\|_{H^{1,a}(B_R)} \to \|u\|_{H^{1,a}(B_R)}$ as $\lambda \to 1^+$, and this is enough to obtain strong convergence. The proof is complete. 
\end{proof}

\subsection{Poincar\'e-Wirtinger inequality and 2-admissible weights}\label{sec:2admissible}
In this section we establish the $2$-admissibility of the weight term $|y|^a$ whenever $a+n>0$. The definition of $2$-admissible weights can be found in \cite[\S 1.1]{HeiKilMar06}, and can be resumed into four properties, in order: the doubling condition on the measure $d\mu=|y|^adz$, a condition on the well-definedness of the weak gradient, a Sobolev inequality, which is proved later in Lemma \ref{L:Sstab} in a more general form, and the Poincar\'e-Wirtinger inequality, which is stated and proved below. Checking the first two conditions is easy and we omit the proofs.

\begin{Proposition}[Poincar\'e-Wirtinger inequality]
    Let $a+n>0$ and $R>0$. Then there exists a positive constant $c>0$ depending only on $d$, $a$ such that for any $u\in H^{1,a}(B_R)$
    \begin{equation}\label{poinwirtR}
    c\int_{B_R}|y|^a|u-\langle u\rangle_R^a|^2 \,dz\leq R^2\int_{B_R}|y|^a|\nabla u|^2 \,dz
\end{equation}
where
\begin{equation*}
    \langle u\rangle_R^a:=\frac{1}{\mu_a(B_R)}\int_{B_R}|y|^a u \,dz,\qquad \mathrm{and}\qquad  \mu_a(B_R)=\int_{B_R}|y|^a \,dz.
\end{equation*}
\end{Proposition}
The proof is classical and can be found for instance in \cite[\S 5.8.1]{Eva10}. However, we report it in order to show why the inequality holds true whenever the weight is locally integrable and the codimension $n\geq2$ while it does not hold in the superdegenerate setting when $n=1$.
\begin{proof}
Let us prove the inequality in the unitary ball; that is,
   \begin{equation}\label{poinwirt1}
    c\int_{B_1}|y|^a|u-\langle u\rangle_1^a|^2 \,dz\leq \int_{B_1}|y|^a|\nabla u|^2 \,dz.
\end{equation}
Then, one can recover \eqref{poinwirtR} by scaling. Assume by contradiction that \eqref{poinwirt1} does not hold. Then, along a sequence $\{u_k\}\subset H^{1,a}(B_1)$
\begin{equation*}
    \|u_k-\langle u_k\rangle_1^a\|_{L^{2,a}(B_1)}>k\|\nabla u_k\|_{L^{2,a}(B_1)}.
\end{equation*}
Let
\begin{equation*}
    v_k=\frac{u_k-\langle u_k\rangle_1^a}{\|u_k-\langle u_k\rangle_1^a\|_{L^{2,a}(B_1)}}.
\end{equation*}
    Then
\[
        \|v_k\|_{L^{2,a}(B_1)}=1, \quad  \langle v_k\rangle_1^a=0, \quad\mathrm{and} \quad \|\nabla v_k\|_{L^{2,a}(B_1)}<1/k.
\]
    By Lemma \ref{L:compactness}, there exists $v\in H^{1,a}(B_1)$ such that, $v_k\rightharpoonup v$ in $H^{1,a}(B_1)$ and $v_k\to v$ in $L^{2,a}(B_1)$ up to subsequences, with $\|v\|_{L^{2,a}(B_1)}=1$, $\langle v\rangle_1^a=0$ and $\nabla v=0$ almost everywhere in $B_1$. Since the weak gradient is defined as a $L^1_{\rm loc}(B_1)$ function whenever $a+n<2n$ and as a $L^1_{\rm loc}(B_1\setminus\Sigma_0)$ function whenever $a+n\geq2$, in any case, due to the connectedness respectively of $B_1$ or $B_1\setminus\Sigma_0$, one can conclude that $v$ must be constant almost everywhere in $B_1$. Then, the two conditions $\|v\|_{L^{2,a}(B_1)}=1$ and $\langle v\rangle_1^a=0$ are in contradiction.
\end{proof}
Notice that, when $n=1$ and in the superdegenerate setting, i.e. $a\geq1$, $B_1\setminus\Sigma_0$ is not connected. Then, the proof above is not valid. In fact, the Poincar\'e-Wirtinger inequality is not valid too. Consider the jump function $u=1$ on $B_1^+$, $u=0$ on $B_1^-$, see \cite[Example 1.4]{SirTerVit21a}. However, we would like to stress the fact that in the codimension $1$ case, the weight $|y|^a$ is still $2$-admissible from one or the other side of the hyperplane $\Sigma_0$.

\section{Functional setting with stability in perforated domains}\label{sec:3}

Let us consider a symmetric matrix $ A \in C^1(B_R; \R^{d,d}) $ which satisfies the uniform ellipticity condition \eqref{eq:unif:ell} and express $A$ in blocks form as in \eqref{Blocks}; that is,
\begin{equation*}
    A = \begin{pmatrix} 
A_1 & A_2 \\
A_2^\top & A_3
\end{pmatrix}\,,
\end{equation*}
where $ A_1 \in C^1(B_R; \R^{d-n,d-n}) $, $ A_2 \in C^1(B_R; \R^{d-n,n}) $ and $ A_3 \in C^1(B_R; \R^{n,n}) $\,.

\begin{remark}\label{R:A3regularity}
    Throughout the paper, the regularity assumptions stated above can be relaxed as follows: $ A_1 \in C^0(B_R; \R^{d-n,d-n}) $, $ A_2 \in C^0(B_R; \R^{d-n,n}) $ and $ A_3 \in C^1(B_R; \R^{n,n}) $. In other words, the only block that actually requires one additional degree of regularity is $A_3$. However, since we will smooth the original coefficients by convolution with standard mollifiers and to simplify the notation, we will assume the stronger condition $A\in C^1(B_R; \R^{d,d})$. In the same way, when we assume $ A \in C^{1,\alpha}(B_R; \R^{d,d}) $ for some $\alpha\in(0,1)$ we could instead assume $ A_1 \in C^{0,\alpha}(B_R; \R^{d-n,d-n}) $, $ A_2 \in C^{0,\alpha}(B_R; \R^{d-n,n}) $ and $ A_3 \in C^{1,\alpha}(B_R; \R^{n,n}) $.
\end{remark}

Let us consider a small $\e_0>0$ such that Lemma \ref{L:normal} holds true in $B_R$.  The latter depends on $R$, $\lambda$, $\Lambda$ and $L$ where $\|A\|_{C^1(B_R)}\leq L$. For every $\e \in (0,\e_0)$, let us define 
    the $(\eps, A)$-neighborhood of $\Sigma_0$ as
    \[
     \Sigma^A_\e=\{(x,y)\in \R^{d-n}\times\R^n \mid {A_3^{-1}(x,y)y \cdot y}\le\e^2\}\,,
    \]
     and its boundary 
     \[
     \partial \Sigma^A_\e=\{(x,y)\in \R^{d-n}\times\R^n \mid {A_3^{-1}(x,y)y \cdot y}=\e^2\}\,.
    \]
    When $A=\mathbb{I}$ we simply write
\[
\Sigma_\e:=\Sigma_\e^{\mathbb{I}}= \{|y|\le \e\}\,\quad \partial\Sigma_\e:=\partial \Sigma_\e^{\mathbb{I}}= \{|y|= \e\}. 
\]
Moreover, let us point out that for every matrix $A$ one has that $\{y=0\}=\Sigma^A_0=\partial\Sigma_0^A$.

\begin{remark}\label{R:C1:boundary}
    
    The choice of $\e_0$ ensures that the normal vector to $\partial \Sigma^A_\e$ is continuous and well defined in $B_R$. Hence, the boundary $\partial \Sigma^A_\e$ is locally of class $C^1$ whenever $A\in C^1(B_R)$ and of class $C^{1,\alpha}$ whenever $A\in C^{1,\alpha}(B_R)$ for any given $\alpha\in(0,1)$. 
\end{remark}
In the rest of the paper we write $0<\e\ll1$ to denote that there exists a possibly small $\eps_0$ such that we consider $\e \in (0,\e_0)$. Moreover, in the context of perforated domains, we always assume that the matrix $A$ is at least $C^1$ to guarantee that the boundary $\partial \Sigma_\varepsilon^A$ is $C^1$.

\subsection{Functional spaces on perforated domains}
\subsubsection{Smooth functions}
For $R >0$, $0<\eps\ll 1$, we define the following spaces of smooth functions.
\begin{align*}
    &C^\infty(\overline{B_R\setminus\Sigma_\e^A})=\big\{
    u_{|_{B_R\setminus\Sigma_\e^A}} \ \mid \ u\in C^\infty(\R^d)
    \big\},\\
    &C_c^\infty(\overline{B_R}\setminus\Sigma_\e^A)=\big\{
    u_{|_{B_R\setminus\Sigma_\e^A}} \ \mid \ u\in C^\infty(\R^d) \text{ and }
    \supp(u) \subset\subset \R^d\setminus\Sigma_\e^A
    \big\},\\
    &C_c^\infty({B_R}\setminus\mathring{\Sigma}_\e^A)=
    \big\{
    u_{|_{B_R\setminus\Sigma_\e^A}} \ \mid \ u\in C^\infty(\R^d) \text{ and }
    \supp(u) \subset\subset B_R
    \big\},\\
    &C_c^\infty({B_R}\setminus\Sigma_\e^A)=\big\{
    u_{|_{B_R\setminus\Sigma_\e^A}} \ \mid \ u\in C^\infty(\R^d) \text{ and }
    \supp(u) \subset\subset B_R\setminus\Sigma_\e^A
    \big\}.
\end{align*}

\subsubsection{Weighted $L^p$ spaces}

For $a \in \R$, and $p\in[1,\infty)$, we define the spaces
$$L^{p,a}(B_R\setminus\Sigma_\e^A):= L^p(B_R\setminus\Sigma_\e^A, |y|^a dz), \qquad L^{p,a}(B_R\setminus\Sigma_\e^A)^d:= L^p(B_R\setminus\Sigma_\e^A, |y|^a dz)^d.$$ 
In both cases we denote the norm as $\|\cdot\|_{L^{p,a}(B_R\setminus\Sigma_\e^A)}$, for sake of simplicity.

\subsubsection{Weighted Sobolev spaces}

Let us define the norm

\begin{equation}\label{epsnorm}
\|u\|_{H^{1,a}(B_R\setminus\Sigma_\e^A)}=\Big(\int_{B_R\setminus \Sigma_\e^A} |y|^a (u^2 +|\nabla u|^2) dz\Big)^{\frac12}.
\end{equation}
We define the following Sobolev spaces:

\begin{itemize}
\item[] ${H}^{1,a}(B_R\setminus\Sigma_\e^A) = $ the completion of $C^\infty(\overline{B_R\setminus\Sigma_\e^A})$ with respect to the norm in \eqref{epsnorm};
\item[] $\tilde{H}^{1,a}(B_R\setminus\Sigma_\e^A) = $ the completion of $C_c^\infty(\overline{B_R}\setminus\Sigma_\e^A)$ with respect to the norm in \eqref{epsnorm};
\item[] $\hat{H}^{1,a}(B_R\setminus\Sigma_\e^A) = $ the completion of $C_c^\infty(B_R\setminus\mathring{\Sigma}_\e^A)$ with respect to the norm in \eqref{epsnorm};
\item[] ${H}^{1,a}_{0}(B_R\setminus\Sigma_\e^A) = $ the completion of $C_c^\infty({B_R}\setminus\Sigma_\e^A)$ with respect to the norm in \eqref{epsnorm}.
\end{itemize}
Note that the following inclusion of spaces hold true
\[ {H}^{1,a}_{0}(B_R\setminus\Sigma_\e^A)\subset \hat{H}^{1,a}(B_R\setminus\Sigma_\e^A) \subset {H}^{1,a}(B_R\setminus\Sigma_\e^A), \]
\[
{H}^{1,a}_{0}(B_R\setminus\Sigma_\e^A)\subset \tilde{H}^{1,a}(B_R\setminus\Sigma_\e^A) \subset {H}^{1,a}(B_R\setminus\Sigma_\e^A).
\]

\begin{remark}
We point out that for every $a \in \R$ it holds
\[
\begin{gathered}
H^{1,a}(B_R \setminus \Sigma_0) = H^{1,a}(B_R), \qquad \tilde H^{1,a}(B_R \setminus \Sigma_0) = \tilde H^{1,a}(B_R)\\
\hat H^{1,a}(B_R \setminus \Sigma_0) = H^{1,a}_0 (B_R \setminus \Sigma_0) = H^{1,a}_0(B_R)\,.  
\end{gathered}
\]
If $a+n\in(-\infty,0]\cup[2,\infty)$, by Proposition \ref{P:densitysupersingular} and Proposition \ref{P:densitysuperdegenerate} we also have 
\[
    H^{1,a}(B_R)= \tilde{H}^{1,a}(B_R\setminus\Sigma_0), \qquad {H}^{1,a}_0(B_R)={H}^{1,a}_0(B_R\setminus\Sigma_0)\,.
    \]
\end{remark}

\subsection{Stable Hardy-Poincar\'e and Poincar\'e inequalities}

\begin{Proposition}[$\varepsilon$-stable Hardy-Poincar\'e inequality]\label{L:Hardy_eps}

Let $R>0$, $0\le \eps\ll1 $ and $\delta \in \R \setminus \{-n\}$. Then 
\begin{enumerate}
\item[$i)$] if $\delta +n > 0$, for every $u \in C^\infty(\R^d)$ it holds
\[
\Big(\frac{\delta+n}{2}\Big)^2 \int_{B_R\setminus \Sigma_\eps^A}|y|^\delta|u|^2 dz \leq \frac{ \delta +n}{2}\int_{\partial B_R \setminus \Sigma_\eps^A}|y|^{\delta+1}|u|^2 ds  + \int_{B_R\setminus \Sigma_\eps^A}|y|^{\delta+2}|\nabla u |^2 dz
\]
\item[$ii)$] if $\delta + n < 0$, for every $u \in C^\infty_c(\R^d \setminus \Sigma_\eps^A)$ it holds
\[
\Big(\frac{\delta+n }{2}\Big)^2 \int_{B_R\setminus \Sigma_\eps^A}|y|^{\delta}|u|^2 dz \leq \int_{B_R\setminus \Sigma_\eps^A}|y|^{\delta+2}| \nabla u |^2 dz
\]

\end{enumerate}
\end{Proposition}
\begin{proof}
First we prove $i)$. Let $\delta + n >0 $, $\delta \neq -2$. Using that $\Delta |y|^{\delta+2} = (\delta+2)(\delta+n)|y|^{\delta}$ and integrating by parts, we find that for every $u \in C^\infty(\R^d)$ it holds 
\begin{equation}\label{eq:Hard_pass}
\begin{aligned}
(\delta+n)\int_{B_R\setminus \Sigma_\eps^A}|y|^{\delta}|u|^2 dz &= -  2 \int_{B_R\setminus \Sigma_\eps^A} |y|^{\delta} u \nabla u\cdot y\, dz\\
&+ \int_{\partial B_R\setminus \Sigma_\eps^A}|y|^{\delta+1} |u|^2 d\sigma
-\int_{\partial \Sigma_\eps^A \cap B_R}|y|^{\delta} y \cdot \nu(z) |u|^2 d\sigma \\
& \leq 2 \int_{B_R\setminus \Sigma_\eps^A} |y|^{\delta+1} |u| |\nabla u|\, dz + \int_{\partial B_R\setminus \Sigma_\eps^A}|y|^{\delta+1} |u|^2 d\sigma\,,
\end{aligned}
\end{equation}
where we used that $\nu(z) \cdot y >0 $ on $\partial \Sigma_\eps^A$ by using Lemma \ref{L:normal}, $ii)$ in the ball $B_R$.
In fact, letting $\delta \to -2$ we see that \eqref{eq:Hard_pass} holds also for $\delta =-2$. 

Next we apply H\"older and Young inequalities to get
\[
2\int_{B_R\setminus \Sigma_\eps^A} |y|^{\delta+1} |u| |\nabla u|\, dz \leq  \frac{\delta+n}{2}\int_{B_R\setminus \Sigma_\eps^A}|y|^{\delta} |u|^2 dz
+ \frac{2}{\delta+n} \int_{B_R\setminus \Sigma_\eps^A}|y|^{\delta+2}|\nabla u|^2dz\,,
\]
and $i)$ follows. 

Let now $\delta + n < 0 $ and $u \in C^\infty_c(\R^d \setminus \Sigma_\eps^A)$. Arguing exactly as in the proof of \eqref{eq:Hard_pass} we find that
\[
\begin{aligned}
|\delta+n |\int_{B_R\setminus \Sigma_\eps^A}|y|^{\delta}|u|^2 dz = -\int_{\partial B_R\setminus \Sigma_\eps^A}|y|^{\delta+1} |u|^2 d\sigma + 2 \int_{B_R\setminus \Sigma_\eps^A}|y|^{\delta}u \nabla u \cdot y\, dz\leq  2 \int_{B_R\setminus \Sigma_\eps^A}|y|^{\delta+1}|u||\nabla u |dz\,,
\end{aligned}
\]
and thus we obtain $ii)$ via H\"older's inequality.
\end{proof}
 
\begin{Proposition}[$\varepsilon$-stable Poincar\'e inequality]\label{P:Poincare_eps}
Let $R>0$, $0\le \eps \ll 1$ and $a\in \R$. Then
\begin{enumerate}
\item[$i)$] if $a+ n>0$ and $u\in C^\infty_c(B_R)$ it holds
\begin{equation}\label{eq:poincare:eps:degenerate}
    \Big(\frac{a+n}{2R}\Big)^2\int_{B_R\setminus \Sigma^A_\eps}|y|^au^2\leq\int_{B_R\setminus \Sigma^A_\eps}|y|^a|\D u|^2.
\end{equation}
\item[$ii)$] if $a+ n <2 $ and $u\in C^\infty_c(\R^d\setminus \Sigma^A_\eps)$ it holds
\[
    \Big(\frac{a+n-2}{2R}\Big)^2\int_{B_R\setminus{\Sigma^A_\eps}}|y|^au^2\leq\int_{B_R\setminus \Sigma^A_\eps}|y|^a|\D u|^2.
\]
\end{enumerate}
\end{Proposition}
\begin{proof}
Case $ i) $ follows from $ i) $ in Proposition \ref{L:Hardy_eps} with $ \delta = a $, while case $ ii) $ follows from $ ii) $ in Proposition \ref{L:Hardy_eps} with $ \delta = a - 2 $.
\end{proof}

\subsection{Stable trace inequalities on the boundary of the hole}

\begin{Lemma}[$\varepsilon$-stable trace inequality on the boundary of the hole]\label{L:trace:e}
Let $a+n>0$, $R>0$ given by Lemma \ref{L:appendix3} and $0<\e\ll1$. There exists a constant $c >0$ depends on $a$, $n$, $R$, $\lambda$, $\Lambda$, $\|A\|_{C^1(B_1)}$ such that 
\[
c \int_{\partial \Sigma_\eps^A\cap B_R}|y|^a |u|^2 d\sigma \leq G_\eps\int_{B_R \setminus \Sigma_\eps^A}|y|^a |\nabla u|^2 dz\,, \qquad  u \in {\hat H}^{1,a}(B_R\setminus \Sigma^A_\eps)\,,
\]
where
\[
G_\eps = 
\begin{cases}
\eps & a + n >2\\
\eps \log \eps & a + n = 2\\
\eps^{a+n-1} & 0 < a + n< 2 \,.
\end{cases}
\]
\end{Lemma}
\begin{proof}
First, we consider the case $A = \mathbb{I}$. The proof is similar to the proof of Lemma \ref{L:tracePo}, so we omit some details. By density, it suffices to prove the result for $u \in C^{\infty}_c(B_R \setminus \mathring \Sigma_\eps)$. Let $x \in B_{\sqrt{R^2-\eps^2}}^{d-n}$ be fixed, and notice that $(x,y) \in {B_R \setminus \mathring \Sigma_\eps}$ if and only if $\eps \leq |y| < R_x $, where $R_x = \sqrt{R^2-|x|^2} \leq R$. We point out that $u(x, R_x \sigma) = 0$ for any $\sigma \in \mathbb{S}^{n-1}$. 
Thus, using the fundamental theorem of calculus and the H\"older inequality we get
\[
|u(x, \eps \sigma)|^2 = |u(x, R_x \sigma) - u(x, \eps \sigma) |^2 \leq \int^{R_x}_\eps 
\tau^{a+n-1}|\nabla u|^2(\tau \sigma)d\tau \int^{R}_\eps \tau^{1-a-n}d\tau\,.
\]
Integrating over $\sigma \in \mathbb{S}^{n-1}$ we find
\begin{equation}\label{eq:epsTr_pass}
\int_{\partial B_\eps^n}|y|^a |u|^2 d\sigma = \eps^{a+n-1}\int_{\mathbb{S}^{n-1}}|u(x, \eps \sigma)|^2 d\sigma \leq \Big( \eps^{a+n-1}\int^{R}_\eps \tau^{1-a-n}d\tau \Big)\int_{B^{n}_{R^x}\setminus B_\eps^n}|y|^a |\nabla u|^2 dy\,.
\end{equation}
Notice that 
\[
\eps^{a+n-1}\int^{R}_\eps \tau^{1-a-n}d\tau \leq c \,G_\eps
\]
for a constant $c>0$ depending only on $a$, $n$ and $R$. To conclude, we integrate \eqref{eq:epsTr_pass} over $x \in B^{d-n}_{\sqrt{R^{2} - \eps^2}}$.

Let us now consider the general case. Let $u \in C^\infty_c(B_R \setminus \mathring \Sigma_\eps^A)$, and define $\tilde u = u \circ \Phi^{-1}$, where $\Phi$ is the $C^1$-diffeomorphism from Lemma \ref{L:appendix3}. For every $z \in B_R$, we have $|\Phi(z)| \leq cR$ with $c>0$ depending only on $\lambda$ and  $\Lambda$. Thus $\tilde u \in C^1_c(B_{cR} \setminus \mathring \Sigma_\eps)$. Moreover, taking into account $ii)$ and $iii)$ from Lemma \ref{L:appendix3}, and using the change of variables $(x, \tau ) = \Phi(z)$, we obtain
\[
\int_{\R^d \setminus \Sigma_\eps} |\tau|^a |\nabla \tilde u(x, \tau)|^2 dx d\tau = \int_{\R^d \setminus \Sigma_\eps^A} |A^{-\frac12}_3(z)y|^a (J_\Phi^\top J_\Phi)^{-1}\nabla u \cdot \nabla u |\det J_\Phi| dz \leq c \int_{\R^d \setminus \Sigma_\eps^A} |y|^a |\nabla u|^2 dz\,.
\]
On the other hand, by taking the same change of variables in the integration over $\partial\Sigma_\e$, for instance see \cite[\S 11]{Maggi}, and using Lemma \ref{L:appendix3}, $iv)$, we get 
\[
\int_{\partial \Sigma_\eps} |\tau|^a |\tilde u(x, \tau)|^2 d\sigma(x,\tau) = \int_{\partial\Sigma_\eps^A} |A^{-\frac12}_3(z)y|^a |u|^2 |\det  (\Pi_\eps \circ J_\Phi^\top J_\Phi\circ\Pi_\eps)| d\sigma(z) \geq c \int_{\partial \Sigma_\eps^A} |y|^a |u|^2 d\sigma(z)\,.
\]
To conclude, apply to $\tilde u$ the previous step, which holds also for functions in $C^1_c(B_{cR} \setminus \mathring \Sigma_\eps)$.
\end{proof}

\subsection{Stable Sobolev inequalities}

\begin{Lemma}[$\eps$-stable Sobolev embeddings]\label{L:Sstab}
Let $a+n>0$, and let $R>0$ be as in Lemma \ref{L:appendix3}. Moreover, assume $0\le \e\ll1$. When $a_+ + d >2$, define 
\[
2^*_a = \frac{2(a_++d)}{a_++ d - 2}\,.
\]
\begin{itemize}
\item[$i)$]
if $a_+ + d >2$, then for any $2 \leq q \leq 2^*_a$ it holds
\[
c\Big(\int_{B_R \setminus \Sigma_\eps^A}|y|^a |u|^q dz  \Big)^{\frac{2}{q}} \leq \int_{B_R\setminus \Sigma_\eps^A} |y|^a |\nabla u|^2 dz\,, \qquad u \in \hat H^{1,a}(B_R \setminus \Sigma_\eps),
\]
for a constant $c>0$ which depends only on $d$, $n$, $a$, $\lambda$, $\Lambda$, $R$ and $\|A\|_{C^{1}(B_R)}$.
\item[$ii)$] if $d = n = 2$ and $a \leq 0$, the same inequality holds for all $q \in [2,\infty)$, with a constant $c>0$ which depends only on $d$, $n$, $a$, $\lambda$, $\Lambda$, $R$, $\|A\|_{C^{1}(B_R)}$ and $q$.
\end{itemize}
\end{Lemma}

The proof is a consequence of some lemmas. The first lemma is an $\eps$-stable $L^1$-version of the Caffarelli-Kohn-Nirenberg inequality.

\begin{Lemma}[$\varepsilon$-stable $L^1$-CKN]\label{L:CKNy}
Let $a + n >0$, $\eps \geq 0$. Then for any $1 \leq q \leq \frac{n}{n-1}$ it holds
\[
c\Big(\int_{\R^n \setminus B^n_\eps}|y|^a |u|^q dy  \Big)^{\frac{1}{q}} \leq \int_{\R^n\setminus B^n_\eps} |y|^{\frac{a+n}{q} - (n-1)} |\nabla u| dy\,, \qquad u \in C^{0,1}_c(\R^n),
\]
for a constant $c>0$ which depends only on $d$, $n$ and $a$.
\end{Lemma}

\begin{proof}
Let $1 \leq q \leq \frac{n}{n-1}$ be fixed. Set, for $\eps \geq 0$, 
\[
S_{q,\eps} = \inf_{u \in C_c^{0,1}(\R^n)} \frac{\int_{\R^n\setminus B^n_\eps} |y|^{\frac{a+n}{q} - (n-1)} |\nabla u| dy}{\Big(\int_{\R^n \setminus B^n_\eps}|y|^a |u|^q dy  \Big)^{\frac{1}{q}}}\,.
\]
Our aim is to show that $S_{q,\eps} \geq c >0$, for a constant $c>0$ not depending on $\eps$ or $q$. 

First we notice that, by a standard scaling argument, it holds $S_{q,\eps} = S_{q,1}$ for any $\eps >0$. 
Next we show that $S_{q, 0} \leq c S_{q,1}$ for a constant $c>0$ independent from $q$. Let $u \in C_c^{0,1}(\R^n)$ be fixed. Arguing as in the proof of Lemma \ref{L:Hardy_eps} (take $|u|$ instead of $|u|^2$ in \eqref{eq:Hard_pass}), we see that the following Hardy-type inequality holds: 
\begin{equation}\label{eq:p1hardy}
\frac{a+n}{q}\int_{\R^n\setminus B^n_1} |y|^{\frac{a+n}{q} - n} |u|dy \leq \int_{\R^n\setminus B^n_1} |y|^{\frac{a+n}{q} - (n-1)} |\nabla u|dy\,.
\end{equation}

Next, we extend $u$ in the unitary ball $B^n_1$ by the Kelvin transform
\[
\tilde u(y) = 
\begin{cases}
u(y) & |y| \geq 1\\
|y|^{-\frac{2(a+n)}{q}} u(I(y)) & |y| \leq 1\,,
\end{cases}
\]
where $I(y) = |y|^{-2}y$ is the inversion map. Clearly, $\tilde u \in C_c^{0,1}(\R^n)$. Since $J_IJ_I =|y|^{-4}{\mathbb I}_n$ and thus $|\det J_I| = |y|^{-2n}$, a standard computation yields
\[
\int_{\R^n} |y|^a |\tilde u|^qdy = 2\int_{\R^n\setminus B^n_1} |y|^a |u|^qdy\,.
\] 
Moreover, using also \eqref{eq:p1hardy}, we estimate
\[
\begin{aligned}
\int_{\R^n} |y|^{\frac{a+n}{q} - (n-1)} |\nabla \tilde u|dy &\leq 2 \Big(\int_{\R^n\setminus B^n_1} |y|^{\frac{a+n}{q} - (n-1)} |\nabla u|dy +\frac{a+n}{q} \int_{\R^n\setminus B^n_1} |y|^{\frac{a+n}{q} - n} |u|dy\Big)\\
& \leq 3 \int_{\R^n\setminus B^n_1} |y|^{\frac{a+n}{q} - (n-1)} |\nabla u|dy\,.
\end{aligned}
\]
Thus we infer that 
\[
S_{q,0} \leq \frac{\int_{\R^n} |y|^{\frac{a+n}{q} - (n-1)} |\nabla \tilde u| dy}{\Big(\int_{\R^n }|y|^a |\tilde u|^q dy  \Big)^{\frac{1}{q}}} \leq c \frac{\int_{\R^n \setminus B^n_1} |y|^{\frac{a+n}{q} - (n-1)} |\nabla u| dy}{\Big(\int_{\R^n\setminus B^n_1 }|y|^a |u|^q dy  \Big)^{\frac{1}{q}}}\,,
\]
for any $u \in C_c^{0,1}(\R^n)$, that is, $S_{q,0} \leq cS_{q,1}$. Since by the classical Caffarelli-Kohn-Nirenberg inequality \cite{CafKohNir84} we have $S_{q, 0} >0$ uniformly in $q$, the proof is complete.
\end{proof}

The second lemma we need is an equivalence between weighted Sobolev-type inequalities and the $ L^1 $-Moser weighted inequality. The unweighted version of this equivalence for functions defined on all the space is well known; see, e.g., \cite{BCM95} and the references therein. The presence of the weights and the restriction to subsets does not add any difficulty; nevertheless, we present the proof for the sake of completeness.

\begin{Lemma}\label{L:StoM}
Let $ d \geq 1 $ and let $A \subset \mathbb{R}^d$ be a set with smooth boundary and non empty interior.  Let $\mu, \nu \in L^1_{\rm loc}(A)$, $\mu, \nu \geq 0$. Moreover, let $q \geq 1$. Then
\begin{equation}\label{eq:Sd}
c_0\Big(\int_{A}\mu(z) |f|^q dz  \Big)^{\frac{1}{q}} \leq \int_{A}\nu(z)|\nabla f| dz\,, \quad \text{for every} \quad f \in C_c^{0,1}(\R^d)
\end{equation}
if and only if 
\begin{equation}\label{eq:Mos}
c_0\int_{A}\mu(z)|f|^{1+\frac{q-1}{q}} dz \leq \Big(\int_{A}\mu(z) |f| dz  \Big)^{\frac{q-1}{q}} \int_{A} \nu(z)|\nabla f| dz\, \quad \text{for every} \quad f \in C_c^{0,1}(\R^d)\,,
\end{equation}
for the same constant $c_0>0$.  
\end{Lemma}
\begin{proof}
If $q = 1$ the two inequalities coincide. Then, let $q >1$.

Since \eqref{eq:Sd}$\Rightarrow$\eqref{eq:Mos} simply follows by the H\"older inequality, we only need to show \eqref{eq:Mos}$\Rightarrow$\eqref{eq:Sd}. Assume then that \eqref{eq:Mos} holds, and fix $f \in C_c^{0,1}(\R^n)$, $\delta >1$, and for $k \in \Z$ define
\[
f_k(z) = 
\begin{cases}
\delta^{k} & |f| \geq \delta^{k+1}\\
|f| - \delta^k & \delta^k \leq |f|< \delta^{k+1}\\
0 & |f| < \delta^k\,.
\end{cases}
\]
Notice that $f_k \in C_c^{0,1}(\R^n)$ and $f_k \geq 0$. Next we define
\[
a_k = \delta^{qk}\int_{A\cap\{\delta^k\leq|f| \}}\mu(z) dz\,, \quad \text{ and } \quad b_k = \int_{A\cap\{\delta^k\leq|f|< \delta^{k+1} \}}\nu(z)|\nabla f| dz\,.
\]
We have
\[
\int_{A}\mu(z)|f_k|^{1+\frac{q-1}{q}} dz \geq \delta^{k + k\frac{q-1}{q}- k(q-1)}a_{k+1}\,,\qquad
\int_{A}\mu(z) |f_k| \leq \delta^{k+1-qk}a_k\,.
\]
Therefore, applying \eqref{eq:Mos} to $f_k$ we infer
\[
a_{k+1} \leq c_0^{-1}\delta^{\frac{q^2+q-1}{q}}b_k a_k^{\frac{q-1}{q}}\,.
\] 
Summing over $k \in \Z$ and using the H\"older inequality we get 
\[
\sum_{k\in \Z}a_k = \sum_{k\in \Z}a_{k+1} \leq c_0^{-1}\delta^{\frac{q^2+q-1}{q}} \sum_{k\in \Z} b_k a_k^{\frac{q-1}{q}}\leq c_0^{-1}\delta^{\frac{q^2+q-1}{q}} \Big( \sum_{k \in \Z} b_k^q\Big)^{\frac{1}{q}} \Big(\sum_{k \in \Z}a_k\Big)^{\frac{q-1}{q}}\,,
\]
hence
\[
\Big(\sum_{k\in \Z}a_k\Big)^{\frac{1}{q}} \leq c_0^{-1}\delta^{\frac{q^2+q-1}{q}}\Big( \sum_{k \in \Z} b_k^q\Big)^{\frac{1}{q}} \leq c_0^{-1}\delta^{\frac{q^2+q-1}{q}}\sum_{k \in \Z} b_k = c_0^{-1}\delta^{\frac{q^2+q-1}{q}}\int_\Omega \nu(z) |\nabla f| dz\,.
\]
Finally, since
\[
\int_{\Omega}\mu(z) |f|^q dz = \sum_{k \in \Z} \int_{A\cap\{\delta^k\leq|f|< \delta^{k+1} \}}\mu(z) |f|^q dz \leq \delta^q \sum_{k \in \Z} a_k\,,
\]
we infer that
\[
c_0\Big(\int_{A}\mu(z) |f|^q dz  \Big)^{\frac{1}{q}} \leq \delta^{\frac{q^2+2q-1}{q}}\int_A \nu(z) |\nabla f| dz\,.
\]
Letting $\delta \to 1$ we obtain \eqref{eq:Sd} and complete the proof.
\end{proof}

The final preparatory lemma is an $L^1$-version of Lemma \ref{L:Sstab}, in the case $A = \mathbb{I}$. In its proof, a technique introduced in \cite{CouGriLev03} is crucial; see also \cite{Grig85}. This allows us to recover $L^1$-Sobolev inequalities on weighted spaces which are Cartesian products of weighted spaces. 

\begin{Lemma}\label{L:L1sob}
Let $a + n >0$, $R >0$, $\eps \in [0, R)$, and $1\leq q \leq \frac{a_+ + d}{a_+ + d-1}$. It holds
\[
c \Big(\int_{B_R \setminus \Sigma_\eps}|y|^a|u|^q dz \Big)^\frac{1}{q} \leq \int_{B_R \setminus \Sigma_\eps} |y|^a|\nabla u|dz\,, \qquad u \in C^\infty_c(B_R)\,,
\]
for a constant $c>0$ which depends on $d$, $n$, $R$ and $a$.

\end{Lemma}
\begin{proof}
The case $d - n= 0$ derives directly form Lemma \ref{L:CKNy} by noticing that for 
\[
1 \leq q \leq \min\Big\{\frac{n}{n-1}, \frac{a+ n}{a+ n-1}\Big\} = \frac{a_++ n}{a_++ n-1}
\]
it holds 
\[
|y|^{\frac{a + n}{q} - (n-1)} \leq |y|^a R^{\frac{a + n}{q} - (a + n-1)}\leq c |y|^a\,,
\]
where $c >0$ depends only on $a$, $n$, $R$.

Let then $d - n \geq 1$. By classical and well known results (one can use, for instance, the classical Caffarelli-Kohn-Nirenberg inequality when $d - n >1$ and \cite[Proposition 1]{Cas17} for the case $d - n = 1$) we have that for every $p \geq 1$, $(d-n-1) p \leq d-n$, it holds
\[
c_1\Big(\int_{\R^{d-n}}|f|^p dx \Big)^{\frac{1}{p}} \leq \int_{\R^{d-n}} |x|^{\frac{d-n}{p} - (d-n-1)} |\nabla f| dx\,, \qquad  f \in C_c^{0,1}(\R^{d-n})\,, 
\]
for a constant $c_1 >0$ not depending on $p$. Thus, using Lemma \ref{L:StoM} we get that
\begin{equation}\label{eq:L1mosx}
c_1\int_{\R^{d-n}}|f|^{1+\frac{p-1}{p}} dx \leq \Big(\int_{\R^{d-n}} |f| dx  \Big)^{\frac{p-1}{p}} \int_{\R^{d-n}} |x|^{\frac{d-n}{p} - (d-n -1)}|\nabla_x f| dx\,, \qquad f \in C_c^{0,1}(\R^{d-n})\,,
\end{equation}
whenever $p \geq 1$, $(d-n-1) p \leq d-n$.

On the other hand, by Lemma  \ref{L:CKNy} and Lemma \ref{L:StoM} we have that for every $g \in C_c^{0,1}(\R^{n})$ it holds
\begin{equation}\label{eq:L1mosy}
c_2\int_{\R^{n}\setminus B^n_\eps}|y|^a|g|^{1+\frac{q-1}{q}} dy \leq \Big(\int_{\R^n \setminus B^n_\eps}|y|^a |g| dy  \Big)^{\frac{q-1}{q}} \int_{\R^n \setminus B^n_\eps} |y|^{\frac{a+n}{q} - (n-1)}|\nabla_y g| dy\,,
\end{equation}
whenever $1 \leq q \leq \frac{n}{n-1}$, where the constant $c_2>0$ does not depend neither on $q$ nor $\eps$. 

The inequalities \eqref{eq:L1mosx} and \eqref{eq:L1mosy} are equivalent, respectively, to the $F$-Sobolev and $G$-Sobolev inequalities
\begin{equation}\label{eq:FGSob}
\begin{gathered}
\int_{\R^{d-n}}|f| F\Big(\frac{|f|}{\|f\|_{L^{1}(\R^{d-n})}}\Big) dx \leq \int_{\R^{d-n}} |x|^{\frac{d-n}{p} - (d-n -1)}|\nabla_x f| dx\, \quad \text{for every} \quad f \in C_c^{0,1}(\R^{d-n})\\
\int_{\R^n \setminus B_\eps^n}|y|^a|g| G\Big(\frac{|g|}{\|g\|_{L^{1,a}(\R^n\setminus B_\eps^n)}}\Big)dy \leq \int_{\R^n \setminus B^n_\eps} |y|^{\frac{a+n}{q} - (n-1)}|\nabla_y g| dy\, \quad \text{for every} \quad g \in C_c^{0,1}(\R^{n})\,,
\end{gathered}
\end{equation}
where
\[
F(s) = c_1 s^{\frac{p-1}{p}} \quad \text{ and } \quad G(t) = c_2 t^{\frac{q-1}{q}}\,.
\]
Let $u \in C^{0,1}_c(\R^d)$. Arguing as in the proof of \cite[Proposition 3.3]{CouGriLev03} with minor modifications, thanks to \eqref{eq:FGSob} we obtain 
\[
\int_{\R^d \setminus \Sigma_\eps}|y|^a|u|H\Big( \frac{|u|}{\|u\|_{L^{1,a}(\R^d \setminus \Sigma_\eps)}}\Big)dxdy \leq \int_{\R^d \setminus \Sigma_\eps} (|x|^{\frac{d-n}{p} - (d-n -1)}|y|^a + |y|^{\frac{a+n}{q} - (n-1)})|\nabla u| dxdy\,,
\]
where
\[
H(r) = \inf_{st = r}[F(s) + G(t)]\,.
\]
After some standard computations, we see that for $p \geq 1$, $(d-n-1) p \leq d-n$, $1 \leq q \leq \frac{n}{n-1}$, it holds
\[
H(r) \geq \min\{c_1, c_2\} r^{\frac{\gamma_{p,q}-1}{\gamma_{p,q}}}\,,
\]
where
\[
\gamma_{p,q} = \frac{p(q-1) + q(p-1)}{pq-1}\,, \qquad \gamma_{1,1} = 1\,.
\]
Thus
\[
\begin{aligned}
\min\{c_1, c_2\}&\int_{\R^d \setminus \Sigma_\eps}|y|^a|u|^{1 + \frac{\gamma_{p,q}-1}{\gamma_{p,q}}}dxdy \\
&\leq \Big(\int_{\R^d \setminus \Sigma_\eps}|y|^a |u| dy  \Big)^{\frac{\gamma_{p,q}-1}{\gamma{p,q}}} \int_{\R^d \setminus \Sigma_\eps} (|x|^{\frac{d-n}{p} - (d-n -1)}|y|^a + |y|^{\frac{a+n}{q} - (n-1)})|\nabla u| dxdy\,.
\end{aligned}
\]
Applying Lemma \ref{L:StoM} once again, we infer that, for every $u \in C^{0,1}_c(\R^d)$ it holds
\begin{equation}\label{eq:CKNsupergeneral}
c\Big(\int_{\R^d\setminus \Sigma_\eps}|y|^a|u|^{\gamma_{p,q}} dz\Big)^{\frac{1}{\gamma_{p,q}}}\leq \int_{\R^d\setminus \Sigma_\eps}(|x|^{\frac{d-n}{p} - (d-n -1)}|y|^a + |y|^{\frac{a+n}{q} - (n-1)})|\nabla u| dz\,,
\end{equation}
for some constant $c >0$ not depending on $p$, $q$, $\eps$.

Now, let us take $1 \leq \tau \leq \frac{d}{d-1}$, and put $q = \frac{(d-n)-(d-n-1)\tau}{(d-n+1) - (d-n)\tau}$ in \eqref{eq:CKNsupergeneral}. Further, let $p= \frac{d-n}{d-n-1}$ if $d-n >1$ or take the limit $p \to \infty$ if $d-n = 1$. We find that for every $1 \leq \tau \leq \frac{d}{d-1}$ and $u \in C^{0,1}_c(\R^d)$ it holds
\[
c\Big(\int_{\R^d \setminus \Sigma_\eps}|y|^a|u|^\tau dz\Big)^{\frac{1}{\tau}}\leq \int_{\R^d \setminus \Sigma_\eps} |y|^a(1+ |y|^{\frac{a + d - \tau(a + d-1)}{d-n - \tau(d-n-1)}})|\nabla u| dz\,.
\]

Let now $u \in C^\infty_c(B_R)$. To conclude it suffice to notice that for 
\[
1 \leq \tau \leq \min\Big\{\frac{d}{d-1}, \frac{a+ d}{a+ d-1}\Big\} = \frac{a_++ d}{a_++ d-1}
\]
it holds 
\[
|y|^{\frac{a + d - \tau(a + d-1)}{d-n - \tau(d-n-1)}} \leq R^{\frac{a + d - \tau(a + d-1)}{d-n - \tau(d-n-1)}} \leq c \,,
\]
for a constant $c>0$ not depending on $\eps$ nor $\tau$. 
\end{proof}

\begin{proof}[Proof of Lemma \ref{L:Sstab}] The proof is divided in two steps. 

\smallskip

\noindent
\emph{Step 1: The case $A = \mathbb{I}$.} Fix $q$ such that $2 \leq q \leq 2^*_a$ if $ a_+ + d >2 $ or $q \geq 2$ if $a_++d= 2$. Let us denote
\[
r = \frac{2q}{q+2}\,,
\]
and notice that $r < q$ and $1 \leq r \leq \frac{a_++d}{a_++d-1}$. 

We fix $ u \in C^\infty_c (B_R)$ and define $v = |u|^{\frac{q}{r}}$. Since $q >r$, then $v$ is smooth and we can compute
\[
\int_{B_R\setminus \Sigma_\eps}|y|^a |u|^q dz = \int_{B_R\setminus \Sigma_\eps}|y|^a |v|^r dz \leq c \Big( \int_{B_R\setminus \Sigma_\eps}|y|^a|\nabla v| dz \Big)^{r},
\]
where the last inequality holds true thanks to Lemma \ref{L:L1sob}.   

Next we use that $|\nabla v| = \frac{q}{r}|u|^{\frac{q-r}{r}}|\nabla u|$ and the H\"older inequality to obtain 
\[
\frac{c}{q+2} \Big(\int_{B_R\setminus \Sigma_\eps}|y|^a |u|^q dz\Big)^{\frac{1}{r}} \leq \Big( \int_{B_R\setminus \Sigma_\eps}|y|^a|\nabla u|^2 dz \Big)^{\frac{1}{2}} \Big( \int_{B_R\setminus \Sigma_\eps}|y|^a|u|^{\frac{2(q-r)}{r}} dz \Big)^{\frac{1}{2}},
\]
and since 
\[
\frac{2(q-r)}{r} = q \quad \text{ and }\quad \frac{1}{r} - \frac{1}{2} = \frac{1}{q}\,,
\]
we readily conclude. 

\smallskip

\emph{Step 2: the general case.} Let $u \in C^\infty_c(B_R)$, and define $\tilde u = u \circ \Phi^{-1}$, where $\Phi$ is the $C^1$-diffeomorphism from Lemma \ref{L:appendix3}. For every $z \in B_R$, we have $|\Phi(z)| \leq cR$ with $c>0$ depending only on $\lambda$, $\Lambda$. Thus $\tilde u \in C^1_c(B_{cR})$. Moreover, taking into account $ii)$ and $iii)$ from Lemma \ref{L:appendix3}, and using the change of variables $(x, \tau ) = \Phi(z)$, we obtain
\[
\int_{\R^d \setminus \Sigma_\eps} |\tau|^a |\tilde u(x, \tau)|^q dx d\tau = \int_{\R^d \setminus \Sigma_\eps^A} |A^{-\frac12}_3(z)y|^a |u|^q |\det J_\Phi| dz \geq c \int_{\R^d \setminus \Sigma_\eps^A} |y|^a |u|^q dz\,,
\]
\[
\int_{\R^d \setminus \Sigma_\eps} |\tau|^a |\nabla \tilde u(x, \tau)|^2 dx d\tau = \int_{\R^d \setminus \Sigma_\eps^A} |A^{-\frac12}_3(z)y|^a (J_\Phi^\top J_\Phi)^{-1}\nabla u \cdot \nabla u |\det J_\Phi| dz \leq c \int_{\R^d \setminus \Sigma_\eps^A} |y|^a |\nabla u|^2 dz\,.
\]
The conclusion follows applying to $\tilde u $ the previous step. Notice that the latter holds true for functions in $C^1$ by a standard density argument.
\end{proof}

\section{Solutions, stable \texorpdfstring{$L^\infty $}{L} estimates and approximation results}\label{sec:4}

In this section, we provide the notion of weak solutions to \eqref{generalPDE}. We define the regularized problems for \eqref{generalPDE}, and we prove some approximation lemmas. Moreover, we prove local stable $L^\infty$ estimates for the regularized solutions.

\subsection{Notions of solutions}

\begin{Definition}\label{def:weak:sol:hom:0}
Let $a+n>0$, $R>0$ and $A$ satisfying \eqref{eq:unif:ell}. Let $f\in L^{p,a}(B_R)$ where $p= (2^*_a)'$ if $d+a_+>2$ or $p>1$ if $d+a_+=2$ and $F\in L^{q,a}(B_R)^d$ where $q\ge2$. We say that $u$ is a weak solution to
\begin{equation}\label{eq:weak:sol:conormal:0}
-\dive(|y|^a A \nabla u)= |y|^a f+\dive(|y|^aF), \quad \text{in }B_R
\end{equation}
if $u\in H^{1,a}(B_R)$ and satisfies
\begin{equation}\label{eq:weak:sol:integral}
\int_{B_R}|y|^a A\D u\cdot \D\phi = \int_{B_R}|y|^a (f\phi-F\cdot\D\phi),
\end{equation}
for every $\phi\in C_c^\infty(B_R).$

We say that $u$ is an entire solution to 
\begin{equation*}
-\dive(|y|^a A \nabla u)= |y|^a f+\dive(|y|^aF), \quad \text{in }\R^d,
\end{equation*}
if $u$ is a weak solution to \eqref{eq:weak:sol:conormal:0} for every $R>0$.
\end{Definition}

\begin{remark}\label{rem:conormal}
  We highlight that our notion of weak solution implies a \emph{weighted conormal boundary condition} on the lower-dimensional set $\Sigma_0$. This is a consequence of the fact that the weak formulation \eqref{eq:weak:sol:integral} involves test functions whose support may touch the thin manifold $\Sigma_0$. We refer to these solutions as \emph{solutions across} $\Sigma_0$. Let us assume for sake of simplicity that $f = 0$, $F = 0$, and fix $\phi \in C_c^\infty(B_R)$. Multiplying \eqref{eq:weak:sol:conormal:0} by $\phi$ and integrating by parts in $B_R \setminus \Sigma_\varepsilon$, we obtain
\[
0 = \int_{B_R \setminus \Sigma_{\varepsilon}} -\dive(|y|^a A\nabla u) \phi \, dz = \int_{B_R \setminus \Sigma_{\varepsilon}} |y|^a A\nabla u \cdot \nabla \phi \, dz - \int_{\partial \Sigma_\varepsilon \cap B_R} |y|^a \phi A\nabla u \cdot \nu \, d\sigma.
\]  
Formally, taking the limit as $\varepsilon \downarrow 0$, we find
\[
\int_{B_R} |y|^a A\nabla u \cdot \nabla \phi \, dz  = \int_{B_R^{d-n}} \mathcal{D}_u(x) \phi(x, 0) \, dx,
\]  
where
\[
\mathcal{D}_u(x) := -\lim_{\varepsilon \downarrow 0} \varepsilon^{1-n} \int_{\partial B_{\varepsilon}^n} |y|^{a+n-2} A\nabla u \cdot y \, d\sigma(y).
\]  
Hence, in the weak formulation of \eqref{eq:weak:sol:conormal:0}, we are assuming that $\mathcal{D}_u = 0$. It is worth noting that if $a + n \geq 2$, due to the zero weighted capacity of the thin manifold, only solutions to \eqref{eq:weak:sol:conormal:0} make sense, since one could not impose any different boundary condition at $\Sigma_0$. Conversely, when $a + n \in (0, 2)$, the weighted capacity of $\Sigma_0$ is positive and locally finite, allowing the imposition of both inhomogeneous Dirichlet and inhomogeneous conormal boundary conditions, respectively $u=g$ and $\mathcal{D}_u = h $ on $\Sigma_0$.

\end{remark}

\begin{Definition}\label{def:weak:sol:hom:eps}
Let $a+n>0$, $R>0$,  $A \in C^1(B_R;\R^{d,d})$ satisfying \eqref{eq:unif:ell} and $0< \e\ll 1$. Let $f\in L^{p,a}(B_R\setminus\Sigma_\e^A)$, where $p= (2^*_a)'$ if $d>2$ or $p>1$ if $d=2$ and $F\in L^{q,a}(B_R\setminus\Sigma_\e^A)^d$ where $q\ge2$. We say that $u$ is a weak solution to
\begin{equation}\label{eq:weak:sol:conormal:eps}
\begin{cases}
-\dive(|y|^a A \nabla u)= |y|^a f+\dive(|y|^aF), & \text{in }B_R \setminus\Sigma_\e^A\\
(A\D u+F)\cdot\nu=0      & \text{ on } \partial\Sigma_\e^A\cap B_R
\end{cases}
\end{equation}
if $u\in H^{1,a}(B_R\setminus\Sigma_\e^A)$ and satisfies
\[
\int_{B_R\setminus\Sigma_\e^A}|y|^a A\D u\cdot \D\phi dz= \int_{B_R\setminus\Sigma_\e^A}|y|^a (f\phi-F\cdot\D\phi)dz,
\]
for every $\phi\in C_c^\infty(B_R).$

We say that $u$ is an entire solution to  
\begin{equation*}
\begin{cases}
-\dive(|y|^a A \nabla u)= |y|^a f+\dive(|y|^aF), & \text{in }\R^d \setminus\Sigma_\e^A\\
(A\D u+F)\cdot\nu=0     & \text{ on } \partial\Sigma_\e^A
\end{cases}
\end{equation*}
if $u$ is a weak solution to \eqref{eq:weak:sol:conormal:eps} for every $R>0$.

Moreover, we adopt the convention that solutions to \eqref{eq:weak:sol:conormal:eps} with $\e = 0$ coincide with those of \eqref{eq:weak:sol:conormal:0}.
\end{Definition}

\begin{remark}\label{R:existence:uniqueness}
    Fixed $ \Bar{u}\in H^{1,a}(B_R)$, a weak solutions to \eqref{eq:weak:sol:conormal:0} satisfying the boundary condition $u-\Bar{u}\in H^{1,a}_0(B_R)$ is a minimizer of the functional
\[
\mathcal{E}(v):=\int_{B_R} |y|^a \Big(\frac{A\D v\cdot\D v}{2}-fv+F\cdot\D v\Big)dz,
\]
over $X:=\{v\in H^{1,a}(B_R) \ \mid \ v-\Bar{u}\in H^{1,a}_0(B_R) \}$.

By the Poincar\'e inequality and the Sobolev type embedding, see Proposition \ref{P:Poincare} and Lemma \ref{L:Sstab} respectively, the assumption on the data allows us to obtain that the functional $\mathcal{E}$ is coercive, so by the Weierstrass theorem, we have existence and uniqueness of solutions to \eqref{eq:weak:sol:conormal:0} with prescribed trace on $\partial B_R$ (see also Lemma \ref{L:H^1_0:trace}).
\end{remark}

\subsection{Caccioppoli inequality and local boundedness of solutions}

In this section we provide local $L^2 \to L^\infty$ estimates for weak solutions to \eqref{eq:weak:sol:conormal:eps} that are uniform with respect to $\varepsilon$.  
We note that, since the weight $|y|^a$ is $2$-admissible, the local $L^\infty$ bounds for weak solutions to \eqref{eq:weak:sol:conormal:0} (when $\varepsilon = 0$ and $A \in L^\infty$) follows from \cite{FabKenSer82}.

Furthermore, assuming that the matrix $A \in C^1(B_R)$, where the radius $R>0$ is taken small enough in order to ensure that the $\varepsilon$-stable Sobolev embedding holds true (see Lemma \ref{L:Sstab}), we obtain $\varepsilon$-stable $L^\infty$ bounds for weak solutions to \eqref{eq:weak:sol:conormal:eps} (when $\varepsilon > 0$).  
This result is established through a standard argument based on an iteration technique involving the Caccioppoli-type inequality below and Sobolev embeddings.
The proof is omitted here, as it has been carried out in a similar context in \cite[\S 2.4]{Fio24}.

\begin{Lemma}[Caccioppoli inequality]\label{lemma:caccioppoli}
Let $a+n>0$, $R>0$, $0\le \e\ll 1$, $p\geq (2_*^a)'$ if $d+a_+>2$ or $p>1$ if $d+a_+=2$, $q\ge2$ and let $A \in C^{1}(B_R)$ satisfy \eqref{eq:unif:ell}. Let $f\in L^{p,a}(B_R\setminus\Sigma_\e^A)$, $F\in L^{q,a}(B_R\setminus\Sigma_\e^A)^d$, and let $u$ be a weak solution to \eqref{eq:weak:sol:conormal:0} if $\e=0$ or to \eqref{eq:weak:sol:conormal:eps} if $\e>0$. Then, there exists a constant $c>0$ depending only on $d$, $n,$ $a,$ $\lambda$ and $\Lambda$ such that for every $ 0 < R_1< R_2<R$ it holds 
\begin{equation}\label{eq:caccioppoli}
\begin{aligned}
    \int_{B_{R_1}\setminus\Sigma_\e^A}|y|^a |\D u|^2 dz \le c \Big(&
    {(R_2-R_1)^{-2}}  
    \|u\|^2_{L^{2,a}(B_{R_2}\setminus\Sigma_\e^A)}
    +
    \|f\|^2_{L^{p,a}(B_{R_2}\setminus\Sigma_\e^A)}+ 
    \|F\|^2_{L^{q,a}(B_{R_2}\setminus\Sigma_\e^A)}
    \Big) \,.
\end{aligned}
\end{equation}
In the case $d+a_+=2$ the constant $c$ depends also on $p$. Moreover, if $\e=0$ it is enough to assume that $A \in L^\infty(B_R)$.
\end{Lemma}

\begin{Proposition}[$\varepsilon$-stable $L^\infty$ bounds]\label{prop:moser}
Let $a+n>0$, $0\leq \e\ll 1$, and $A \in C^1(B_R)$ satisfy \eqref{eq:unif:ell}, where $R>0$ is given by Lemma \ref{L:appendix3}. Furthermore, let $p>(a_++d)/{2}$, $q>a_++d$ and let $f\in L^{p,a}(B_R\setminus\Sigma_\e^A)$ and $F\in L^{q,a}(B_R\setminus\Sigma_\e^A)^d$ and let $u$ be a weak solution \eqref{eq:weak:sol:conormal:eps}.

Then, for every $0<r<R$, there exists a constant $c>0$ depending only on $d$, $n$, $a,$ $\lambda,$ $\Lambda,$ $p,$ $q,$ $r$, $\|A\|_{C^1(B_R)}$ and $R$ such that
  \[
      \|u\|_{L^\infty(B_r\setminus\Sigma_\e^A)}\le c\big(
      \|u\|_{L^{2,a}(B_{R}\setminus\Sigma_\e^A)}+
      \|f\|_{L^{p,a}(B_{R}\setminus\Sigma_\e^A)}+
      \|F\|_{L^{q,a}(B_{R}\setminus\Sigma_\e^A)}
      \big).
  \]
  In addition, if $\e=0$ and $u$ is a weak solution to to \eqref{eq:weak:sol:conormal:0}, the same estimate holds under the assumption that $A \in L^\infty(B_R)$ and the constant does not depend on $\|A\|_{C^1(B_R)}$.
\end{Proposition}

\subsection{Approximation results}

In this section, we establish two approximation results. The first one, inspired by \cite[Lemma 2.15]{SirTerVit21a} and \cite[Lemma 4.3]{AudFioVit24a}, asserts that any solution $ u $ of \eqref{eq:weak:sol:conormal:0} ($\varepsilon=0$) can be approximated by a family of solutions $ u_\varepsilon $ of the conormal problem on $\eps$-perforated domains \eqref{eq:weak:sol:conormal:eps} ($0 < \varepsilon \ll 1$). Here, we assume that the matrix $A \in C^1$, in order to have a $C^1$ boundary $\partial \Sigma_\e^A$, see Remark \ref{R:C1:boundary}.
The second result shows that any solution $ u $ of \eqref{eq:weak:sol:conormal:0} ($\varepsilon=0$) can be approximated by a family of solutions to the same problem but with smooth coefficients, via a standard mollification technique.

\begin{Lemma}\label{L:approximation}
    Let $a+n>0$, $R>0$ given by Lemma \ref{L:appendix3}, $r\in(0,R)$ and $A \in C^1(B_R)$ satisfying \eqref{eq:unif:ell}. Let $f\in L^{p,a}(B_R)$ where $p\ge (2^*_a)'$ if $d+a_+>2$ or $p>1$ if $d+a_+=2$, $F\in L^{q,a}(B_R)^d$ where $q\ge 2$ and let $u$ be a weak solution to \eqref{eq:weak:sol:conormal:0}.

    Then, there exists a family $\{u_\e\}_{0<\e\ll 1}$, such that $u_\e$ are weak solutions to
\[
\begin{cases}
-\dive(|y|^a A \nabla u_\e)= |y|^a f+\dive(|y|^aF), & \text{in }B_{r}\setminus  \Sigma_{\e}^A \\
(A\D u+F)\cdot\nu=0, & \text{on } \partial \Sigma_\e^A \cap B_{r},
\end{cases}
\]
\begin{equation}\label{eq:stima:u:e}
\|u_\e\|_{H^{1,a}(B_{r}\setminus\Sigma_{\e}^A)}\le c \big(\|u\|_{H^{1,a}(B_R)}+\|f\|_{L^{p,a}(B_R)}
+\|F\|_{L^{q,a}(B_R)}\big)\,,
\end{equation}
for some constant $c>0$ depending only on $d$, $n,$ $ a,$ $ \lambda,$ $ \Lambda$, $r$, $R$, $\|A\|_{C^{1}(B_R)}$ and there exists a sequence $\e_k\to0$ such that
\[
 u_{\e_k}\to u \text{ in } H^{1}_{\rm loc}(B_{r}\setminus\Sigma_0).
\]
Moreover, in the case $d+a_+=2$ the constant $c$ depends also on $p$.
\end{Lemma}

\begin{proof}  
We consider the case $d + a_+ > 2$, as the other one can be treated analogously. The proof is divided into several steps.

\smallskip

\noindent\emph{Step 1. Cut-off solution $\tilde{u}$.}

\smallskip
Let us consider a cut-off function $\xi\in C_c^\infty(B_R)$ such that  
\[
   \xi=1 \text{ on } B_{r}\,,\qquad \supp(\xi)\subset B_{\frac{R+r}{2}}\,,\qquad 0\le\xi\le1\,,\qquad |\D\xi|\le c_0\,,
\]
for some $c_0>0$ depending only on $d$, $r$ and $R$. Define $\tilde{u}:=\xi u$, and notice that $\tilde u \in {H}^{1,a}_0(B_R)$ by construction.

Fix $\phi\in C_c^\infty(B_R)$. Then,
\begin{align*}
&\int_{B_R}|y|^a   A\D \tilde{u}\cdot\D\phi\,dz= \int_{B_R}|y|^a  \Big( \xi A\D {u} \cdot\D\phi +  u   A\D \xi \cdot\D\phi  \Big)\,dz \\
&= 
\int_{B_R}|y|^a  \Big( A\D {u} \cdot \D(\phi\xi) - 
 \phi A\D {u} \cdot \D \xi +
 u   A\D \xi \cdot\D\phi \Big)\, dz\\
&= 
\int_{B_R}|y|^a\Big( 
f\phi\xi - F\cdot \D (\phi\xi)
- 
 \phi A\D {u} \cdot \D \xi +
 u   A\D \xi \cdot\D\phi \Big)\, dz\\
 &= 
 \int_{B_R}|y|^a \Big( 
 f\phi\xi
 - \xi F\cdot \D \phi
 - \phi F\cdot \D \xi
 - 
 \phi A\D {u} \cdot \D \xi +
 u   A\D \xi \cdot\D\phi \Big)\, dz\, ,
\end{align*}
that is, $\tilde{u}$ is a weak solution to 
\begin{equation}\label{eq:cut:off:approximation}
\begin{cases}
    -\dive(|y|^a A \D\tilde{u})=|y|^a \tilde{f}+ \dive(|y|^a  \tilde{F})& \text{ in }B_R,\\
    u=0 & \text{ on }\partial B_R
\end{cases}
\end{equation}
where we have set
\[
\tilde{f}= f\xi-F\cdot\D\xi-A\D u\cdot \D\xi,
\qquad
\tilde{F}= F\xi-uA\D\xi\,.
\]

\smallskip

\noindent\emph{Step 2. Construction of the approximating solutions $u_{\e}$ in perforated domains.}

\smallskip

Let us consider the family of weak solution $\{u_\e\}_{0<\e\ll 1}$ to
\begin{equation}\label{eq:approximation:u:e}
\begin{cases}
-\dive(|y|^a A \nabla u_\e)= |y|^a \tilde{f}+\dive(|y|^a \tilde{F}), & \text{in }B_R\setminus \Sigma_{\e}^A\\
u_\e=0 & \text{on }\partial B_R \setminus \Sigma_\e^A\\
(A\D u_\e+\tilde{F})\cdot\nu=0 & \text{on } \partial\Sigma_\e^A\cap B_R.
\end{cases}
\end{equation}
By testing \eqref{eq:approximation:u:e} with $u_\e$, using \eqref{eq:unif:ell}, the Poincar\'e inequality \eqref{eq:poincare:eps:degenerate}, the Sobolev type embedding in Lemma \eqref{L:Sstab}, H\"older and Young inequalities, we get 
\begin{align*}
    &\lambda\int_{B_R\setminus\Sigma_\e^A}|y|^a|\D u_\e|^2\, dz \le \int_{B_R\setminus\Sigma_\e^A}|y|^a A\D u_\e\cdot\D u_\e\, dz\, 
    \\
    &= \int_{B_R\setminus\Sigma_\e^A} |y|^a \Big( f\xi u_\e
    -F\cdot \D\xi u_\e -A\D u \cdot \D \xi u_\e - F\xi \cdot \D u_\e - u A \D \xi \cdot \D u_\e
    \Big)\, dz\\
     &\le \Big(\int_{B_R\setminus\Sigma_\e^A}|y|^a |{f}|^p\, dz\Big)^{\frac1p}\Big(\int_{B_R\setminus\Sigma_\e^A}|y|^a |u_\e|^{p'}\, dz\Big)^{\frac{1}{p'}} \\
     &+
     c
     \Big(\int_{B_R\setminus\Sigma_\e^A}|y|^a |F|^q\, dz\Big)^{\frac1q}\Big(\int_{B_R\setminus\Sigma_\e^A}|y|^a (|u_\e|^{2}+|\D u_\e|^{2})\, dz\Big)^{\frac{1}{2}}
     \\
     &+ c
\Big(\int_{B_R\setminus\Sigma_\e^A}|y|^a (u^2+|\D u|^2)\, dz\Big)^{\frac12}\Big(\int_{B_R\setminus\Sigma_\e^A}|y|^a (|u_\e|^{2}+|\D u_\e|^{2})\, dz\Big)^{\frac{1}{2}}     
     \\
    &\le \frac{\lambda}{2} \int_{B_R\setminus\Sigma_\e^A}|y|^a|\D u_\e|^2\, dz+c
\big(\|u\|_{H^{1,a}(B_R)}+\|f\|_{L^{p,a}(B_R)}
+\|F\|_{L^{q,a}(B_R)}\big)^2\,,
\end{align*}
for some $c>0$ depending only on $d$, $n$, $a$, $\lambda$, $\Lambda$ and $R$. Hence, combining this inequality and the Poincar\'e inequality \eqref{eq:poincare:eps:degenerate}, we have
\begin{equation}\label{eq:stima:u:e:2}
    \int_{B_R\setminus\Sigma_\e^A}|y|^a(|u_\e|^2+|\D u_\e|^2)\, dz\le c
\big(\|u\|_{H^{1,a}(B_R)}+\|f\|_{L^{p,a}(B_R)}
+\|F\|_{L^{q,a}(B_R)}\big)^2\,.
\end{equation}

\smallskip

\noindent\emph{Step 3. Strong convergence of the sequence $ u_{\e_k} $ in $H^1_{\rm loc}(B_R\setminus \Sigma_0)$. }

\smallskip

Let us fix two compact sets $ K \subset K'\subset {B_R} \setminus \Sigma_0 $. Then, for every $\eps$ small enough, we have $ K' \subset {B_R} \setminus \Sigma_\eps^A $. Since $ |y|^a $ is uniformly bounded and bounded away from zero on $ K' $, by \eqref{eq:stima:u:e:2} we have that $\|u_\eps\|_{H^{1}(K')} \leq c$ for some constant $ c > 0 $ that does not depend on $\eps$. This implies that there exists a sequence $\eps_k \to 0^+$ and $\bar u \in H^1(K')$ such that
\begin{equation}\label{eq:weak:convergence}
    u_{\e_k}\rightharpoonup \bar{u} \ \text {weakly in } H^{1}(K'), \qquad u_{\e_k}\to \bar{u} \ \text{ strongly in } L^2(K').
\end{equation}
By testing the equation \eqref{eq:approximation:u:e} with $\phi^2(u_\e-\bar{u})$, where $\phi\in C_c^\infty(K')$, we obtain
\begin{align*}
\int_{K'}|y|^a \phi^2 A\nabla u_\e\cdot\nabla u_\e\, dz&=\int_{K'}|y|^a\Big( \phi^2  A\nabla u_\e\cdot\nabla \bar{u} -2(u_\e-\bar{u})\phi A\nabla u_\e\cdot\nabla \phi+\tilde{f}\phi^2(u_\e-\bar{u})\\
&\qquad-\phi^2 \tilde{F}\cdot \nabla(u_\e-\bar{u})-2\phi (u_\e-\bar{u})\tilde{F}\cdot\nabla\phi
\Big)\, dz.
\end{align*}
By \eqref{eq:weak:convergence} we can take the limit as $\e_k\to 0$ in the right hand side and obtain
\begin{equation}\label{eq:robba}
\lim_{\e_k\to 0} \int_{K'}|y|^a \phi^2 A\nabla u_{\e_k}\cdot\nabla u_{\e_k}\, dz=\int_{K'} |y|^a \phi^2 A\nabla \bar{u}\cdot\nabla \bar{u}\, dz\,.
\end{equation}

Fix $\delta >0$. Take $\phi \in C_c^\infty(K')$ such that $0\leq \phi\leq 1$, $\phi=1$ in $K$ and $|\supp(\phi)\setminus K| \ll 1$ so that, noticing that $|y|^a A \D \bar{u} \cdot \D \bar{u} \in L^1(B_R)$, we have $\||y|^a A \D \bar{u} \cdot \D \bar{u} \|_{L^1(\supp(\phi)\setminus K)}< \delta$.  Then, 
\begin{align*}
    c\int_{K}&|\D u_{\e_k}- \D \bar{u}|^2\,dz \le \int_{K} |y|^a A (\D u_{\e_k} - \D \bar{u}) \cdot (\D u_{\e_k} - \D \bar{u})\,dz \\ 
    &= \int_{K} |y|^a A \D u_{\e_k} \cdot \D u_{\e_k}+ |y|^a A \D \bar{u}\cdot\D \bar{u} - 2 |y|^a A  \D \bar{u} \cdot \D u_{\e_k}\,dz \\
    &\le \int_{K'} |y|^a A \D u_{\e_k} \cdot \D u_{\e_k}\, \phi^2 + |y|^a A \D \bar{u}\cdot \D \bar{u}\, \phi^2\,dz - 2 \int_{K} |y|^a A  \D \bar{u} \cdot \D u_{\e_k}\,dz \\
    &\leq \Big|\int_{K'} |y|^a A \D u_{\e_k} \cdot \D u_{\e_k}\, \phi^2\,dz-\int_{K'} |y|^a A \D \bar u \cdot \D \bar u\, \phi^2\,dz\Big| + 2\int_{K'\setminus K} |y|^a A \D \bar{u} \cdot \D \bar{u}\, \phi^2\,dz\\
    &\leq\Big|\int_{K'} |y|^a A \D u_{\e_k} \cdot \D u_{\e_k}\, \phi^2\,dz-\int_{K'} |y|^a A \D \bar u \cdot \D \bar u\, \phi^2\,dz\Big| + 2\delta \to  2\delta \,, \qquad \text{ as } \eps_k \to 0\,, 
\end{align*} 
thanks to \eqref{eq:robba}.
By to the arbitrariness of $\delta$, we infer that $ \nabla  u_{\e_k}  \to  \nabla \bar{u} $  in $L^2(K)^d$, and thus $u_{\e_k}\to \bar{u}$ strongly in $H^1(K)$. Finally, a standard diagonal argument yields that 
\begin{equation}\label{eq:strong:diagonal}
u_{\e_k}\to \bar{u}\quad \text{strongly in } H^1_{\rm loc}(B_R\setminus\Sigma_0).
\end{equation}

\smallskip

\noindent\emph{Step 4. The limit function $\bar{u}\in  H^{1,a}(B_R)$.}

\smallskip

Let us now prove that $\bar{u}\in  H^{1,a}(B_R)$. Let us fix $\phi\in C_c^\infty(B_R)$ if $a+n\in(0,2)$ or $\phi\in C_c^\infty(B_R\setminus\Sigma_0)$ if $a+n\geq 2$. For every $\e_k$ fixed, one has that $u_{\e_k}\in W^{1,1}_{\rm loc}(B_R\setminus\Sigma_{\e_k}^A)$, so it holds
\begin{equation}\label{eq:u:e:W11}   \int_{B_R\setminus\Sigma_{\e_k}^A}\partial_{z_\ell}u_{\e_k}\phi\, dz=- \int_{B_R\setminus\Sigma_{\e_k}^A}u_{\e_k}\partial_{z_\ell}\phi\, dz+\int_{\partial\Sigma_{\e_k}^A\cap B_R}u_{\e_k}\phi\, d\sigma\,,
\end{equation}
for every $\ell=1,\dots,d$.

If $a+n\geq2$, $\supp(\phi)\subset B_R\setminus\Sigma_{\e_k}^A$ for every $\e_k$ small enough, so
\[
\int_{\partial\Sigma_{\e_k}^A\cap B_R}u_{\e_k}\phi\, d\sigma =0\,.
\]
On the other hand, let us fix a measurable set $E\subset B_R$. By using \eqref{eq:stima:u:e:2} and observing that $|y|^{-a}\in L^\infty(\supp(\phi))$, we get
\begin{align*}
   \Big|\int_{E} \partial_{z_\ell}u_{\e_k}\phi\, dz\Big|&\le \Big(\int_{E\cap\supp(\phi)}|y|^a|\partial_{z_\ell}u_{\e_k}|^2\, dz\Big)^{\frac12}\Big(\int_{E\cap\supp(\phi)}|y|^{-a}|\phi|^2\, dz\Big)^{\frac12} \\ 
   & \leq C_\phi
\big(\|u\|_{H^{1,a}(B_R)}+\|f\|_{L^{p,a}(B_R)}
+\|F\|_{L^{q,a}(B_R)}\big)|E|^{\frac12}\,,
\end{align*}
where the constant $C_\phi$ depends on $\|\phi\|_{L^\infty}$ and $\supp(\phi)$, which implies that $\partial_{z_\ell}u_{\e_k}\phi$ is uniformly integrable. By the a.e. convergences $\D u_{\e_{k}}\to \D\bar{u}$ and $\chi_{B_R\setminus\Sigma_{\e_k}^A}\to \chi_{B_R}$, the Vitali convergence theorem yields
\[
\int_{B_R\setminus\Sigma_{\e_k}^A}\partial_{z_\ell}u_{\e_k}\phi\, dz\, \to \int_{B_R} \partial_{z_\ell}\bar{u}\phi\, dz\, .
\]
By employing a similar argument, one has that 
\[
\int_{B_R\setminus\Sigma_{\e_k}^A}u_{\e_k}\partial_{z_\ell}\phi\, dz\to \int_{B_R}\bar{u}\partial_{z_\ell}\phi\, dz\,,
\]
so, by taking the limit as $\e_k\to0^+$ in \eqref{eq:u:e:W11} we have
\begin{equation}\label{eq:u:e:W11:2}
\int_{B_R}\partial_{z_\ell}\bar{u}\phi\, dz=\int_{B_R}\bar{u}\partial_{z_\ell}\phi\, dz\,,
\end{equation}
which means that $\bar{u}\in W^{1,1}_{\rm loc}(B_R\setminus\Sigma_0)$.

If $a+n\in(0,2)$, by using Lemma \ref{L:trace:e} combined with \eqref{eq:stima:u:e:2}, we obtain
\begin{align*}
\Big|\int_{\partial \Sigma_{\e_k}^A\cap B_R}u_{\e_k}\phi\, d\sigma \Big|&\le 
\Big(
\int_{\partial \Sigma_{\e_k}^A\cap B_R}|y|^a 
|u_{\e_k}|^2\, d\sigma
\Big)^{\frac12}
\Big(
\int_{\partial \Sigma_{\e_k}^A\cap B_R} |y|^{-a}
|\phi|^2\, d\sigma
\Big)^{\frac12}\\
&\le c\, {\e_k}^{n-1} \Big(
\int_{B_R\setminus \Sigma_{\e_k}^A} 
|y|^a |\D u_{\e_k}|^2 \, dz
\Big)^{\frac12}
\Big(
\int_{B_R\setminus \Sigma_{\e_k}^A} 
|y|^{-a} |\D \phi|^2\, dz
\Big)^{\frac12}\\
&\le c\, {\e_k}^{n-1} 
\big(\|u\|_{H^{1,a}(B_R)}+\|f\|_{L^{p,a}(B_R)}
+\|F\|_{L^{q,a}(B_R)}\big)\,, 
\end{align*}
and taking the limit $\e_k\to0^+$ we get that the term  $\int_{\partial \Sigma_{\e_k}^A\cap B_R}u_{\e_k}\phi\, d\sigma$ vanishes.

Using the same argument as for the case $a + n \geq 2$ and noting that $|y|^{-a} \in L^1(B_R)$, we can take the limit as $\eps_k \to 0^+$ in \eqref{eq:u:e:W11}. This allows us to conclude that \eqref{eq:u:e:W11:2} holds true, which means that $\bar{u} \in W^{1,1}_{\rm loc}(B_R)$.

Now, using the Fatou Lemma and \eqref{eq:stima:u:e:2}, we have
\begin{align*}
    \int_{B_R}|y|^a&(|\bar{u}|^2+|\D \bar{u}|^2)\, dz=\int_{B_R}\lim_{\e_k\to0^+}|y|^a(|u_{\e_k}|^2+|\D u_{\e_k}|^2) \rchi_{B_R\setminus\Sigma_{\e_k}^A}\, dz\\
    &\le \liminf_{\e_k\to0^+} \int_{B_R\setminus\Sigma_{\e_k}^A} |y|^a(|u_{\e_k}|^2+|\D u_{\e_k}|^2)\, dz\le c
\big(\|u\|_{H^{1,a}(B_R)}+\|f\|_{L^{p,a}(B_R)}
+\|F\|_{L^{q,a}(B_R)}\big)^2\,,
\end{align*}
which means that $\|\bar u \|_{H^{1,a}(B_R)}< \infty$. Thus, by Lemma \ref{L:Wchar} and Lemma \ref{L:H=W}, we conclude that $\bar{u}\in H^{1,a}(B_R)$.

\smallskip

\noindent\emph{Step 5. $\bar{u}$ has zero trace on $\partial B_R$.}

\smallskip

Finally, we prove that $\bar{u}$ has zero trace on $\partial B_R$, meaning that $T_R \bar u = 0$. First, notice that for any fixed $\delta \ll 1$, it follows from \eqref{eq:strong:diagonal} that $u_{\e_k} \to \bar u$ strongly in $H^1(B_R \setminus \Sigma_\delta^A)$. Consequently, we also have $T_R u_{\e_k} \to T_R \bar u$ strongly in $L^{2}(\partial B_R \setminus \Sigma_\delta^A)$.
Since $T_R u_{\e_k} = 0$ by construction, we immediately get that $T_R \bar u = 0 $ $L^2$-a.e. in $\partial B_R \setminus \Sigma_0$. 

Now, let us fix $\nu >0$. Since, by Lemma \ref{L:H^1_0:trace}, it holds $|y|^a|T_R \bar u|^2  \in L^{1}(\partial B_R)$, the absolute continuity of the integral ensures that there exists a sufficiently small $\mu$ such that 
\[
\int_{\partial B_R \cap \Sigma_\mu^A}|y|^a |T_R \bar u|^2 ds \leq \nu\,. 
\]
Thus
\[
\int_{\partial B_R}|y|^a |T_R \bar u|^2 ds \leq \nu + \int_{\partial B_R\setminus\Sigma_\mu^A}|y|^a |T_R \bar u|^2 ds \leq\nu + c\int_{\partial B_R\setminus\Sigma_\mu^A}|T_R \bar u|^2 ds = \nu\,,
\]
where in the last equality holds because $T_R \bar u = 0 $ $L^2$-a.e. on $\partial B_R \setminus \Sigma_0$. Since $\nu$ is arbitrary, we conclude that $T_R \bar u  = 0$. Therefore, again by Lemma \ref{L:H^1_0:trace}, we have $u \in H^{1,a}_0(B_R)$.

\smallskip

\noindent\emph{Step 6. $\bar{u}=\tilde{u}$ and conclusion.}

\smallskip 

Next, we prove that $\bar{u}=\tilde{u}$. Let us test \eqref{eq:approximation:u:e} with $\phi \in C_c^\infty(B_R)$. Then, we have
\begin{equation}\label{eq:conv:bar:u:1}
    \int_{B_R\setminus\Sigma_{\e_k}^A}|y|^a A\D u_{\e_k}\cdot \D \phi\, dz \to 
\int_{B_R}|y|^a A\D \bar{u}\cdot \D \phi\, dz\,, \quad \text{ as } \e_k\to0^+.
\end{equation} In fact, let us fix $E\subset B_R$ measurable. By using \eqref{eq:unif:ell}, \eqref{eq:stima:u:e:2} and $|y|^a \in L^1(B_R)$, we have that 
\begin{align*}
   \int_{E\setminus\Sigma_{\e_k}^A}|y|^a |A\D u_{\e_k}\cdot \D \phi|\, dz
   &=
\int_{E}|y|^a | A\D u_{\e_k}\cdot \D \phi| \rchi_{E\setminus\Sigma_{\e_k}^A}\, dz
\\
    &\le c \|u_{\e_k}\|_{H^{1,a}(B_R\setminus\Sigma_{\e_k}^A)}\|\D\phi\|_{L^\infty(B_R)}\int_{E}|y|^a\, dz \le \delta(|E|)\,,
\end{align*}
where $\delta(|E|)\to 0$ as $|E|\to 0$. Combining this with the a.e. convergence $\D u_{\e_k}\to \D \bar{u}$ we can conclude that \eqref{eq:conv:bar:u:1} holds true by applying the Vitali convergence theorem.
With similar computations, we get
\begin{equation}\label{eq:conv:bar:u:2}
\int_{B_R\setminus\Sigma_{\e_k}^A}|y|^a (\tilde{f}\phi-\tilde{F}\cdot\D\phi)\, dz\to \int_{B_R}|y|^a (\tilde{f}\phi-\tilde{F}\cdot\D\phi)\, dz\,,\quad \text{as } \e_k \to 0^+ .
\end{equation}
Combining \eqref{eq:conv:bar:u:1} and \eqref{eq:conv:bar:u:2} we get
\[
\int_{B_R}|y|^a A\D\bar{u}\cdot\D\phi\, dz=\int_{B_R}|y|^a (\tilde{f}\phi-\tilde{F}\cdot\D\phi)\, dz\,,
\]
and since $\bar{u}\in {H}^{1,a}_0(B_R)$, we get that $\bar{u}$ is a weak solution to \eqref{eq:cut:off:approximation}. By uniqueness of solutions (see Remark \ref{R:existence:uniqueness}), we  get that $\bar{u}=\tilde{u}$ a.e. in $B_R$. Since  $\tilde u = u$, $\tilde{f}=f$ and $\tilde{F}=F$ in $B_r$, and by using \eqref{eq:stima:u:e:2},
\eqref{eq:strong:diagonal} we obtain that our statement holds true. \end{proof}

The next approximation lemma provides the convergence of solutions of equations with regularized coefficients through convolution with standard mollifiers.

\begin{Lemma}\label{L:approximation:smooth}
    Let $a+n>0$, $R>0$, $r\in(0,R)$, and $A \in C^0(B_{R})$ satisfying \eqref{eq:unif:ell}. 
    Let $f\in L^{p,a}(B_R)$ where $p\ge (2_*^a)'$ if $d+a_+>2$ or $p> 1$ if $d+a_+=2$, $F\in L^{q,a}(B_R)^d$ where $q\ge 2$, let $u$ be a weak solution to \eqref{eq:weak:sol:conormal:0}. For $\delta>0$, let $\{\rho_\delta\}$ be a family of smooth mollifiers and let us define $A_\delta:=A\ast \rho_\delta$. Then, there exists a family $\{u_\delta\}_{0<\delta\ll 1}$, such that $u_\delta$ are weak solutions to
\[
-\dive(|y|^a A_\delta \nabla u_\delta)= |y|^a f+\dive(|y|^aF), \quad \text{in }B_{r},
\]
\[
\|u_\delta\|_{H^{1,a}(B_{r})}\le c \big(\|f\|_{L^{p,a}(B_R)}
+\|F\|_{L^{q,a}(B_R)}+\|u\|_{H^{1,a}(B_R)}\big)
\,,
\]
for some constant $c>0$ depending only on $d$, $n,$ $ a,$ $ \lambda,$ $ \Lambda$, $r$, $R$, and there exists a sequence $\delta_k\to0$ such that
\[
 u_{\delta_k}\to u \text{ in } H^{1,a}(B_{r}).
\]
Moreover, in the case $d+a_+=2$ the constant $c$ depends also on $p$.
\end{Lemma}
\begin{proof}
We consider the case $d + a_+ > 2$, as the other one can be treated analogously.
Let us fix $0<r<R'<R$ and consider a smooth cut-off function $\xi\in C_c^\infty(B_R)$ such that
\[
\xi = 1 \text{ in }B_{r},\quad 0\le\xi\le1,\quad \supp(\phi)\subset B_{R'}.
\]
Arguing as in Lemma \ref{L:approximation}, the function $\tilde u:= \xi u$ is a weak solution to
\begin{equation}\label{eq:u:smooth:cut:off}
\begin{cases}
-\dive(|y|^a A \D \tilde u)=|y|^a \tilde{f}+ \dive(|y|^a \tilde{F}), & \text{ in }B_{R'}\\
\tilde u=0, & \text{ on }\partial {B_{R'}},
\end{cases}
\end{equation}
where, we set
\[
\tilde{f}= f\xi-F\cdot\D\xi-A\D u\cdot \D\xi,
\qquad
\tilde{F}= F\xi-uA\D\xi\,.
\]

For $\delta>0$, let $\{\rho_\delta\}$ be a family of smooth mollifiers and let us define $A_\delta:=A\ast \rho_\delta$, which is a uniformly elliptic matrix in ${B_{R'}}$ (choosing $\delta$ small enough).
For every $0<\delta\ll1$, recalling Remark \ref{R:existence:uniqueness}, let $u_\delta$ be the unique weak solution to 
\begin{equation}\label{eq:u_delta:smooth}
\begin{cases}
-\dive(|y|^a A_\delta \D u_\delta)=|y|^a \tilde{f}+ \dive(|y|^a \tilde{F}), & \text{ in }{B_{R'}}\\
u_\delta=0, & \text{ on }\partial {B_{R'}}.
\end{cases}
\end{equation}
By testing \eqref{eq:u_delta:smooth} with $u_\delta$ and arguing as in Lemma \ref{L:approximation}, we obtain
\[
\|u_\delta\|_{H^{1,a}({B_{R'}})}\le c\big(\|u\|_{H^{1,a}({B_{R}})}+\|{f}\|_{L^{p,a}({B_{R}})}+\|{F}\|_{L^{q,a}({B_{R}})}\big),
\]
for some $c>0$ depending only on $d$, $\lambda$, $\Lambda$ and $R$. Hence, $\{u_\delta\}$ is uniformly bounded in $H^{1,a}_0({B_{R'}})$ and by using Lemma \ref{L:compactness}, we get
\begin{equation}\label{eq:v_delta:convergence:smooth}
u_\delta \to v \text{ strongly in } L^{2,a}({B_{R'}})\,,\qquad
\D u_\delta \rightharpoonup \D v \text{ weakly in } L^{2,a}({B_{R'}})^d\,,
\end{equation}
for some $v\in H^{1,a}_0(B_R)$.
Taking $u_\delta-v$ as test function in \eqref{eq:u_delta:smooth}, using \eqref{eq:v_delta:convergence:smooth} and recalling that $A_\delta \to A$ uniformly, we get
\[
\int_{{B_{R'}}}|y|^a A_\delta \D u_\delta \cdot \D u_\delta = \int_{{B_{R'}}}|y|^a  \big(A_\delta \D u_\delta \cdot \D v + 
\tilde f (u_\delta-v)
- \tilde F\cdot \D (u_\delta-v)
\big)\to \int_{{B_{R'}}} |y|^a A \D v \cdot \D v.
\]
Hence, since $A_\delta$ satisfies \eqref{eq:unif:ell}, we have that $\D u_\delta\to \D v$ strongly in $L^{2,a}({B_{R'}})^d$.

Finally, let us fix $\phi \in C_c^\infty({B_{R'}})$. Then, 
\[
\int_{{B_{R'}}} |y|^a \big(
\tilde f\phi
- \tilde F\cdot \D \phi
\big)=
\int_{{B_{R'}}}|y|^a A_\delta \D u_\delta\cdot \D\phi \to 
\int_{{B_{R'}}}|y|^a A \D v \cdot \D\phi \,,
\]
that is, $v$ is a weak solution to \eqref{eq:u:smooth:cut:off} and by uniqueness of solutions (see Remark \ref{R:existence:uniqueness}), it follows that $ v=\tilde u$. Next, since $\tilde{f}=f$, $\tilde{F}=F$ and $\tilde{u}=u$ in $B_{r}$, we have proved that $u_\delta\to u$ in $H^{1,a}(B_{r})$ and $u_\delta$ solves 
\[
-\dive(|y|^a A_\delta \D u_\delta)=|y|^a {f}+ \dive(|y|^a {F}) \ \text{ in }B_{r}.
\]
The proof is complete.
\end{proof}

\section{Liouville theorems}\label{sec:5}

The aim of this section is to prove the Liouville Theorem \ref{L:Liouville} for entire solutions to \eqref{eq:Lioueq}, see Definition \ref{def:weak:sol:hom:eps}. First, we need some preliminary results. The first result addresses the unweighted tangential variables $ x \in \mathbb{R}^{d-n} $, for which the operator is invariant under translation. Its proof uses a standard difference quotients technique and involves an iterative argument based on the Caccioppoli inequality in Lemma \ref{lemma:caccioppoli} (see for example \cite[Corollary 4.2, Lemma 4.3]{TerTorVit24a}), and is therefore omitted here.

\begin{Proposition}\label{P:diffquot}
Let $u$ be an entire solution to \eqref{eq:Lioueq}. Then
\begin{itemize}
\item[$i)$] for every $j= 1, \ldots, d-n$, the weak derivative ${\partial_{x_j} u}$ is also an entire solution to  \eqref{eq:Lioueq};
\item[$ii)$] if there exists constants $c, \gamma >0$ such that
\[
|u(z)| \leq c(1+|z|^\gamma)\,,
\]
then $u$ is a polynomial of degree at most $\lfloor\gamma\rfloor$ in the variable $x \in \R^{d-n}$, with coefficients depending only on the variable $y \in \R^n$.
\end{itemize}
\end{Proposition}

In the following crucial lemmas we provide a characterization of the solutions to \eqref{eq:Lioueq} with polynomial growth at infinity in the case $ n = d $.
The first lemma gives us a basis of $L^2(\mathbb{S}^{n-1})$ which depends on the matrix $A$. 

\begin{Lemma}\label{L:spectral}
Let $A \in \R^{n,n}$ be a constant matrix which satisfies \eqref{eq:unif:ell}. The following holds:
\begin{enumerate}
\item[$i)$] there exists an increasing, diverging sequence of eigenvalues $\{\mu_k(A)\}_{k\geq 0}$, corresponding to critical levels of the quotient 
\[
\mathcal{R}(g)=\frac{\int_{\mathbb{S}^{n-1}}|A^{\frac12}\sigma|^a |\nabla_\sigma g|^2 d\sigma }{\int_{\mathbb{S}^{n-1}}|A^{\frac12}\sigma|^a | g|^2 d\sigma }\,.
\]
Each eigenvalue has finite multiplicity, denoted by $m(\mu_k(A))$;
\item[$ii)$] let $g_{k,j}$, for $j = 1, \ldots, m(\mu_k(A))$, be the normalized eigenfunctions associated to $\mu_k(A)$. The eigenfunctions $g_{k,j}$ satisfy the following properties:
\begin{equation}\label{eq:Leigprop}
\begin{gathered}
\int_{\mathbb{S}^{n-1}}|A^{\frac12}\sigma|^a\nabla_\sigma g_{k,j}\cdot\nabla_\sigma \eta\, d \sigma = \mu_k(A) \int_{\mathbb{S}^{n-1}}|A^{\frac12}\sigma|^a g_{k,j}\eta\, d \sigma\,, \qquad \eta \in H^1(\mathbb{S}^{n-1})\,,\\
\int_{\mathbb{S}^{n-1}}|A^{\frac12}\sigma|^a g_{k,j}g_{k', j'}\, d\sigma \, = \delta_{kk'}\delta_{jj'}\,,
\end{gathered}
\end{equation}
where $\delta_{ij}$ is the Kronecker delta. 

Moreover, ${g_{k,j}}$ provides a basis of $L^2(\mathbb{S}^{n-1})$, which is also orthonormal with respect to the scalar product $\langle u, v \rangle_{L^2(\mathbb{S}^{n-1})} := \int_{\mathbb{S}^{n-1}}|A^{\frac12}\sigma|^a\,u\,v\, d\sigma$. Thus, for every $u\in L^2(\mathbb{S}^{n-1})$ it holds
\[
u(\sigma) = \sum_{k=0}^\infty\sum_{j=1}^{m(\mu_k(A))}\Big(\int_{\mathbb{S}^{n-1}}|A^{\frac12}\sigma|^a u(\eta)g_{k,j}(\eta)d\eta\Big)g_{k,j}(\sigma)\,; 
\]
\item[$iii)$] the eigenvalue $\mu_0(A) = 0$ is simple and $g_0 = g_{0,1}$ is constant;
\item[$iv)$] the first nontrivial eigenvalue $\mu_1(A)$ is such that
\[\mu_1(A) \geq \begin{cases}
\displaystyle \Big(\frac{\lambda}{\Lambda}\Big)^{\frac{|a|}{2}}(n-1) & \text{ if }n \geq 3\,,\\
 \displaystyle\Big({\frac{4}{\pi}}\arctan 
 \Big(\frac{\lambda}{\Lambda}\Big)^{\frac{|a|}{4}} 
 \Big)^2\, & \text{ if }n =2\,. 
\end{cases}\]
\end{enumerate}
\end{Lemma}
\begin{proof}
By condition \eqref{eq:unif:ell}, we have that 
\begin{equation}\label{eq:sigmaell}
\lambda \leq |A^{\frac12}\sigma|^2  \leq \Lambda\,, \quad \text{ for every } \sigma \in \mathbb{S}^{n-1}\,.
\end{equation}
Thus
\[
\int_{\mathbb{S}^{n-1}}|A^{\frac12}\sigma|^a(|\nabla_\sigma g|^2 + | g|^2) d\sigma 
\]
provides a norm equivalent to the standard one in $H^1(\mathbb{S}^{n-1})$. 
Keeping this in mind, we obtain $i)$, $ii)$ by a standard application of Hilbert-Schmidt theorem. 

Since $iii)$ is trivial, it remains to prove $iv)$. Given $g \in H^1(\mathbb{S}^{n-1})$, let us denote 
\[
\langle g\rangle_A =  \big(\int_{\mathbb{S}^{n-1}}|A^{\frac12}\sigma|^a d\sigma\big)^{-1}\int_{\mathbb{S}^{n-1}}|A^{\frac12}\sigma|^a g\,d\sigma\,,
\]
and recall that
\begin{equation}\label{eq:meanmin}
\inf_{\xi \in \R} \int_{\mathbb{S}^{n-1}}|A^{\frac12}\sigma|^a | g - \xi|^2 d\sigma = \int_{\mathbb{S}^{n-1}}|A^{\frac12}\sigma|^a | g - \langle g\rangle_A|^2 d\sigma\,.
\end{equation}
We have that 
\begin{equation}\label{eq:LBvar}
\mu_1(A) = \inf_{g \in H^1(\mathbb{S}^{n-1})\setminus\{0\}, \langle g \rangle_A = 0}\frac{\int_{\mathbb{S}^{n-1}} |A^{\frac12}\sigma|^a|\nabla_\sigma g|^2 d\sigma }{\int_{\mathbb{S}^{n-1}} |A^{\frac12}\sigma|^a| g|^2 d\sigma }\,.
\end{equation}
Notice that the eigenvalues $\mu_k({\mathbb{I}})$ of the Laplace-Beltrami operator on the sphere are well known, and in particular $\mu_1({\mathbb{I}}) = n-1$. 
Now, set
\[
m:= \min\Big\{\Big(\frac{\lambda}{\Lambda}\Big)^{\frac{a}{2}}, \Big(\frac{\Lambda}{\lambda}\Big)^{\frac{a}{2}}\Big\} = \Big(\frac{\lambda}{\Lambda}\Big)^{\frac{|a|}{2}}
\]
and let $\psi \in H^1(\mathbb{S}^{n-1})$ be an eigenfunction associated to $\mu_1(A)$. Notice that $\langle \psi\rangle_A = 0$. Thus, by \eqref{eq:meanmin} we have that
\[
\int_{\mathbb{S}^{n-1}}|A^{\frac12}\sigma|^a | \psi |^2 d\sigma \leq \int_{\mathbb{S}^{n-1}}|A^{\frac12}\sigma|^a| \psi - \langle \psi\rangle_\mathbb{I}|^2 d\sigma
\]
As a consequence, thanks to \eqref{eq:sigmaell} we have
\[
\begin{aligned}
\mu_1(A) = \frac{\int_{\mathbb{S}^{n-1}}|A^{\frac12}\sigma|^a|\nabla_\sigma \psi|^2 d\sigma }{\int_{\mathbb{S}^{n-1}}|A^{\frac12}\sigma|^a | \psi|^2 d\sigma } &\geq \frac{\int_{\mathbb{S}^{n-1}} |A^{\frac12}\sigma|^a|\nabla_\sigma (\psi- \langle\psi\rangle_\mathbb{I})|^2 d\sigma }{\int_{\mathbb{S}^{n-1}}|A^{\frac12}\sigma|^a| \psi-\langle\psi\rangle_\mathbb{I}|^2 d\sigma }\\
&\geq m\frac{\int_{\mathbb{S}^{n-1}} |\nabla_\sigma (\psi - \langle\psi\rangle_\mathbb{I})|^2 d\sigma }{\int_{\mathbb{S}^{n-1}}| \psi -\langle\psi\rangle_\mathbb{I} |^2 d\sigma } \geq  m(n-1)\,,
\end{aligned}
\]
where in the last inequality we used that $\langle\psi - \langle\psi\rangle_\mathbb{I}\rangle_\mathbb{I}=0$ and \eqref{eq:LBvar} with $A = \mathbb{I}$. 

The improved estimate in the case $n=2$ is a consequence of \cite[Lemma 1]{PicSpa72}.
\end{proof}

\begin{Lemma}\label{L:Liouvillen}
Let $a+n>0$, $\eps \geq 0$, and let $A \in \R^{n,n}$ be a constant matrix which satisfies \eqref{eq:unif:ell}. Let $\mu_k(A)$, $g_{k, j}$ be as in Lemma \ref{L:spectral}, and define
\[
\gamma_k^\pm = \frac{2 -a-n \pm \sqrt{(a+n-2)^2 + 4 \mu_k(A)}}{2}\,.
\]
Let $u$ be such that $u \in H^{1,a}(B_R\setminus \Sigma^A_\eps)$ for any $R>0$ and  
\begin{equation*}
\int_{\R^n \setminus \Sigma^A_\eps}|y|^{a} A\nabla u \cdot \nabla \phi = 0\,, \quad \text{ for any }\quad \phi \in C^{\infty}_c(\R^n)\,.
\end{equation*}
Assume there exist constants $c, \gamma >0$ such that
\begin{equation}\label{eq:nLiobound}
|u(y)| \leq c(1 + |y|^\gamma)\,.
\end{equation}
Then there exist $C \in \R$, $\hat k = \hat k(\gamma) \in \N$ such that 
\begin{equation}\label{eq:Lepspos}
u(y) = C + \sum_{k=1}^{\hat k} \sum_{j=1}^{m(\mu_k^A)} (c_{k,j}^+ |A^{-\frac12}y|^{\gamma_k^+} + c_{k,j}^- |A^{-\frac12}y|^{\gamma_k^-})g_{k,j}\Big(\frac{A^{-\frac12}y}{|A^{-\frac12}y|}\Big)\,,
\end{equation}
for some constants $c_{k,j}^+, c_{k,j}^- \in \R$. In particular, if $\eps = 0$, then
\begin{equation}\label{eq:Lepszero}
u(y) = C + \sum_{k=1}^{\hat k} \sum_{j=1}^{m(\mu^A_k)} c_{k,j}^+ |A^{-\frac12}y|^{\gamma_k^+}g_{k,j}\Big(\frac{A^{-\frac12}y}{|A^{-\frac12}y|}\Big)\,.
\end{equation}
Moreover, if $\gamma < \gamma_1^+$, then $u$ is constant. 
\end{Lemma}

\begin{proof}
Since the proof we assume $d = n$, we will simply write $B_R$ instead of $ B_R^n$ to denote general balls in $\R^{n}$. Let us define $ v(\xi) = u(A^{\frac{1}{2}}\xi)$. Performing some standard computations using the change of variables $\xi = A^{-\frac12}y$ and the uniform ellipticity of $A$, we see that $v \in H^{1,a}(B_R\setminus \Sigma_\eps)$ for any $R>0$ and satisfies
\begin{equation}\label{eq:LiouVeq}
\int_{\R^n \setminus \Sigma_\eps}|A^\frac12 \xi|^{a} \nabla v \cdot \nabla \phi = 0\,, \quad \text{ for any }\quad \phi \in C^{\infty}_c(\R^n)\,.
\end{equation}
The remaining part of the proof is divided in three steps. 

\smallskip

\noindent \emph{Step 1: Spectral decomposition.}

\smallskip

In the following we extensively use polar coordinates $\xi = r \sigma$, where $r = |\xi| >0$ and $\sigma = |\xi|^{-1}\xi \in \mathbb{S}^{n-1}$. 
Since $v \in H^{1,a}(B_R \setminus \Sigma_\eps)$ for any $R>0$, then by the Fubini-Tonelli theorem it holds that $v(r\cdot) \in H^1(\mathbb S^{n-1})$ for a.e. $r > \eps$. Thus, for a.e. $r > \eps$, we can decompose $v(r\cdot)$ via the basis $\{g_{k,j}\}$ (see Lemma \ref{L:spectral}, $ii)$), obtaining
\[
v(r\sigma) = \sum_{k=0}^{\infty} \sum_{j=1}^{m(\mu_k^A)} v_{k,j}(r)g_{k,j}(\sigma)\,,
\]
where the coefficients $v_{k,j}$ depend on $r \in [\eps, \infty)$, and are such that
 \begin{equation}\label{eq:ukjdef}
 v_{k,j}(r) = \int_{\mathbb{S}^{n-1}}|A^{\frac12}\sigma|^a u(r\sigma)g_{k,j}d\sigma\,.
 \end{equation}
 
By previous discussions, the functions $v_{k,j}:[\eps, \infty)\to \R$ are well defined for a.e. $r \geq \eps$. Moreover, using once again that $v \in H^{1,a}(B_R \setminus \Sigma_\eps)$ for any $R>0$ we can see that 
 \begin{equation}\label{eq:ukjprop}
 v_{k,j} \in H^1((\eps, R), r^{a+n-1}dr) \quad \text{for any }R>0\,,\qquad
v'_{k,j}(r) = \int_{\mathbb{S}^{n-1}}|A^{\frac12}\sigma|^a \nabla v(r\sigma)\cdot \sigma g_{k,j}(\sigma)d\sigma\,.
 \end{equation}
 
 \smallskip

\noindent \emph{Step 2: Solutions of an associated ODE.}

 \smallskip

Let us rewrite \eqref{eq:LiouVeq} in polar coordinates as 
 \begin{equation}\label{eq:Lpolar}
\int_\eps^\infty \int_{\mathbb{S}^{n-1}}r^{a+n-1}|A^{\frac12}\sigma|^a\Big(\partial_r v \partial_r \phi + r^{-2}\nabla_\sigma v \cdot \nabla_\sigma \phi\Big)dr\, d\sigma\, = 0\,, \quad \text{ for any }\phi \in C^{\infty}_c(\R^n)\,.
\end{equation}
Now, let us take any $f \in C^\infty_c([0, \infty))$ such that 
\begin{equation}\label{eq:Ltextcond}
\mu_k(A) \int_{\eps}^\infty r^{a + n - 3}|f|^2 dr < \infty\,.
\end{equation}
Notice that \eqref{eq:Ltextcond} is always satisfied if $\eps >0$, or if $k = 0$ (by Lemma \ref{L:spectral}, $iii)$). Thanks to \eqref{eq:Leigprop} and \eqref{eq:Ltextcond} we have that $\phi(r\sigma) = f(r) g_{k,j}(\sigma) \in H^{1,a}(B_R \setminus \Sigma_\eps)$ for any $R>0$. Moreover, $\phi$ is compactly supported, and we can use it as a test function in \eqref{eq:Lpolar}.  Recalling also \eqref{eq:ukjdef} and \eqref{eq:ukjprop}, we obtain that $v_{k,j}$ satisfies 
\begin{equation}\label{eq:modeeq}
\int_\eps^\infty r^{a+n-1} (f' v_{k,j}' + r^{-2}\mu_k(A)  f v_{k,j})dr = 0\,, \quad \text{ for every } f \in C^\infty_c([\eps, \infty)) \text{ satisfying }\eqref{eq:Ltextcond}\,.
\end{equation}
Now, let us take any $\varphi \in C^\infty_c(\R)$, and test \eqref{eq:modeeq} with $f(r) = \varphi (\log(r))$. Notice that such $f$ is admissible. 
Performing the change of variables $r= e^\tau$, we find that the function $w_{k,j}(\tau) = v_{k,j}(e^\tau)$ satisfies
\[
\int_{\log\eps}^\infty e^{\tau(a+n-2)} (w'_{k,j}\varphi'+ \mu_k(A) w_{k,j}\varphi)\,d\tau = 0\,, \quad \text{ for all } \varphi \in C^\infty_c(\R)\,. 
\]
Equivalently, we have that $w_{k,j}$ is a solution to the elementary equation 
\[
w''_{k,j} + (a + n-2)w'_{k,j} - \mu_k(A)w_{k,j} = 0 \quad \text{ in }\quad [\log \eps, \infty)\,.
\]
Recalling that the multiplicity of the first eigenvalue $m(\mu_0(A))=1$, one has that 
\begin{equation}\label{fomr_v0}
v_{0,1}(r)=\begin{cases}
c_1+c_2 r^{2-a-n}, & \text{ if } 2-a-n\not=0,\\
c_1+c_2 \log r, & \text{ if } 2-a-n=0,
\end{cases}
\end{equation}
for some constants $c_1,c_2 \in \R$.
Furthermore, defining
\[
\gamma_k^\pm = \frac{2 - a-n \pm \sqrt{(a+n-2)^2 + 4 \mu_k(A)}}{2}\,,
\]
we readily get that $v_{k,j}$ is in the form 
\[
    v_{k,j}(r) = c_{k,j}^+ r^{\gamma_k^+} + c_{k,j}^- r^{\gamma_k^-},  \quad \text{ if }k\ge 1\,,
\]
for some constant $c_{k,j}^+, c_{k,j}^- \in \R$. 
We point out that 
\begin{equation}\label{eq:gammaprop}
\gamma_0^\pm = \pm (2-a-n)_\pm\,,\quad \text{ and }\quad \gamma^+_k >0\,, \quad \gamma^-_k < 0 \quad 
\text{if } k\geq 1\,.
\end{equation}

\smallskip

\noindent \emph{Step 3: Conclusion.}

\smallskip

First, let us show that $v_{0,1}$ is constant. 
Since $\mu_0(A)=0$, condition \eqref{eq:Ltextcond} is always satisfied. Thus, we can take $f\in C^\infty_c([\eps, \infty))$ such that $f(\eps) = 1$ in \eqref{eq:modeeq} and recalling \eqref{fomr_v0}, we find
\[
0 =\int_\eps^\infty r^{a+n-1} (f' v_{0,1}' + r^{-2}\mu_0(A)  f v_{0,1})\,dr = c_2(2 - a-n)\int_\eps^\infty f'\,dr = c_2(a+n-2)\,,
\]
if $2-a-n\not=0$, which implies that $c_2 =0$. Instead, if $2-a-n=0$, we find 
\[
0 =\int_\eps^\infty r^{a+n-1} (f' v_{0,1}' + r^{-2}\mu_0(A)  f v_{0,1})\,dr = c_2\int_\eps^\infty f'\,dr = -c_2\,,
\]
and we conclude as well.

Let us now focus on the case $k \geq 1$. First, let us assume $\eps =0$. By \eqref{eq:ukjprop} we have that $v'_{k,j} \in L^2((0, 1), r^{a+n-1} dr)$. As a consequence, 
\[
(c_{k,j}^-\gamma_k^-)^2 \int_0^1 r^{a + n -3 + 2 \gamma_k^-} dr < \infty\,.
\] 
Thanks to \eqref{eq:gammaprop}, this readily implies $c_{k,j}^- = 0$. Now, by condition \eqref{eq:nLiobound} and thanks to \eqref{eq:unif:ell} we have that 
\[
|v_{k,j}|\leq \int_{\mathbb S^{n-1}}|A^{\frac12}\sigma|^a|v(r\sigma)||g_{k,j}|d\sigma \leq c(1 + r^\gamma)\,.
\]
Therefore, since for every $\gamma$ there exists $\hat k$ such that $\gamma_{\hat k}^+ > \gamma$, we easily infer that $c_{k,j}^+ = 0$ for every $k \geq \hat k$, that is, 
$v_{k,j} \equiv 0$ for every $k \geq \hat k$. 
Recalling that $u(y) = v(A^{-\frac12}y)$, we readily get \eqref{eq:Lepszero}.

Let now consider the case $\eps >0$. Arguing as before, we find that condition \eqref{eq:nLiobound} implies that  there exists $\hat k$ such that for every $k \geq \hat k$ it holds $\gamma_{k}^+ > \gamma$ and 
\[
v_{k,j} = c_{k,j}^- r^{\gamma_k^-}.
\]
Let $k \geq \hat k$ be fixed. Since $\eps >0$, condition \eqref{eq:Ltextcond} is always satisfied. Thus we can take $f\in C^\infty_c([\eps, \infty))$ such that $f(\eps) = 1$ in \eqref{eq:modeeq} and we find, integrating by parts and performing some standard computations
\[
c_{k,j}^-\gamma_k^- \eps^{a + n - 2 + \gamma_k^-}\, = 0\,.
\]
Since $k\geq1$, then $\gamma_k^- \neq 0$ and we infer that $c_{k,j}^- = 0$, that is, $v_{k,j} \equiv 0$. Also in this case we use $u(y) = v(A^{-\frac12}y)$ to complete the proof of \eqref{eq:Lepspos}. 

Finally, we notice that if $\gamma < \gamma^+_1$ then in both cases it holds $v_{k,j} \equiv 0$ for every $k \geq 1$. Since we have already proved that $v_{0,1}$ is constant, then $v$ is also constant, and $u$ too. 
\end{proof}

\begin{proof}[Proof of Theorem \ref{L:Liouville}] Let $A$ be as in the assumptions, and recall the notation \eqref{Blocks}.
We first assume that $\gamma < 1$. Thus, by Proposition \ref{P:diffquot} we have that $u = u(y)$ does not depend on $x$, which implies that $u$ is an entire solution to \eqref{eq:Lioueq} with $d=n$ and $A=A_3$. Hence, $u$ satisfies the assumptions of Lemma \ref{L:Liouvillen}, and using that $\gamma < \gamma^+_1$, we conclude that $u$ must be constant.

\smallskip

Let now assume that $\gamma <2$. By Proposition \ref{P:diffquot} we have that $u$ is a polynomial of degree at most $1$ in the $x$ variable; that is, $u$ can be written in the form
\[
u(x, y) = u^0(y) + \sum_{j=1}^{d-n}u^j(y) x_j\,.
\]
First, we notice that by \eqref{eq:Lgrowth} it holds
\[
|u^0(y)| = |u(0, y)| \leq c(1 + |y|^\gamma)\,.
\]
Thus
\begin{equation}\label{eq:uigrowth}
|u^j(y)| = |u(e_{x_j}, y) - u^0(y)| \leq c(1 + |y|^\gamma)\,;
\end{equation}
that is, each component $u^0$, $u^j$ still satisfies the growth condition \eqref{eq:Lgrowth}. 

On the other hand, we have that $u^j(y) = {\partial_{x_j} u}$ for $j=1, \ldots, d-n$ and thus, again by Proposition \ref{P:diffquot}, they are entire solutions to \eqref{eq:Lioueq}. Keeping \eqref{eq:uigrowth} into account, and using that any $u^j$ does not depend on $x$, we can argue as in the previous part of the proof and conclude that they are constant. 

Summing up, we proved that there exists $\alpha = \alpha(u) \in \R^{d-n}$ constant such that 
\[
u(x, y) = u^0(y) + \alpha \cdot x\,.
\]
We claim that the function $u^* $ given by
\[
u^*(x, y) = \alpha \cdot x - (A_3^{-1}A_2^\top \alpha) \cdot y
\]
is an entire solution to \eqref{eq:Lioueq}.
Indeed, we can easily compute 
\[
A \nabla u^* = (A_1 \alpha - A_2 A_3^{-1}A_2^\top\alpha, 0),
\]
and thus for every $\phi \in C^\infty_c(\R^d)$ it holds
\[
\int_{\R^d \setminus \Sigma^A_\eps} |y|^a A \nabla u^*\cdot \nabla \phi \, dz = \int_{|A^{-\frac12}_3 y| \geq \eps} |y|^a (A_1 \alpha - A_2 A_3^{-1}A_2^\top\alpha) \cdot \Big(\int_{\R^{d-n}}\nabla_x \phi dx\Big)dy = 0\,,
\]
where in the last equality we used the divergence theorem.  

By linearity, the function $\hat u = u - u^*$ is an entire solution to \eqref{eq:Lioueq}. Moreover, since we are assuming $\gamma \in [1, 2)$, it holds
\[
|\hat u| \leq |u| + |u^*| \leq c(1 + |z|^\gamma) + c|z| \leq c(1 + |z|^\gamma)\,,
\]
that is, $\hat u$ satisfies \eqref{eq:Lgrowth}. Since by construction $\hat u=\hat u(y)$ does not depend on $x$, we can argue as in the $\gamma <1$ case and infer that $\hat u$ is constant, by Lemma \ref{L:Liouvillen} and using that $\gamma < \gamma^+_1$. In conclusion, we showed that 
\[
u = \hat u + u^* = c + \alpha\cdot x + \beta \cdot y\,,
\]
as needed. 
\end{proof}

\section{Stable regularity estimates in perforated domains}\label{sec:6}

In this section we prove Theorem \ref{T:0:1:alpha:eps}, namely, the stable regularity estimates in perforated domains for weak solutions to \eqref{approxPDEA}. 
For simplicity, we set this section in the unit ball $B_1$; however, the analysis should be carried out in a smaller ball $B_R$, where $R > 0$ is specified in Proposition \ref{prop:moser}.
We divide the proof into two parts, obtaining first the H\"older estimate for solutions and then the H\"older estimate for their gradient.

\subsection{Stable H\"older estimates}

\begin{proof}[Proof of Theorem \ref{T:0:1:alpha:eps}, i), stable $C^{0,\alpha}$-estimates] Let $u_\e$ be a family of solutions to \eqref{approxPDEA}.
Since the weight $|y|^a$ is uniformly elliptic away from $\Sigma_0$, and since the boundary $\partial\Sigma_\e^a$ is of class $C^1$ by the assumption $A\in C^1$ with $\|A\|_{C^{1,\omega}(B_1)}\leq L$ for some modulus of continuity $\omega$ (see Remark \ref{R:C1:boundary}), well-known results in elliptic regularity imply that for every $0<\eps \ll 1$, there exists a constant $C_\eps$ (depending only on $d$, $n$, $a$, $\lambda$, $\Lambda$, $p$, $q$, $\alpha$, $L$ and $\eps$) such that, for any solution $u_\eps$ to \eqref{approxPDEA}, the following estimate holds:
\begin{equation}\label{eq:alpha:RR}
    \|u_\e\|_{C^{0,\alpha}(B_{1/2}\setminus\Sigma_\e^A)}\le C_\eps\big(
    \|u_\e\|_{L^{2,a}(B_1\setminus\Sigma_\e^A)}
    +
    \|f\|_{L^{p,a}(B_{1})}
    +
    \|F\|_{L^{q,a}(B_{1})}
    \big).
\end{equation}
Our goal is to show that $C_\eps$ remains uniformly bounded as $\eps \to 0$. The proof proceeds by contradiction and is divided into several steps.

\smallskip

\noindent \emph{Step 1. Contradiction argument and blow-up sequences.}

\smallskip

 Assume, for the sake of contradiction, that there exist two sequences ${\eps_k}$ and ${u_k}$ such that $\eps_k \to 0$ as $k \to \infty$, the functions $u_k$ are non trivial, satisfy \eqref{approxPDEA} with $\eps = \eps_k$ and it holds that
\begin{equation}\label{eq:contradiction1}
      \|u_k\|_{C^{0,\alpha}(B_{1/2}\setminus\Sigma_{\eps_k}^A)}\geq k\big(
    \|u_k\|_{L^{2,a}(B_1\setminus\Sigma_{\eps_k}^A)}
    +
    \|f\|_{L^{p,a}(B_{1})}
    +
    \|F\|_{L^{q,a}(B_{1})}
    \big).
\end{equation}
By Proposition \ref{prop:moser}, inequality \eqref{eq:contradiction1} implies the existence of a constant $c > 0$, independent of $k$, such that
\[
      [u_k]_{C^{0,\alpha}(B_{1/2}\setminus\Sigma_{\eps_k}^A)}\geq c\, k\|u_k\|_{L^{\infty}(B_{3/4}\setminus\Sigma_{\eps_k}^A)}\,.
\]
Now, let $\eta \in C^\infty_c(B_{3/4})$ be a function satisfying $0 \leq \eta \leq 1$ and $\eta = 1$ on $B_{1/2}$.
We define 
\[
M_k = [\eta u_k]_{C^{0,\alpha}(B_{1}\setminus\Sigma_{\eps_k}^A)}\,,
\]
and observe that
\begin{equation}\label{eq:alpha:RR:contradiction2}
M_k \geq [u_k]_{C^{0,\alpha}(B_{1/2}\setminus\Sigma_{\eps_k}^A)} \geq c\, k \|u_k\|_{L^{\infty}(B_{3/4}\setminus\Sigma_{\eps_k}^A)}\,. 
\end{equation}
Consider two sequences of points $z_k=(x_k,y_k), \hat{z}_k=(\hat{x}_k,\hat{y}_k)\in B_{1}\setminus\Sigma_{\e_k}^A$ such that
\begin{equation}\label{eq:points:seminorm:alpha}
    \frac{|(\eta u_k)(z_k)- (\eta u_k)(\hat{z}_k)|}{|z_k-\hat{z}_k|^\alpha}\geq \frac12 M_k\,.
\end{equation}
Without loss of generality, we may assume that $z_k \in B_{3/4}$. Define $r_k= |z_k-\hat{z}_k|$. From \eqref{eq:alpha:RR:contradiction2} and \eqref{eq:points:seminorm:alpha} we obtain 
\[
c k \|u_k\|_{L^{\infty}(B_{3/4}\setminus\Sigma_{\eps_k}^A)} \leq \frac{|(\eta u_k)(z_k)- (\eta u_k)(\hat{z}_k)|}{r_k^\alpha} \leq \frac{2\|u_k\|_{L^{\infty}(B_{3/4}\setminus\Sigma_{\eps_k}^A)}}{r_k^\alpha}\,.
\]
This implies 
\[
r_k\le c k^{-\frac{1}{\alpha}}\to0\,,\quad \text{ as } k\to \infty\,.
\]
Let us define $d_k={\rm dist}(z_k, \partial \Sigma_{\eps_k}^A)$. From now on we distinguish three distinct cases:
\smallskip
\begin{itemize}
    \item[$\,$] \textbf{Case 1:} 
    $\quad\displaystyle \frac{d_k}{r_k}\to \infty \text{ as } k\to\infty\,,$

\item[$\,$] \textbf{Case 2:} $\quad\displaystyle\frac{d_k}{r_k}\leq c \text{ uniformly in } k\,,
     \quad \text{ and } \quad \displaystyle{
    \frac{|y_k|}{r_k}\to \infty  \text{ as } k\to\infty\,.}$

\item[$\,$] \textbf{Case 3:}
    $\quad
    \displaystyle{
    \frac{|y_k|}{r_k}\le c \text{ uniformly in } k\,.}$
\end{itemize}
\smallskip
Define
\[
Z_k = (x_k, Y_k) = \Big(x_k, \tau \frac{y_k}{|y_k|}\Big)\,,
\]
where $\tau$ is chosen such that $Z_k \in \partial \Sigma_{\eps_k}^A$. Such a point exists, is unique, and satisfies $\tau < |y_k|$, thanks to Lemma \ref{L:normal} point $ii)$. Therefore,
\[
|y_k|>  |y_k| - \tau = |y_k - Y_k| = |z_k - Z_k|\geq d_k\,.
\]
As a result, along a suitable subsequence, the three cases are mutually exclusive and collectively exhaustive, covering all possible scenarios. Heuristically, in {\bf Case 1}, the blow-up scale does not capture $ \partial \Sigma_{\varepsilon_k}^A $; in {\bf Case 2}, $ \partial \Sigma_{\varepsilon_k}^A $ is visible, but $ \Sigma_0 $ is not; and in {\bf Case 3}, both $ \Sigma_0 $ and $ \partial \Sigma_{\varepsilon_k}^A $ are visible (although they may or may not coincide in the rescaled limit).

 Let $z_k^0 = (x_k^0,y_k^0)$ denote a chosen projection of $z_k$ onto $\partial \Sigma_{\e_k}^A$, and define 
\[
\tilde{z}_k=(\tilde{x}_k,\tilde{y}_k) = \begin{cases}
(x_k,y_k),&\text{ in }\textbf{Case 1},\\
(x_k^0,y_k^0),&\text{ in }\textbf{Case 2},\\
(x_k,0),&\text{ in }\textbf{Case 3}.\\
\end{cases}
\]
We observe that, by construction, $|z_k - \tilde z_k| \leq cr_k$.

\begin{center}
\includegraphics[page=1,scale=1]{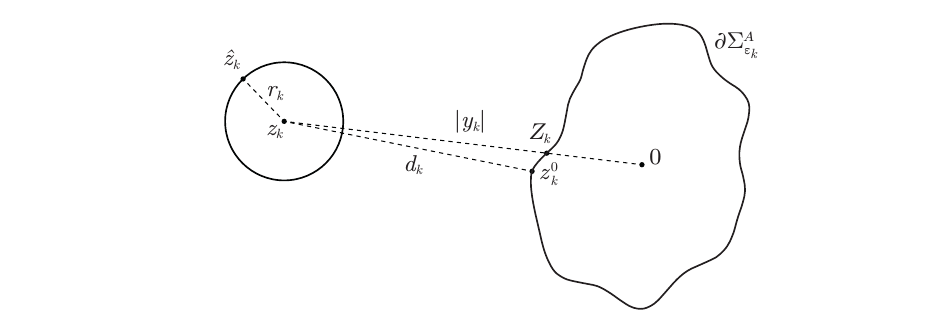}
\captionof{figure}{This image describes {\bf Case 1} when $n=d=2$. In this case we have $d_k/r_k\to\infty$ and $\tilde z_k=z_k$.}
\end{center}

\begin{center}
\includegraphics[page=1,scale=1]{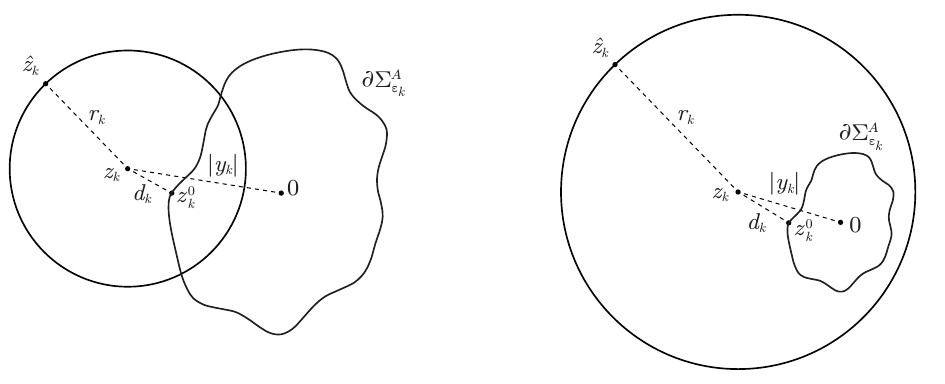}
\captionof{figure}{The images on the left and on the right describe respectively {\bf Case 2} and {\bf Case 3} when $n=d=2$. In {\bf Case 2} we have $d_k/r_k\leq c$, $|y_k|/r_k\to\infty$ and $\tilde z_k=z^0_k$. In {\bf Case 3} we have $|y_k|/r_k\leq c$ and $\tilde z_k=0$.}
\end{center}

We introduce the sequence of rescaled domains
\[
   \Omega_k=\frac{B_{1}\setminus\Sigma_{\e_k}^A-\tilde{z}_k}{r_k}=\Big\{z=(x,y)\in \R^d \mid|\tilde{z}_k+r_kz|<1, A_3^{-1}(\tilde{z}_k+r_k z)(\tilde{y}_k+r_ky)\cdot (\tilde{y}_k+r_ky)\geq \e_k^2\Big\}\,,
\]
and the limit blow-up domain 
\begin{equation}\label{eq:limit:domain:def}
    \Omega_\infty = \{z = (x,y) \in \R^d \mid \text{ exists }\hat k, r>0 \text{ s.t. } B_r(z) \subset \Omega_k \text{ for every }k\geq \hat k\}\,.
\end{equation}
Note that, by this definition, $\Omega_\infty$ is an open set.

Then, we introduce the sequences of functions
\begin{equation*}
    v_k(z)=\frac{(\eta u_k)(\tilde z_k+r_kz)-(\eta u_k)(z_k)}{r_k^\alpha M_k }, \quad \text{ for }z\in \Omega_k\,,
\end{equation*}
and 
\begin{equation*}
    w_k(z)=\frac{\eta( z_k)(u_k(\tilde z_k+r_kz)-u_k(z_k))}{r_k^\alpha M_k}, \quad \text{ for }z\in \Omega_k\,.
\end{equation*}

\smallskip

\noindent \emph{Step 2. Blow-up limit domains.}

\smallskip

 Our goal is to characterize $\Omega_\infty$ along a suitable subsequence. As previously noted, the three distinct regimes described in {\bf Cases 1, 2, 3} correspond to different scenarios, leading to distinct limit blow-up domains. In {\bf Case 1}, we \emph{do not see} the rescaled boundary of the perforation at any scale of the blow-up. Thus, the hole moves farther and farther away from the blow-up centers and the limit domain $\Omega_\infty$ coincides with the entire space $\R^d$. In {\bf Case 2}, we \emph{see} the rescaled boundary of the perforation at all scales, whereas the rescaled center of the perforation remains invisible and increasingly distant. As a result, the rescaled boundary flattens progressively and the limit domain $\Omega_\infty$ becomes a half-space. Finally, in {\bf Case 3}, both the rescaled boundary and the center of the perforation are visible at all scales, leading, as limit domain, to a (potentially) perforated space. In the next part, we will rigorously justify this visual intuition, suggested by Figures 1 and 2.

Notice that for every $z \in \R^d$ we have $|\tilde z_k + r_kz|<1$ for sufficiently large $k$. Indeed, since $|z_k - \tilde z_k|\leq cr_k$ and $r_k \to 0$, then 
\[
|\tilde{z}_k + r_k z|\le |z_k| + |z_k - \tilde {z}_k|+r_k|z| \leq \frac34 + cr_k < 1\,.
\]
Thus, to characterize $\Omega_\infty$, it suffices to determine, for a given $z \in \R^d$, whether a neighborhood of the point $z_k + r_kz$ lies inside or outside $\Sigma^A_{\eps_k}$. 

We claim that $\Omega_\infty = \R^d$ in \textbf{Case 1}, that is, for every $z \in \R^d$, a neighborhood of $z_k + r_kz$ lies outside of $\Sigma^A_{\eps_k}$ for sufficiently large $k$. 
Indeed, let $z \in \R^d$ be fixed, and assume by contradiction that $z_k + r_kz \in \Sigma^A_{\eps_k}$. 
Since $z_k \not \in \Sigma^A_{\eps_k}$, there exists a point $P_k = t_k z_k + (1-t_k)(z_k + r_kz)$ with $t_k \in [0,1)$, such that $P_k \in \partial \Sigma_{\eps_k}^A$. As a consequence,
\[
\infty \leftarrow \frac{d_k}{r_k} \leq \frac{|z_k - P_k|}{r_k} = \frac{|(1-t_k)r_k z|}{r_k} \le |z|\,, 
\]
which leads to a contradiction. 

\smallskip

Let us now consider \textbf{Case 2}. Recall that $\tilde z_k = z_k^0$ and $d_k = |z_k - z_k^0| \leq c\,r_k$. Furthermore, note that by the uniform ellipticity condition,  
\[
\infty \leftarrow\frac{|y_k|}{r_k} \leq \frac{|y_k - y_k^0|}{r_k} + \frac{|y_k^0|}{r_k} \leq \frac{d_k}{r_k} + c\frac{\eps_k}{r_k} \leq c\Big(1 + \frac{\eps_k}{r_k}\Big) \,.
\]
Therefore, in this case, we have ${\eps_k}/{r_k} \to \infty$ as $k \to \infty$.

Since $z_k^0 \in \partial \Sigma_{\eps_k}^A$, a point $z \in \R^d$ satisfies $z_k^0 + r_k z \not\in \Sigma_{\eps_k}^A$ if and only if
\begin{equation}\label{eq:case2equiv}
\frac{\sqrt{A_3^{-1}(z_k^0 + r_k z)(y_k^0 + r_k y)\cdot (y_k^0 + r_k y)} - \sqrt{A_3^{-1}(z_k^0)y_k^0\cdot y_k^0}}{r_k} > 0\,.
\end{equation}
To proceed, we must analyze the behavior of \eqref{eq:case2equiv} as $k \to \infty$.

We define the function $\Psi: \R^d \to \R$ as 
\[
\Psi(z) = \sqrt{A_3^{-1}(z)y\cdot y}\,.
\]
It follows that
\[
\Psi(z)=\e_k, \quad \text{ for every } z \in \partial\Sigma_{\e_k}^A\,.
\]
Moreover, the gradient of $\Psi$ is given by
\[
\nabla \Psi (z) = \frac{(0, A_3^{-1}(z) y)}{\Psi(z)} + \frac{G(z)}{\Psi(z)}\,,
\]
where the function $G: \R^d \to \R^d$ is given by 
\[
G(z) = \frac12\big(\sum_{i,j=1}^n \partial_{z_\ell}b_{i,j}(z)y_iy_j\big)_{\ell=1, \ldots, d}\,, \quad \text{ with } \quad A_3^{-1}(z) = (b_{i,j}(z) )_{i,j = 1, \ldots, n}.
\]
As a result, $\Psi \in C^{0,1}(B_1) \cap C^{1}(B_1 \setminus \Sigma_0)$, and it satisfies $\|\nabla \Psi\|_{L^\infty(B_1)} \le cL$.
Next we show that, for every $ z \in \R^d$, it holds
\begin{equation}\label{eq:dom:bordo:exp}
    \frac{\Psi(z_k^0+r_k z)-\Psi(z_k^0)}{r_k}- \frac{A_3^{-1}(z_k^0) y_k^0}{\Psi(z_k^0)} \cdot y \to 0, \quad\text{as }k\to\infty.
\end{equation}
By the mean value theorem, there exists a point $z^*_k=(x_k^*,y_k^*) \in \R^d$ of the form $z_k^* = z_k^0 + \theta_k r_k z$ with $\theta_k \in [0,1]$, such that
\begin{equation}\label{eq.dom:**0}
\frac{\Psi(z_k^0+r_k z) - \Psi(z_k^0)}{r_k} = \nabla \Psi (z_k^*)\cdot z = \frac{A_3^{-1}(z_k^*)y_k^*}{\Psi(z_k^*)}\cdot y + \frac{G(z_k^*)}{\Psi(z_k^*)}\cdot z\,.
\end{equation}
Using \eqref{eq:unif:ell}, we estimate
\begin{equation}\label{eq.dom:**1}
    \Big|\frac{G(z_k^*)}{\Psi(z_k^*)}\cdot z\Big| \le cL |y_k^*| \leq  c L (|y_k^0| + cr_k) \le c (\e_k + r_k)\,,
\end{equation}
and 
\begin{align}\label{eq.dom:**2}
\begin{split}
    &\Big|
    \frac{A_3^{-1}(z_k^*)y_k^*}{\Psi(z_k^*)} - \frac{A_3^{-1}(z_k^0)y_k^0}{\Psi(z_k^0)}
    \Big| \le 
    \Big|
    \frac{A_3^{-1}(z_k^*)y_k^*}{\Psi(z_k^*)} - \frac{A_3^{-1}(z_k^0)y_k^*}{\Psi(z_k^0)}
    \Big| + 
     \Big|
    \frac{A_3^{-1}(z_k^0)y_k^*}{\Psi(z_k^0)} - \frac{A_3^{-1}(z_k^0)y_k^0}{\Psi(z_k^0)}
    \Big|\\
    &\le
    \Big|
    \frac{A_3^{-1}(z_k^*)y_k^*}{\Psi(z_k^*)} - \frac{A_3^{-1}(z_k^0)y_k^*}{\Psi(z_k^*)}
    \Big|
    +
    \Big|
    \frac{A_3^{-1}(z_k^0)y_k^*}{\Psi(z_k^*)} - \frac{A_3^{-1}(z_k^0)y_k^*}{\Psi(z_k^0)}
    \Big|+  \frac{c}{\e_k} |y_k^*-y_k^0| \\
   & \le  c|z_k^*-z_k^0|\frac{|y_k^*|}{|\Psi(z_k^*)|} + c|y_k^*| 
   \frac{|\Psi(z_k^0)-\Psi(z_k^*)|}{|\Psi(z_k^0)||\Psi(z_k^*)|}
   + c\frac{r_k}{\e_k} \le  c r_k + c\frac{r_k}{\e_k}.
    \end{split}
\end{align}
Combining \eqref{eq.dom:**0}, \eqref{eq.dom:**1}, and \eqref{eq.dom:**2}, we obtain
\[
\Big|\frac{\Psi(z_k^0+r_k z)-\Psi(z_k^0)}{r_k}- \frac{A_3^{-1}(z_k^0) y_k^0}{\Psi(z_k^0)} \cdot y\Big| \leq c (r_k + \eps_k + \frac{r_k}{\eps_k})\to 0\,, 
\]
since ${r_k}/{\eps_k} \to 0$. This completes the proof of \eqref{eq:dom:bordo:exp}.

Next we define, up to subsequences (recall that $z_k^0 \in B_1$), 
\[
\bar{e}=\lim_{k\to\infty}\frac{A_3^{-1}(z_k^0)y_k^0}{|A_3^{-1}(z_k^0)y_k^0|} \in \mathbb{S}^{n-1},\quad \text{ and } \quad  \Pi_{\bar e}=\{z=(x,y) \mid \bar{e}\cdot y > 0 \}.
\]
We claim that $\Omega_\infty = {\Pi}_{\bar e}$, which corresponds to a half-space. 

Fix $z \in \Pi_{\bar e}$. Suppose by contradiction that \eqref{eq:case2equiv} does not hold. Then, by \eqref{eq:dom:bordo:exp}, we have
\[
    0\ge \frac{\Psi(z_k^0)}{|A_3^{-1}(z_k^0)y_k^0|}\frac{\Psi(z_k^0+r_k z)-\Psi(z_k^0)}{r_k} = o(1) + \frac{A_3^{-1}(z_k^0)y_k^0}{|A_3^{-1}(z_k^0)y_k^0|}\cdot y\to \bar{e}\cdot y >0\,,
\]
where we used that, by the uniform ellipticity condition,
\[
\frac{\lambda}{\Lambda^{\frac12}} \leq \frac{\Psi(z_k^0)}{|A_3^{-1}(z_k^0)y_k^0|} \leq \frac{\Lambda}{\lambda^{\frac12}}\,.
\]
Therefore we reach a contradiction, and we infer that $\Pi_{\bar e} \subset \Omega_\infty$. 

Now, fix $z \in \R^d \setminus \overline{\Pi}_{\bar e}$, that is,  such that $\bar{e}\cdot y < 0$ and assume that \eqref{eq:case2equiv} holds. Then,
\[
    0\le\frac{\Psi(z_k^0)}{|A_3^{-1}(z_k^0)y_k^0|}\frac{\Psi(z_k^0+r_k z)-\Psi(z_k^0)}{r_k}\to \bar{e}\cdot y <0\,,
\]
which is a contradiction. Hence $\Omega_\infty \subset \Pi_{\bar e}$ and the claim is proved. 

\smallskip

Finally, let us consider \textbf{Case 3}. Recall that in this case $\tilde{z}_k=(x_k,0)$, and $|y_k| \leq cr_k$. Moreover, we note that 
\[
\e_k < 
|A_3^{-\frac12}(z_k)y_k|
\le
\lambda^{-\frac12}|y_k| \leq cr_k\,,
\]
which implies that $\eps_k/r_k \leq c$ in this case. 
Thus we define the following limits
\[
(\bar{x},0) = \lim_{k\to \infty} (x_k,0),\quad
\bar{A}=A(\bar{x},0) = \lim_{k\to \infty} A(\tilde{z}_k),\quad
\bar{\e} = \lim_{k\to \infty}\frac{\e_k}{r_k}\in[ 0, \infty)\, .
\]
We claim that $\Omega_\infty = \R^d \setminus \Sigma_{\bar{\e}}^{\bar{A}}$. Let us fix $z\in\R^d \setminus \Sigma_{\bar{\e}}^{\bar{A}}$, that is, such that $\bar{A}_3^{-1}y\cdot y > \bar{\e}^2$. Assume by contradiction that $\tilde z_k + r_k z \in \Sigma_{\eps_k}^A$, that is, 
\[
A_3^{-1}(\tilde z_k + r_kz)r_ky\cdot r_k y\le \e_k^2\,.
\] 
Dividing the previous inequality by $r_k^2$ and taking the limit as $k \to \infty$, we obtain
\[
\bar{\e}^2 < \bar{A}_3^{-1}y\cdot y \le \bar{\e}^2\,,
\]
a contradiction. Thus $\R^d \setminus \Sigma_{\bar{\e}}^{\bar{A}} \subseteq \Omega_\infty$. By performing similar computations, we get that every $z\in \Sigma_{\bar{\e}}^{\bar{A}}$ does not belong to $\Omega_k$. Therefore, $\Omega_\infty \subseteq {\R^d \setminus \Sigma_{\bar{\e}}^{\bar{A}} }$ and the claim is proved. 
To summarize, we have shown that
\begin{equation}\label{eq:limit:domain*}
    \Omega_\infty=\begin{cases}
        \R^d, & \text{in }\textbf{Case 1},\\
        \Pi_{\bar e}, & \text{in }\textbf{Case 2},\\
        \R^d \setminus\Sigma_{\bar{\e}}^{\bar{A}}, & \text{in }\textbf{Case 3}.
    \end{cases}
\end{equation}

\smallskip

\noindent \emph{Step 3. H\"older estimates and convergence of the blow-up sequences.}

\smallskip

Let $z,z'\in \Omega_k$. It holds
\[
|v_k(z)-v_k(z')|=\frac{|(\eta u_k)(\tilde z_k+r_kz)-(\eta u_k)(\tilde z_k+r_kz')|}{r_k^\alpha M_k }\le |z-z'|^\alpha\,.
\]
Therefore,
\begin{equation}\label{eq:alpha:estimate}
[v_k]_{C^{0,\alpha}(\Omega_k)}\le 1\,.
\end{equation}

Now, fix a compact set $K \subset \Omega_\infty$, and observe that $K \subset \Omega_k$ for sufficiently large $k$. Let $z \in K$, and consider the sequence of points
\[
\xi_k =\frac{{z}_k-\tilde{z}_k}{r_k}\,.
\]
Note that $\xi_k \in \Omega_k$, $|\xi_k| \leq c$ uniformly in $k$, and $v_k(\xi_k) = 0$ for every $k$. Then, by \eqref{eq:alpha:estimate}, we obtain
\[
|v_k(z)| = |v_k(z) - v_k(\xi_k)|\leq |z-\xi_k|^\alpha \le c(K)\,,
\]
which implies that $\|v_k\|_{C^{0,\alpha}(K)} \leq c(K)$.
Hence, we can apply the Arzel\'a-Ascoli theorem to conclude that $v_k\to\bar{v}$ uniformly in $K$. By an exhaustion of $\Omega_\infty$ via compact subsets $K\subset\Omega_\infty$ and a standard diagonal argument, we extend $\bar v$ to the entire $ \Omega_\infty$ and obtain that $\bar{v}$ satisfies
\begin{equation}\label{eq:alpha:v}
[\bar v]_{C^{0,\alpha}( \Omega_\infty)}\le 1\,.
\end{equation}

Let us now show that $w_k \to \bar v$ uniformly on compact sets. 
Fix a compact set $K\subset \Omega_\infty$. We claim that
\begin{equation}\label{eq:samelimit}
\sup_{z \in K}|v_k(z) - w_k(z) | \to 0 \quad \text{ as }k \to \infty. 
\end{equation}
Let $z \in K \subset \Omega_\infty$ be fixed. Since $z \in \Omega_k$ for large $k$, we have $\tilde z_k + r_k z \in B_\tau \setminus \Sigma^A_{\eps_k}$ for some $\tau <4/5$ depending only on $K$. Thus, recalling also that $|z_k - \tilde z_k| \leq cr_k$, we get
\[
\begin{aligned}
|v_k(z) - w_k(z) | = \frac{|u(\tilde z_k + r_kz)||\eta(\tilde z_k + r_kz) - \eta( z_k)|}{r_k^\alpha M_k}
\leq \frac{c\|u\|_{L^\infty(B_\tau\setminus \Sigma^A_{\eps_k})}(|z_k - \tilde z_k| + r_k |z|)}{r_k^\alpha M_k} \leq ck^{-1} r_k^{1-\alpha}\,,
\end{aligned}
\]
where in the last inequality we used \eqref{eq:alpha:RR:contradiction2}. Since $\alpha <1$, we immediately get \eqref{eq:samelimit}.
Therefore, since $v_k \to \bar v$ uniformly, we conclude that $w_k \to \bar v$ uniformly as well. 

\smallskip

\noindent \emph{Step 4. The limit function is not constant.}

\smallskip 

\noindent Let us consider the sequences of points
\[
\xi_k=\frac{{z}_k-\tilde{z}_k}{r_k}, \quad 
\hat\xi_k=\frac{\hat{z}_k-\tilde{z}_k}{r_k}\,.
\]
As already pointed out, $\xi_k \in \Omega_k$, $|\xi_k| \leq c$ uniformly in $k$, and $v_k(\xi_k) = 0$ for every $k$. In fact, it also holds $\hat \xi_k \in \Omega_k$ for every $k$, and since $|\xi_k-\hat\xi_k|=1$, we also have  $|\hat\xi_k| \leq c$ uniformly in $k$.
Moreover, by \eqref{eq:points:seminorm:alpha} (recall that $r_k = |z_k - \hat z_k|$) we get
\[
|v_k(\hat\xi_k)| = |v_k(\xi_k)-v_k(\hat\xi_k)|=\frac{|\eta u_k({z_k})-\eta u_k(\hat{z_k})|}{r_k^\alpha M_k}\ge \frac12\,.
\]
Now, using that $|\xi_k|, |\hat \xi_k| \leq c$, one can see that there exist $\xi, \hat\xi \in \bar \Omega_\infty$ such that $\xi_k\to\xi$, $\hat \xi_k\to \hat\xi$ and $\xi\not=\hat\xi$. Thus, by a simple continuity argument, $\bar v \le \delta $ in a neighbourhood of $\xi$, and $\bar v \geq 1/2-\delta$ in a neighbourhood of $\hat \xi$, for some small $\delta\in(0,1/10)$. Therefore $\bar v$ is not constant.

\smallskip

\noindent\emph{Step 5. The limit function is an entire solution to a homogeneous equation.} 
\smallskip

\noindent Let us denote $A_k(z)=A(\tilde z_k+r_k z)$. Since $\tilde z_k \in B_1$ and $A\in C^1(B_1)$, up to consider a  subsequence the Arzel\'a-Ascoli theorem yields that 
\[
\bar z =\lim_{k\to \infty}\tilde{z}_k \, \quad \text{ and }\quad
\bar{A}=A(\bar{z})=\lim_{k\to\infty}A_k(z)\,,
\]
where $\bar A$ is a constant coefficients symmetric matrix satisfying \eqref{eq:unif:ell}. Next, we define
\[
    \rho_k(y)=\begin{cases}
       \displaystyle{ \frac{|\tilde y_k+r_k y|}{|\tilde y_k|}}, & \text{ in } \textbf{Case 1} \text{ and }\textbf{Case 2},\\
        |y|, & \text{ in } \textbf{Case 3},
    \end{cases}
\]
noticing that, in \textbf{Case 1} and \textbf{Case 2}, $r_k/|\tilde y_k| \to 0$, and thus $\rho_k\to 1$ a.e. in every compact set $K\subset\R^d$.

Let us fix $\phi\in C_c^\infty(\R^d)$, and notice that for $k$ large, ${\rm spt}( \phi)\subset \frac{B_{1}-\tilde z_k}{r_k}$. Since $u_k$ is solution to \eqref{approxPDEA} with $\eps = \eps_k$, we obtain that $w_k$ satisfies
\begin{align}\label{eq:sol:vk}
\begin{aligned}
\int_{\Omega_k}\rho_k^a(y) A_k(z)\nabla w_k(z)\cdot \nabla {\phi}(z)\, dz &=\frac{\eta(z_k)r_k^{2-\alpha}}{M_k} \int_{\Omega_k}\rho_k^a(y) f(\tilde z_k+r_kz){\phi}(z)\, dz \\
& - 
\frac{\eta(z_k)r_k^{1-\alpha}}{M_k}
 \int_{\Omega_k}\rho_k^a(y) F(\tilde z_k+r_kz)\cdot \nabla {\phi}(z)\,dz\, .
 \end{aligned}
\end{align}
First, we prove that the right-hand side of \eqref{eq:sol:vk} vanishes as $k\to\infty$. We compute 
\[
\begin{aligned}
    \Big|\int_{\Omega_k}\rho_k^a(y) f(\tilde z_k+r_kz){\phi}(z)  dz    \Big|
    &\le \Big(
    \int_{\Omega_k} \rho_k^a(y)|f(\tilde z_k+r_kz)|^p  dz  
    \Big)^\frac{1}{p}
    \Big(
    \int_{\supp(\phi)} \rho_k^a(y)|\phi(z)|^{p'}  dz  
    \Big)^{\frac{1}{p'}} \\
    &\le c\|\phi\|_{L^\infty(\R^d)} \Big(
    \int_{\Omega_k} \rho_k^a(y)|f(\tilde z_k+r_kz)|^p  dz  
    \Big)^\frac{1}{p}\,.
\end{aligned}
\]
By \eqref{eq:contradiction1} and \eqref{eq:alpha:RR:contradiction2}, and using that $r_k \leq c |\tilde y_k|$ in \textbf{Case 1} and \textbf{Case 2}, we get 
\[
    \Big(
    \int_{\Omega_k} \rho_k^a(y)|f(\tilde z_k+r_kz)|^p  dz  
    \Big)^\frac{1}{p} \leq ck^{-1}r_k^{-\frac{d + a_+}{p}}M_k\,.
\]
Therefore, thanks to the assumption $\alpha\le 2-\frac{d+a_+}{p}$, we get
\[
\Big|
\frac{\eta(z_k) r_k^{2-\alpha}}{M_k} \int_{\Omega_k}\rho_k^a(y) f(\tilde z_k+r_kz){\phi}(z) dz
\Big|\le c\|\phi\|_{L^\infty(\R^d)}k^{-1}r_k^{2-\alpha-\frac{d+a_+}{p}}\to0\,.
\]
Performing similar computations, we see  that
\[
\Big|\frac{\eta(z_k)r_k^{1-\alpha}}{M_k}
 \int_{\Omega_k}\rho_k^a(y) F(\tilde z_k+r_kz)\cdot \nabla {\phi}(z)\,dz\Big| \leq c\|\D \phi\|_{L^{2}(\R^d, \rho_k^a dz)}k^{-1}r_k^{1-\alpha-\frac{d+a_+}{q}}\to 0\,,
\]
thanks to the assumption $\alpha\le 1-\frac{d+a_+}{q}$.
Thus, the right hand side in \eqref{eq:sol:vk} vanishes as $k \to \infty$. 

It is useful to keep explicit track of the dependence of $\phi$ in the previous computation. Let $R$ be such that $\supp(\phi)\subset R$. Previous part of the proof can be reformulated in the following way: there exists $\delta_k>0$ such that $\delta_k \to 0$ and 
\begin{equation}\label{eq:c0:similcaccioppoli}
\int_{\Omega_k}\rho_k^a(y) A_k(z)\nabla w_k(z)\cdot \nabla {\phi}(z)\, dz \leq \delta_k (\|\phi\|_{L^\infty(\R^d)} + \|\nabla\phi\|_{L^2(\R^d, \rho_k^a dz)})\,.
\end{equation}
We are now in position to show that the left hand side of \eqref{eq:sol:vk} satisfies
\begin{equation}\label{eq:l.h.s.0}
\int_{\Omega_k\cap \supp(\phi)}\rho_k^a A_k\nabla w_k\cdot \nabla {\phi} \, dz \to \int_{\Omega_\infty\cap \supp(\phi)}\bar{\rho}^a\, \bar {A}\nabla \bar{v}\cdot \nabla \phi\, dz\,,
\end{equation}
where
\[
\bar{\rho}(y) = \begin{cases}
    1 &\text{ in }\textbf{Case 1} \text{ and }\textbf{Case 2}\,,\\
    |y| &\text{ in }\textbf{Case 3}\,.
\end{cases}
\]
First, observe that since $u_k \in H^{1,a}(B_1 \setminus \Sigma^A_{\eps_k})$, it follows that $w_k \in H^1(\Omega_k, \rho_k^a dz)$. 
Fix $R>0$. By \emph{Step 3}, the sequence ${w_k}$ is uniformly bounded in $L^\infty(\Omega_k \cap B_{2R})$.
Let $\varphi \in C^\infty_c(B_{2R})$ be such that $0 \leq \varphi \leq 1$ and $\varphi = 1$ in $B_R$. Testing \eqref{eq:c0:similcaccioppoli} with $\phi = w_k \varphi^2$, and applying \eqref{eq:unif:ell} along with H\"older and Young inequalities, standard computations yield
\[
\int_{\Omega_k}\rho_k^a|\nabla (\varphi^2w_k)|^2\, dz \leq c (\|w_k\|^2_{L^\infty(\Omega_k \cap B_{2R})} + 1)\,.
\]
Using the uniform $L^\infty$--bound of $\{w_k\}$  in $\Omega_k \cap B_{2R}$, and the fact that $\varphi=1$ in $B_R$, we conclude that $\{w_k\}$ is uniformly bounded in $H^1(\Omega_k\cap B_{R},\rho_k^adz)$.
Finally, arguing as in Lemma \ref{L:approximation} with minor adjustments, and using that $A_k(z)\to\bar{A}$, $\rho_k^a\to\bar{\rho}^a$ almost everywhere, we conclude that \eqref{eq:l.h.s.0} holds, and that $\bar{v}$ belongs to $H^{1}({\Omega}_\infty\cap B_R,\Bar{\rho}^adz)$ for every $R>0$.
Summing up, we proved that 
\[
\int_{\Omega_\infty}\bar{\rho}^a \bar{A}\nabla \bar{v}\cdot \nabla \phi\, dz = 0 \quad \text{ for every } \phi \in C^\infty_c(\R^d)\,.
\]
Hence, recalling the definition of the limit blow-up domain $\Omega_\infty$, see \eqref{eq:limit:domain*}, we infer that $\bar{v}$ is an entire solution
to:
\begin{itemize}
\item[\,] {\bf Case 1} 
\[
-\dive(\bar{A}\nabla\bar{v})=0,\quad\text{in }\R^d\,;
\]
\item[\,] \textbf{Case 2}
\[
\begin{cases}
-\dive(\bar{A}\nabla\bar{v})=0,&\text{ in }\Pi_{\bar{e}}\,\\
\bar A\nabla \bar v \cdot \nu = 0 & \text{ on }\partial \Pi_{\bar{e}};
\end{cases}
\]
\item[\,] \textbf{Case 3}, if $\bar{\e}=0$
\[
-\dive(|y|^a\bar{A}\nabla\bar{v})=0, \quad\text{ in }\R^d\,,
\]
\item[\,] \textbf{Case 3}, if $\bar \eps >0$
\[
\begin{cases}
-\dive(|y|^a\bar{A}\nabla\bar{v})=0 & \text{in }\R^d\setminus\Sigma_{\bar{\e}}^{\bar A}\,,\\
\bar{A}\nabla\bar{v}\cdot \nu =0  & \text{on } \Sigma_{\bar{\e}}^{\bar A}\,.
\end{cases}
\]
\end{itemize}

\smallskip

\noindent\emph{Step 6. Liouville theorems and contradiction.}

\smallskip

From \eqref{eq:alpha:v} we have that $\bar v $ satisfies 
$$|\bar{v}(z)|\le C(1+|z|)^\alpha,$$
for every $z\in \Omega_\infty$, where $\alpha<\min\{\alpha_*,1\}$, by the assumption \eqref{eq:0:alpha:reg1}.
By invoking appropriate Liouville-type Theorems we get that $\bar v$ must be constant. In \textbf{Case 1}, we use the classical Liouville Theorem in the whole space; in \textbf{Case 2}, we use the Liouville Theorem in an half space with an homogeneous conormal boundary condition (for instance, see \cite[Theorem 1.6]{TerTorVit24a}); in \textbf{Case 3}, we use Theorem \ref{L:Liouville}.
Hence, we reach a contradiction, since $\bar v$ is not constant by \emph{Step 4}.
Thus, \eqref{eq:alpha:RR} holds with a constant $C$ uniformly bounded in $\eps$. The proof is complete.
\end{proof}

\begin{remark}\label{R:stability:eigenvalue}
We point out that our theory can be extended to equations including lower order terms. In particular, we highlight its connection to the study of eigenvalue problems for the Laplacian in domains with small holes (see \cite{FelLivOgn25,Oza85,RauTay75} and the references therein) - which corresponds to our case $d=n$. More generally, for $2\le n\le d$, let us consider the eigenvalue problem
\begin{equation}\label{eq:eigenvalue}
\begin{cases}
        -\Delta u_\e = \lambda_\e u_\e,&\text{ in } \Omega\setminus \Sigma_\e,\\
    \D u_\e\cdot \nu  = 0,&\text{ on }\partial \Sigma_\e \cap \Omega.
\end{cases}
\end{equation}
Theorem \ref{T:0:1:alpha:eps}, \emph{i)}, which establishes local $\e$-stable H\"older estimates, extends naturally to equations like \eqref{eq:eigenvalue}, yielding
\[
\|u_\e\|_{C^{0,\alpha}(K\setminus \Sigma_\e)} \le C \lambda_\e \|u_\e\|_{L^2(\Omega \setminus \Sigma_\e)},
\]
for every $\alpha \in (0,1)$ and compact set $K \subset \Omega$. The constant $C>0$ depends only on $d$, $n$, $\alpha$, $\dist(K,\partial \Omega)$. 
\end{remark}

\subsection{Stable Schauder estimates}

\begin{proof}[Proof of Theorem \ref{T:0:1:alpha:eps}, ii), stable $C^{1,\alpha}$-estimates] Let $u_\e$ be a family of solutions to \eqref{approxPDEA}.
Since the weight $|y|^a$ is uniformly elliptic away from $\Sigma_0$, well-known results in elliptic regularity imply that for every $0 < \eps \ll1$ (so that $\partial \Sigma_\eps^A \cap B_1$ is of class $C^{1,\alpha}$ being $A\in C^{1,\alpha}$ with $\|A\|_{C^{1,\alpha}(B_1)}\leq L$, see Remark \ref{R:C1:boundary}), there exists a constant $C_\eps$ (depending only on $d$, $n$, $a$, $\lambda$, $\Lambda$, $p$, $\alpha$, $L$ and $\eps$) such that for any solution $u_\eps$ to \eqref{approxPDEA} the following estimate holds
\begin{equation}\label{eq:c1:alpha:RR}
    \|u_\e\|_{C^{1,\alpha}(B_{1/2}\setminus\Sigma_\e^A)}\le C_\eps\big(
    \|u_\e\|_{L^{2,a}(B_1\setminus\Sigma_\e^A)}
    +
    \|f\|_{L^{p,a}(B_{1})}
    +
    \|F\|_{C^{0,\alpha}(B_{1})}
    \big).
\end{equation}
Our goal is to show that $C_\eps$ remains uniformly bounded as $\eps \to 0$. The proof proceeds by contradiction and is divided into several steps.

\smallskip

\noindent \emph{Step 1. Contradiction argument, preliminary estimates and blow-up sequences.} 

\smallskip

Assume, for the sake of contradiction, that there exist two sequences ${\eps_k}$ and ${u_k}$ such that $\eps_k \to 0$ as $k \to \infty$, the functions $u_k$ are non trivial, satisfy \eqref{approxPDEA} with $\eps = \eps_k$ 
and it holds that
\begin{equation}\label{eq:c1:contradiction1}
      \|u_k\|_{C^{1,\alpha}(B_{1/2}\setminus\Sigma_{\eps_k}^A)}\geq k\big(
    \|u_k\|_{L^{2,a}(B_1\setminus\Sigma_{\eps_k}^A)}
    +
    \|f\|_{L^{p,a}(B_{1})}
    +
    \|F\|_{C^{0,\alpha}(B_{1})}
    \big).
\end{equation}
Consider $\zeta_k \in B_{1/2}\setminus\Sigma_{\eps_k}^A$ such that
\begin{equation*}
\frac12\|\nabla u_k\|_{L^{\infty}(B_{1/2}\setminus\Sigma_{\eps_k}^A)} \leq |\nabla u_k(\zeta_k)|\,,
\end{equation*}
and compute 
\begin{equation*}
\begin{aligned}
|\nabla u_k(\zeta_k)|^2 &= \Big(\int_{B_{1/2}\setminus\Sigma_{\eps_k}^A}|y|^a dz\Big)^{-1}\int_{B_{1/2}\setminus\Sigma_{\eps_k}^A}|y|^a|\nabla u_k(\zeta_k)|^2 dz\\
 &\leq c \Big( \int_{B_{1/2}\setminus\Sigma_{\eps_k}^A}|y|^a|\nabla u_k(z) - \nabla u_k(\zeta_k)|^2 dz + \int_{B_{1/2}\setminus\Sigma_{\eps_k}^A}|y|^a|\nabla u_k|^2 dz\Big)\\
  &\leq c \Big( [\nabla u_k]^2_{C^{0,\alpha}(B_{1/2}\setminus\Sigma_{\eps_k}^A)} + \int_{B_{1/2}\setminus\Sigma_{\eps_k}^A}|y|^a|\nabla u_k|^2 dz\Big) \,.
\end{aligned}
\end{equation*}
Combining the last two inequalities with  the Caccioppoli inequality \eqref{eq:caccioppoli},
we conclude that there exists a constant $c>0$ independent on $k$ such that 
\begin{equation*}
\|\nabla u_k\|_{L^{\infty}(B_{1/2}\setminus\Sigma_{\eps_k}^A)} \leq c\big(  [\nabla u_k]_{C^{0,\alpha}(B_{1/2}\setminus\Sigma_{\eps_k}^A)} +
    \|u_k\|_{L^{2,a}(B_1\setminus\Sigma_{\eps_k}^A)}
    +
    \|f\|_{L^{p,a}(B_{1})}
    +
    \|F\|_{C^{0,\alpha}(B_{1})}
    \big)\,.
\end{equation*}
Therefore, by also applying Proposition \ref{prop:moser}, we deduce that inequality \eqref{eq:c1:contradiction1} implies the existence of a constant $c > 0$, independent of $k$, such that
\begin{equation*}
      [\nabla u_k]_{C^{0,\alpha}(B_{1/2}\setminus\Sigma_{\eps_k}^A)}\geq c\, k \Big(\|u_k\|_{L^{2,a}(B_1\setminus\Sigma_{\eps_k}^A)}
    +
    \|f\|_{L^{p,a}(B_{1})}
    +
     \|F\|_{C^{0,\alpha}(B_{1})}\Big)\,.
\end{equation*}
Now, let $\eta \in C^\infty_c(B_{3/4})$ be a function such that $0 \leq \eta \leq 1$ and $\eta = 1$ on $B_{1/2}$.
We define 
\[
M_k = [\nabla(\eta u_k)]_{C^{0,\alpha}(B_{1}\setminus\Sigma_{\eps_k}^A)}\,,
\]
and observe that
\begin{equation}\label{eq:c1:RR:contradiction2}
M_k \geq [\nabla u_k]_{C^{0,\alpha}(B_{1/2}\setminus\Sigma_{\eps_k}^A)} \geq c\, k \Big(\|u_k\|_{L^{2,a}(B_1\setminus\Sigma_{\eps_k}^A)}
    +
    \|f\|_{L^{p,a}(B_{1})}
    +
     \|F\|_{C^{0,\alpha}(B_{1})}\Big)\,. 
\end{equation}
It is worth noting that, by combining \eqref{eq:c1:RR:contradiction2} with Theorem \ref{T:0:1:alpha:eps}, \emph{i)}, for every $\tau <1$ and $\beta \in (0,1)$ it holds
\begin{equation}\label{eq:c1:c0alphest}
\|u_k\|_{C^{0,\beta}(B_\tau \setminus \Sigma_{\eps_k}^A)} \leq c k^{-1}M_k\,.
\end{equation}
Additionally, since $\eta u_k\equiv0$ in $B_1\setminus B_{3/4}$, we also have
\begin{equation}\label{eq:c1:linfest}
\|\nabla(\eta u_k)\|_{L^\infty(B_{1}\setminus\Sigma_{\eps_k}^A)} \leq c M_k\,.
\end{equation}
As a consequence of \eqref{eq:c1:RR:contradiction2}, \eqref{eq:c1:c0alphest}, and  \eqref{eq:c1:linfest}, we immediately infer that
\begin{equation}\label{eq:c1:bcestrough}
\|\eta A \nabla u_k + \eta F\|_{C^{0,\alpha}(B_1 \setminus \Sigma_{\eps_k}^A)} \leq c M_k\,.
\end{equation}
Crucially, the boundary condition satisfied by $u_k$ on $\partial \Sigma_{\e_k}^A$ gives us an estimate on the $y$-components of the field $\eta A \nabla u_k + \eta F$. In fact, by Lemma \ref{L:BC} and \eqref{eq:c1:bcestrough}, for every point $z_k^* \in \partial \Sigma_{\e_k}^A \cap B_1$ and for every $i=1,\dots,n$, it holds
\begin{equation}\label{eq:BC:quantitative:1}
    | (\eta A\nabla u_k+\eta F)(z_k^*)\cdot e_{y_i}| \leq  c\, \eps_k^\alpha \|\eta A\nabla u_k+\eta F\|_{C^{0,\alpha}(B_1 \setminus \Sigma_{\eps_k}^A)}\,\le c\, \e_k^\alpha M_k .
\end{equation}

Next, consider two sequences of points $z_k=(x_k,y_k), \hat{z}_k=(\hat{x}_k,\hat{y}_k)\in B_{1}\setminus\Sigma_{\e_k}^A$ such that
\begin{equation}\label{eq:c1:points:seminorm:alpha}
    \frac{|\nabla(\eta u_k)(z_k)- \nabla(\eta u_k)(\hat{z}_k)|}{|z_k-\hat{z}_k|^\alpha}\geq \frac12 M_k\,.
\end{equation}
Without loss of generality, we can always assume that $z_k \in B_{3/4}$.
Let $r_k= |z_k-\hat{z}_k|$. Unlike the proof of Theorem \ref{T:0:1:alpha:eps}, \emph{i)}, establishing that $r_k \to 0$ as $k\to \infty$ is not straightforward here, so we have to proceed more carefully. Let us denote $d_k={\rm dist}(z_k, \partial \Sigma_{\eps_k}^A)$. From now on, we distinguish three cases. 
\smallskip
\begin{itemize}
    \item[$\,$] \textbf{Case 1:} 
    $\quad \displaystyle \frac{d_k}{r_k}\to \infty \text{ as } k\to\infty\,,$

\item[$\,$] \textbf{Case 2:} $\quad \displaystyle\frac{d_k}{r_k}\leq c \text{ uniformly in } k\,,
     \quad \text{ and } \quad \displaystyle{
    \frac{|y_k|}{r_k}\to \infty  \text{ as } k\to\infty\,.}$

\item[$\,$] \textbf{Case 3:}
    $\quad 
    \displaystyle{
    \frac{|y_k|}{r_k}\le c \text{ uniformly in } k\,.}$
\end{itemize}
\smallskip
We refer to the proof of Theorem \ref{T:0:1:alpha:eps}, \emph{i)} for remarks on these cases. We recall that $d_k \leq |y_k|$, and we note that in {\bf Case 1} and {\bf Case 2}, $ r_k \to 0 $.

Let $z_k^0 = (x_k^0,y_k^0)$ denote a chosen projection of $z_k$ onto $\partial \Sigma_{\e_k}^A$, and define
\[
\tilde{z}_k = \begin{cases}
z_k,&\text{ in }\textbf{Case 1},\\
z_k^0,&\text{ in }\textbf{Case 2}\\
(x_k, 0) &\text{ in }\textbf{Case 3}.\\
\end{cases}
\]
We observe that, by construction, $|z_k - \tilde z_k| \leq cr_k$. Moreover, in {\bf Case 2} and {\bf Case 3} it holds
\[
|z_k^0| \leq |z_k| + d_k \leq \frac34 + cr_k\,,
\]
so that for sufficiently large $k$ there exists $\tau <4/5 $ such that $z_k^0 \in B_\tau\setminus \Sigma_{\eps_k}^A$.

As done in Theorem \ref{T:0:1:alpha:eps}, \emph{i)} let us define  the sequence of domains
\[
   \Omega_k=\frac{B_{1}\setminus\Sigma_{\e_k}^A-\tilde{z}_k}{r_k}
\]
and  the limit domain $\Omega_\infty$ as in \eqref{eq:limit:domain:def}. Next we introduce the points 
\[
\xi_k = \frac{z_k- \tilde z_k}{r_k}\,.
\]
Notice that $\xi_k \in \Omega_k$ for every $k$, and $|\xi_k| \leq c$. Thus, up to subsequences, $\xi_k \to \xi \in \overline\Omega_\infty$. 

Finally, we introduce the sequences of functions
\[
    v_k(z)=\frac{\eta(\tilde z_k + r_k z)(u_k(\tilde z_k+r_kz)- u_k(z_k)) - \eta(z_k)\nabla u_k(z_k)\cdot r_k(z - \xi_k)}{r_k^{1+\alpha} M_k }, \quad \text{ for }z\in \Omega_k\,,
\]
and 
\[
    w_k(z)=\frac{\eta(\tilde z_k)(u_k(\tilde z_k+r_kz)-u_k(z_k))-P_k\cdot r_k(z - \xi_k)}{r_k^{1+\alpha} M_k}, \quad \text{ for }z\in \Omega_k\,.
\]
where 
\[
P_k = 
\begin{cases}
\eta(z_k)\nabla u_k(z_k),&\text{ in }\textbf{Case 1},\\
\eta(z_k^0)\nabla u_k(z_k^0),&\text{ in }\textbf{Case 2},\\
A^{-1}(\tilde z_k)((\eta A \nabla u_k)_1(z_k^0) - (\eta F)_2(z_k^0)),&\text{ in }\textbf{Case 3}.
\end{cases}
\]
Thanks to \eqref{eq:c1:RR:contradiction2}, \eqref{eq:c1:c0alphest}, and  \eqref{eq:c1:linfest}, we readily infer  
\begin{equation}\label{eq:c1:pkbound}
|P_k| \leq \|\nabla(\eta u_k)\|_{L^{\infty}(B_1 \setminus \Sigma_{\eps_k}^A) } + c \|u_k\|_{L^{\infty}(B_\tau \setminus \Sigma_{\eps_k}^A) } + c\|F\|_{L^{\infty}(B_1)} \leq cM_k\,,
\end{equation}
where, in {\bf Case 2} and {\bf Case 3}, we used that $z_k^0 \in B_\tau \setminus \Sigma_{\eps_k}^A$ for some $\tau <4/5$.

\smallskip

\noindent \emph{Step 2. H\"older gradient estimates and convergence of $v_k$.}

\smallskip

Let $z, z' \in \Omega_k$. We have 
\[
\begin{aligned}
|\nabla v_k(z) - \nabla v_k(z')| &\leq \frac{|\nabla(\eta u_k) (\tilde z_k + r_k z) - \nabla(\eta u_k) (\tilde z_k + r_k z')|}{r_k^\alpha M_k}
+ \frac{|u(z_k)||\nabla\eta(\tilde z_k + r_k z) - \nabla\eta (\tilde z_k + r_k z')|}{r_k^\alpha M_k}\\
& \leq |z-z'|^\alpha + c\frac{\|u\|_{L^\infty(B_{3/4}\setminus \Sigma_{\eps_k}^A)}}{M_k}r_k^{1-\alpha} |z-z'| \leq (1 + ck^{-1}|z-z'|^{1-\alpha})|z-z'|^\alpha\,,
\end{aligned}
\] 
where the last inequality follows from \eqref{eq:c1:c0alphest}. 
Therefore, for every fixed compact set $K \subset \R^d$, we can choose $k$ sufficiently large so that 
$
[\nabla v_k]_{C^{0,\alpha}(\Omega_k \cap K)} \leq 2\,.
$
Moreover, since $\nabla v_k(\xi_k) = 0$ and $|\xi_k| \leq c$, it follows that 
\begin{equation}\label{eq:d:vk:right:expansion}
    \|\D v_k\|_{L^\infty(\Omega_k\cap K)}= \sup_{z\in \Omega_k\cap K}|\D v_k(z)-\D v_k(\xi_k)|\le 2 \sup_{z\in \Omega_k\cap K}|z-\xi_k|^\alpha\le c(K).
\end{equation}

Next, we prove that for every compact set
$K\subset \Omega_\infty$, it holds $\|v_k\|_{L^\infty(K)}\le c(K)$. We claim that there exists a compact set $K' \supset K$ such that, for every $z \in K$ and all sufficiently large $k$, there exists a curve $\gamma_k$, depending on $z$, satisfying
\begin{equation}\label{eq:c1,alpha:curve:connecting}
    \gamma_k :(0,1)\to \Omega_k, \quad \gamma_k(0)=\xi_k,\quad \gamma_k(1)=z, \quad  \|\gamma_k\|_{C^{0,1}(0,1)}\le c(1+|z|), \quad \supp(\gamma_k)\subset \Omega_k\cap K'\,,
\end{equation}
for some constant $c > 0$ independent of $z$ and $k$. Assuming the claim, we use \eqref{eq:d:vk:right:expansion}, \eqref{eq:c1,alpha:curve:connecting}, and $v_k(\xi_k) = 0$ to obtain
\begin{equation}\label{eq:v_k:Linfty:1+a}
    |v_k(z)|:= \Big|\int_0^1 \D v_k(\gamma_k(t))\cdot \gamma_k'(t)\,dt\Big| \le \|\D v_k\|_{L^\infty(\Omega_k\cap K')}\|\gamma_k\|_{C^{0,1}(0,1)}\le c(1+|z|^{1+\alpha}) \le c(K),
\end{equation}
which implies $\|v_k\|_{L^\infty(K)}\le c(K)$.

We now prove the claim. In \textbf{Case 1} we take $\gamma_k$ to be the straight-line segment from $z$ to $\xi_k$. Since $\Omega_\infty = \R^d$ (see Theorem \ref{T:0:1:alpha:eps}, \emph{i)}, \emph{Step 2}), it follows that $\gamma_k \subset \Omega_k$ for sufficiently large $k$. The set $K'$ can then be chosen accordingly, noting also that $|\xi_k|\leq c$.

In \textbf{Case 2}, consider the point $\xi_k^*:=\xi_k+ y_k/|y_k|$. By Lemma \ref{L:normal}, $ii)$, the segment joining $\xi_k$ and $\xi_k^*$ lies entirely within $\Omega_k$.
We next show that $\xi_k^* \in \Omega_\infty=\{ y\cdot \bar e>0\}$.
Recalling that $\xi_k \to \xi \in \{ y\cdot \bar e\geq0\}$ and $A_3^{-1}(z_k^0)y_k^0/|A_3^{-1}(z_k^0)y_k^0|\to \bar e$, we compute, for sufficiently large $k$,
\begin{align*}
    \bar{e}\cdot \Big(\xi_k +\frac{y_k}{|y_k|} \Big)&\ge  
    \bar{e}\cdot \Big(\frac{y_k}{|y_k|} \Big)-|\xi_k-\xi| 
  \ge
    \frac{A_3^{-1}(z_k^0)y_k^0}{|A_3^{-1}(z_k^0)y_k^0|} \cdot\frac{y_k}{|y_k|} -
    \Big|\bar{e}-\frac{A_3^{-1}(z_k^0)y_k^0}{|A_3^{-1}(z_k^0)y_k^0|} \Big| +o(1)\\
    &\ge   \frac{A_3^{-1}(z_k^0)y_k^0}{|A_3^{-1}(z_k^0)y_k^0|}\cdot\frac{y_k^0}{|y_k^0|}-  \Big|
    \frac{y_k}{|y_k|} - \frac{y_k^0}{|y_k^0|}
    \Big|+o(1)
    \ge \frac{\lambda}{\Lambda} - \Big|
    \frac{y_k|y_k^0|-|y_k|y_k^0}{|y_k^0||y_k|}
    \Big| + o(1)  \\
    &\ge \frac{\lambda}{\Lambda} - 2
    \frac{|y_k-y_k^0|}{|y_k|} + o(1) = \frac{\lambda}{\Lambda} - 2 \frac{d_k}{r_k} \frac{r_k}{|y_k|}+ o(1) \ge \frac{\lambda}{\Lambda} +o(1)\,,
\end{align*}
where we used the uniform ellipticity condition \eqref{eq:unif:ell} and the assumptions of \textbf{Case 2}, namely, $d_k/r_k\le c$ and $|y_k|/r_k\to\infty$. Therefore, for sufficiently large $k$, we have that $\xi_k^* \cdot \bar e \ge c\ge 0$, which implies that $\xi_k^* \in \Omega_\infty$. Since $\Omega_\infty$ is convex, the segment connecting $z$ and $\xi_k^*$ lies entirely in $\Omega_\infty$, hence in $\Omega_k$ for sufficiently large $k$. We thus define $\gamma_k$ as the concatenation of the segment from $\xi_k$ to $\xi_k^*$, and the segment from $\xi_k^*$ to $z$, and we easily see that $\gamma_k$ satisfies \eqref{eq:c1,alpha:curve:connecting}.

We now focus on \textbf{Case 3}. Given a point $\zeta \in \Omega_k$, let $\zeta^{**}$ denote its projection onto the cylinder 
\[
\mathcal{C}_k:= \frac{\partial \Sigma_{2 \Lambda^{1/2}\eps_k} - \tilde z_k}{r_k} = \{z \in \R^d \mid |y| = 2\Lambda^{1/2} {\e_k}/{r_k}\}\,.
\]
Note that, by Lemma \ref{L:normal}, $ii)$, the segment joining $\zeta$ and $\zeta^{**}$ lies entirely within $\Omega_k$, at least for sufficiently large $k$ (otherwise it may intersect the spherical part of $\partial \Omega_k$). Now, fix $z \in K \subset \Omega_\infty$. Consider a path contained in $\mathcal{C}_k \cap \Omega_k$ that connects $z^{**}$ to $\xi_k^{**}$. This path consist of a translation in the $x$-variable, followed by a geodesic on the $n$-dimensional sphere $\partial B_{2\Lambda {\e_k}/{r_k}}$. We then define $\gamma_k$ as the concatenation of the segment connecting $z$ to $z^{**}$, the aforementioned path connecting $z^{**}$ and $\xi_k^{**}$ within $\mathcal{C}_k$, and the segment connecting $\xi_k^{**}$ to $\xi_k$. Since in \textbf{Case 3} we have $\e_k/r_k\le c$, it is straightforward to verify that $\gamma_k$ satisfies \eqref{eq:c1,alpha:curve:connecting}. To conclude, it suffices to choose a compact set $K' \supset K$ large enough so that $\supp(\gamma_k) \subset K'\cap \Omega_k$ for every $z \in K$ and all sufficiently large $k$. The claim is proved. 

Summing up, we have shown that for every $K\subset \Omega_\infty$ it holds $\|v_k\|_{C^{1,\alpha}(K)} \leq c(K)$. Applying Arzel\'a-Ascoli theorem, together with an exhaustion of $\Omega_\infty$ by compact subsets and a standard diagonal argument, we conclude that, for every $\gamma\in(0,\alpha)$, it holds $v_k \to \bar v $ in $C^{1, \gamma}_{\rm loc}(\Omega_\infty)$. Moreover, using \eqref{eq:v_k:Linfty:1+a} we obtain the growth estimate
\begin{equation}\label{eq:growth:c1}
|\bar v(z)| \leq c( 1+ |z|^{1+\alpha}).
\end{equation}

\smallskip

\noindent \emph{Step 3. Convergence of $w_k$.}

\smallskip

First, we show that for sufficiently large $k$, there exists a constant $c>0$ such that
\begin{equation}\label{eq:c1:diffbound}
\frac{|\eta(z_k) \nabla u_k (z_k)  - P_k|}{r^\alpha_kM_k} \leq c\,.
\end{equation}
In {\bf Case 1}, this result is straightforward. In {\bf Case 2}, recall that $z_k, z_k^0 \in B_\tau \setminus \Sigma_{\eps_k}^A$ for some $\tau<4/5$ and $|z_k-z_k^0| \le c r_k$. By \eqref{eq:c1:c0alphest} and \eqref{eq:c1:linfest}, we have 
\begin{equation}\label{eq:c1:tech666}
\begin{aligned}
\frac{|(\eta\nabla u_k)(z_k)  - (\eta\nabla u_k)(z_k^0)|}{r^\alpha_kM_k} &\leq  \frac{|\nabla (\eta u_k)(z_k)  - \nabla(\eta u_k)(z_k^0)|}{r^\alpha_kM_k} + \frac{|(u_k\nabla \eta)(z_k)  - (u_k\nabla \eta)(z_k^0)|}{r^\alpha_kM_k}\\
&\leq \frac{|z_k - z_k^0|^\alpha}{r_k^\alpha} + c\frac{\|u_k\|_{C^{0,\alpha}(B_\tau \setminus \Sigma_{\eps_k}^A)}}{M_k}\frac{|z_k - z_k^0|^\alpha}{r_k^\alpha} \leq c\,.
\end{aligned}
\end{equation}
For {\bf Case 3}, note first that $|z_k^0 - \tilde z_k| \leq d_k + |y_k| \leq cr_k$, $\eps_k\leq cr_k$, and
\[
A^{-1}(\tilde z_k)(\eta A \nabla u_k)(z_k^0) - P_k=(\eta A\nabla u_k + \eta F)_2(z_k^0).
\]
Therefore, using \eqref{eq:c1:tech666}, \eqref{eq:c1:linfest} and \eqref{eq:BC:quantitative:1}, we have
\[
\begin{aligned}
\frac{|(\eta \nabla u_k) (z_k)  - P_k|}{r^\alpha_kM_k}\leq \frac{|(\eta\nabla u_k)(z_k)  - (\eta\nabla u_k)(z_k^0)|}{r^\alpha_kM_k} + \frac{|A^{-1}(z_k^0) - A^{-1}(\tilde z_k)||(\eta A \nabla u_k)(z_k^0)|}{r^\alpha_kM_k}& \\
+ \frac{|A^{-1}(\tilde z_k)(\eta A \nabla u_k)(z_k^0) - P_k |}{r^\alpha_kM_k}\leq c\Big( 1 + \frac{|(\eta A \nabla u_k + \eta F)_2(z_k^0)|}{r^\alpha_kM_k}\Big)\leq c\Big( 1 + \Big(\frac{\eps_k}{r_k} \Big)^\alpha\Big) &\leq c\,.
\end{aligned}
\]
Thus, in all three cases, \eqref{eq:c1:diffbound} is satisfied. 
As a consequence, up to subsequences, we can define 
\begin{equation}\label{eq:c1:wlim}
V = \lim_{k \to \infty}\frac{\eta(z_k) \nabla u_k (z_k)  - P_k}{r_k^\alpha M_k} \qquad \text{ and }\qquad \bar w(z) = \bar v(z) - V\cdot(z- \xi)\,.
\end{equation}
where $\xi \in \bar \Omega_\infty$ is such that $\xi_k \to \xi$. 
To conclude, we show that $w_k \to \bar w$ uniformly on every compact set $K \subset \Omega_\infty$.
Indeed, we have
\[
\begin{aligned}
&\sup_{z \in K}|v_k(z) - V\cdot (z-\xi) - w_k(z)| \leq \sup_{z \in K}\frac{|\eta(\tilde z_k + r_k z)- \eta(\tilde z_k)||u_k(\tilde z_k+r_kz)- u_k(z_k)|}{r_k^{1+\alpha} M_k } \\
&+ \Big|\frac{\eta(z_k) \nabla u_k (z_k)  - P_k}{r_k^\alpha M_k} - V \Big||z-\xi_k|+ |V||\xi -\xi_k| 
\leq c \Big(\frac{\|u_k\|_{C^{0,\alpha}(B_\tau \setminus \Sigma_{\eps_k}^A)}}{M_k} + o(1)\Big) \leq c(k^{-1} + o(1)) \to 0\,.
\end{aligned}
\]
Finally, we observe that from \eqref{eq:growth:c1} and \eqref{eq:c1:wlim} it follows that 
\begin{equation}\label{eq:growth:c1:w}
|\bar w(z)| \leq c( 1 + |z|^{1+\alpha})\,.
\end{equation}

\smallskip

\noindent \emph{Step 4. The limit function is not linear.}

\smallskip

Let us consider the sequences of points
\[
\xi_k=\frac{{z}_k-\tilde{z}_k}{r_k}, \quad 
\hat \xi_k=\frac{\hat{z}_k-\tilde{z}_k}{r_k}.
\]
As previously noted, $\xi_k \in \Omega_k$, $|\xi_k| \leq c$ uniformly in $k$, and $v_k(\xi_k) = \nabla v_k(\xi_k) =  0$ for every $k$. In fact, it also holds that $\hat \xi_k \in \Omega_k$ for every $k$, and since $|\xi_k-\hat \xi_k|=1$, we have  $|\hat\xi_k| \leq c$ uniformly in $k$. 
Moreover, by \eqref{eq:c1:points:seminorm:alpha} (recall that $r_k = |z_k - \hat z_k|$) we get
\[
|\nabla v_k(\hat\xi_k)| \geq \frac{|\nabla (\eta u_k)({z_k})-\nabla(\eta u_k)(\hat{z_k})|}{r_k^\alpha M_k} - \frac{|u(z_k)||\nabla\eta(z_k) - \nabla\eta (\hat z_k)|}{r_k^\alpha M_k} \geq \frac12 - ck^{-1} > \frac14\,.
\]

 Since $\xi_k \to \xi$, $\hat \xi_k\to \hat \xi$ with $\xi, \hat \xi \in \bar \Omega_\infty$ and $\xi\not=\hat \xi$, a simple continuity argument implies that $\nabla\bar v \le \delta $ in a neighbourhood of $\xi$, and $\nabla\bar v \geq 1/4-\delta$ in a neighbourhood of $\hat \xi$, for some small $\delta\in(0,1/10)$. Therefore $\nabla\bar v$ is not constant and, by \eqref{eq:c1:wlim}, we conclude that $\nabla \bar w$ is not constant either. 

\smallskip

\noindent \emph{Step 5. $r_k \to 0$ in \textbf{Case 3}}.

\smallskip

Assume by contradiction that, up to a subsequence, $r_k \to \bar{r} \in (0,2]$. Since $\bar r >0$, it is straightforward to verify that $\Omega_\infty = B_{1/\bar r}(\varsigma)\setminus \Sigma_0$ for some $\varsigma \in \R^d$. 
Let $K \subset \Omega_\infty$ be a compact set. For every $z \in K$, using \eqref{eq:c1:c0alphest}, we estimate
\[
\Big|\frac{\eta(\tilde z_k + r_k z)(u_k(\tilde z_k+r_kz)- u_k(z_k))}{r_k^{1+\alpha} M_k}\Big|  \leq 2\frac{\|u\|_{L^\infty(B_{3/4}\setminus \Sigma_{\eps_k}^A)}}{r_k^{1+\alpha}M_k} \leq ck^{-1}\,,
\]
where in the last inequality we used that $r_k \geq c >0$ for sufficiently large $k$.  
Moreover, from \eqref{eq:c1:c0alphest} and \eqref{eq:c1:linfest}, we obtain
\[
\Big|\frac{\eta(z_k)\nabla u_k(z_k)}{r_k^{\alpha} M_k }\Big|\le \frac{c}{\bar r^\alpha} \leq c\,.
\]
Hence, recalling the definition of $v_k$ and thanks to \emph{Step 2}, we infer that there exists $W \in \R^d$ such that 
\[
\bar v(z) = -W \cdot (z-\xi) \,.
\]
This contradicts the fact, established in \emph{Step 4}, that $\nabla \bar{v}$ is not constant. Therefore, we conclude that $r_k \to 0$. 
At this point, arguing as in Theorem \ref{T:0:1:alpha:eps}, \emph{i)}, \emph{Step 2}, we find that the limit domain $\Omega_\infty$ is given by \eqref{eq:limit:domain*}.

\smallskip

\noindent\emph{Step 6. The limit function is an entire solution to a homogeneous equation.}

\smallskip

Let us denote 
$A_k(z)=A(\tilde z_k+r_k z)$. Since $\tilde z_k \in B_{3/4}$, and $A\in C^{1,\alpha}(B_1)$, we have, up to subsequences
\[
\bar z = (\bar x, \bar y) = \lim_{k\to \infty} (\tilde x_k, \tilde y_k)\, \quad \text{ and }\quad
\bar{A}=A(\bar{z})=\lim_{k\to\infty}A_k(z)\,,
\]
where $\bar A$ is a symmetric matrix with constant coefficients that satisfies \eqref{eq:unif:ell}.
Next, let us define
\[
    \rho_k(y)=\begin{cases}
       \displaystyle{ \frac{|\tilde y_k+r_k y|}{|\tilde y_k|}}, & \text{ in } \textbf{Case 1} \text{ and }\textbf{Case 2},\\
        |y|, & \text{ in } \textbf{Case 3},
    \end{cases}
\]
noting that, in \textbf{Case 1} and \textbf{Case 2}, $\rho_k\to 1$ uniformly on every compact set $K\subset\R^d$ as $k\to \infty$.

Let us fix $\phi\in C_c^\infty(\R^d)$, and observe that for sufficiently large $k$, ${\rm spt}( \phi)\subset \frac{B_{1}-\tilde z_k}{r_k}$. Since $u_k$ is solution to \eqref{approxPDEA}, we find that $w_k$ satisfies
\begin{equation}\label{eq:sol:wk}
\int_{\Omega_k}\rho_k^a(y) A_k(z)\nabla w_k(z)\cdot \nabla {\phi}(z)\, dz = I_1 - I_2 - I_3 - I_4\,,
\end{equation}
where
\[
\begin{aligned}
& I_1 = \frac{r_k^{1-\alpha}}{M_k} \int_{\Omega_k}\rho_k^a(y) \eta(\tilde z_k)f(\tilde z_k+r_kz){\phi}(z)\, dz\,,\\
& I_2 =\frac{r_k^{-\alpha}}{M_k} \int_{\Omega_k}\rho_k^a(y) \eta(\tilde z_k)\big(F(\tilde{z}_k+r_k z)-F(\tilde{z}_k) \big)\cdot \nabla\phi(z)\,dz\,,\\
& I_3 = \frac{r_k^{-\alpha}}{M_k} \int_{\Omega_k}\rho_k^a(y)
(A(\tilde{z}_k+r_k z)-A(\tilde{z}_k))P_k\cdot\nabla\phi(z)\,dz\,,\\
& I_4 = \frac{r_k^{-\alpha}}{M_k} \int_{\Omega_k}\rho_k^a(y)
(A(\tilde z_k)P_k + \eta(\tilde z_k) F(\tilde z_k))\cdot\nabla\phi(z)\,dz\,.
\end{aligned}
\]
Following the same argument as in \emph{Step 5} of the proof of Theorem \ref{T:0:1:alpha:eps}, \emph{i)}, we observe that as $k\to \infty$, 
\[
|I_1| \leq c \|\phi\|_{L^\infty(\R^d)} k^{-1}r_k^{1-\alpha - \frac{d+a_+}{p}}\to 0\,,
\] 
due to the assumption $\alpha \le 1-\frac{d+a_+}{p}$.

\smallskip

Next, using \eqref{eq:c1:RR:contradiction2}, we find that as $k \to \infty$, 
\[
\begin{aligned}
|I_2| \leq \frac{r_k^{-\alpha}}{M_k} \int_{\Omega_k\cap\supp(\phi)}\rho_k^a(y) \big|F(\tilde{z}_k+r_k z)-F(\tilde{z}_k) \big||\nabla \phi(z)|\,dz\le\|\nabla\phi\|_{L^1(\R^d, \rho_k^a dz)}\frac{\|F\|_{C^{0,\alpha}(B_1)}}{M_k} \le c k^{-1}\to0\,.
\end{aligned}
\]

\smallskip

For $I_3$, using \eqref{eq:c1:pkbound}, we obtain
\[
\begin{aligned}
|I_3| \leq  \frac{r_k^{-\alpha}}{M_k} \int_{\Omega_k\cap \supp(\phi)}\rho_k^a(y)
|A(\tilde{z}_k+r_k z)-A(\tilde{z}_k)||P_k||\nabla \phi(z)|\,dz
 \leq \|\nabla\phi\|_{L^1(\R^d, \rho_k^a dz)}Lr_k^{1-\alpha} \frac{|P_k|}{M_k} 
\leq c r_k^{1-\alpha} \to 0\,.
\end{aligned}
\]
Finally, it remains to prove that $I_4$ vanishes as $k \to \infty$. 

Let us first examine \textbf{Case 1}. Here, recall that $\tilde{z}_k=z_k$, $r_k/|y_k|\to0$, $\rho_k^a\to 1$, and $\Omega_\infty = \R^d$. Notice that $|y_k+r_k y|\ge c|y_k|$ for sufficiently large $k$. Using this, we estimate
\[
|\nabla\rho_k^a(y)|=\Big|
\frac{a r_k\rho_k^a(y) (y_k+r_k y)}{|y_k+r_ky|^2}
\Big|
\le c\frac{r_k}{|y_k|}\rho_k^a(y)\,.
\]
Moreover, we know that $\supp (\phi) \subset\subset\Omega_k$ for sufficiently large $k$. Integrating by parts, we obtain 
\[
\begin{aligned}
|I_4| &= \frac{r_k^{-\alpha}}{M_k} \Big|\int_{\supp(\phi)}\rho_k^a(y)(\eta A \nabla u_k+\eta F)({z}_k)\cdot\nabla\phi(z)\, dz\Big|\\
    &= \frac{r_k^{-\alpha}}{M_k} \Big|\int_{\supp(\phi)}\nabla\rho_k^a(y)\cdot (\eta A \nabla u_k+\eta F)({z}_k)\phi(z)\, dz\Big|\le c \|\phi\|_{L^{\infty}(\R^d)}\frac{r_k^{1-\alpha}}{|y_k|M_k} | (\eta A \nabla u_k+\eta F)_2({z}_k)|\,.
\end{aligned}
\]
As pointed out in Theorem \ref{T:0:1:alpha:eps}, \emph{i)}, \emph{Step 1}, there exists $Z_k= (x_k, Y_k) \in \partial \Sigma_{\eps_k}^A$ such that $|Y_k| < |y_k|$ and $y_k \cdot Y_k = |y_k||Y_k|$. Moreover, it holds
\[
d_k \leq |z_k - Z_k| \leq |y_k|\,,
\] 
and, by construction, $Z_k \in B_{3/4}$. By using \eqref{eq:BC:quantitative:1} and \eqref{eq:c1:bcestrough}, we get
\[
\begin{aligned}
|(\eta A \nabla u_k+\eta F)_2({z}_k)| &\leq |(\eta A \nabla u_k+\eta F))({z}_k) - (\eta A \nabla u_k+\eta F)({Z}_k)| + |(\eta A \nabla u_k+\eta F)_2(Z_k)|\\
&\leq M_k|z_k - Z_k|^\alpha  + cM_k\eps_k^\alpha  \leq   cM_k |y_k|^\alpha\,.
\end{aligned}
\]
As a consequence we have
\[
|I_4|\le c\Big(\frac{r_k}{|y_k|}\Big)^{1-\alpha} \to 0 \quad \text{ as }k \to \infty\,,
\]
as needed. 

\smallskip

Let us now turn to {\bf Case 2}. Here, recall that $\tilde z_k = z_k^0$, where $z_k^0$ is such that $d_k = {\rm dist }(z_k, \partial \Sigma_{\eps_k}^A) = |z_k - z_k^0|\le cr_k$, $ |y_k|/r_k \to \infty$ and $\rho_k^a \to 1$.
Additionally, since $|y_k| \leq |z_k - z_k^0| + |y_k^0| \leq cr_k + |y_k^0|$, 
it follows that both $|y_k^0|/r_k \to \infty$ and $\eps_k/r_k \to \infty$ in this case.

Integrating by parts we find 
\[
\begin{aligned}
I_4 = \frac{r_k^{-\alpha}}{M_k} \int_{ \Omega_k}\nabla\rho_k^a(y)
\cdot (\eta A\nabla u_k+ \eta F)_2(z_k^0)\phi(z)\,dz 
+ \frac{r_k^{-\alpha}}{M_k} \int_{ \partial\Omega_k}\rho_k^a(y)
(\eta A\nabla u_k+ \eta F)(z_k^0)\cdot\nu(z_k^0 + r_k z)\phi(z)\,dz\,,
\end{aligned}
\]
where $\nu(z)$ is the normal vector to $\partial \Sigma_{\eps_k}^A$.  Using \eqref{eq:BC:quantitative:1} and arguing as in  \textbf{Case 1}, we readily obtain 
\[
\Big|\frac{r_k^{-\alpha}}{M_k} \int_{ \Omega_k}\nabla\rho_k^a(y)
\cdot (\eta A\nabla u_k+ \eta F)_2(z_k^0)\phi(z)\,dz\Big| \leq c \|\phi\|_{L^\infty(\R^d)}\Big(\frac{r_k}{|y_k|}\Big)^{1-\alpha} \to 0. 
\]
Next, using that $u_k$ satisfy a conormal boundary condition on $\partial \Omega_k$, we compute 
\[
\begin{aligned}
&( \eta A \nabla u_k+ \eta F)(z_k^0)\cdot\nu(z_k^0 + r_k z) = (\eta A \nabla u_k+ \eta F)(z_k^0)\cdot(\nu(z_k^0 + r_k z)- \nu(z_k^0) )\\
&= (\eta A \nabla u_k+ \eta F)_2(z_k^0)\cdot(N(z_k^0 + r_k z)-N(z_k^0))+ (\eta A \nabla u_k+ \eta F)(z_k^0)\cdot(\tilde \nu(z_k^0 + r_k z)-\tilde \nu(z_k^0))\,,
\end{aligned}
\]
where $N(z)$ and $\tilde \nu(z)$ are as in \eqref{eq:N+tilde:nu}.

As in \cite{AudFioVit24a, SirTerVit21a}, here we split the proof into two parts. Assume to restart the present proof, keeping exactly the same assumptions, but aiming to prove estimates in the form  
\begin{equation}\label{eq:suboptimal}
\|u_\e\|_{C^{1,\alpha'}(B_{1/2}\setminus\Sigma_\e^A)}\le C\big(
\|u_\e\|_{L^{2,a}(B_{1}\setminus\Sigma_\e^A)}
+\|f\|_{L^{p,a}(B_{1})}
+\|F\|_{C^{0,\alpha}(B_1)}
\big)\,,
\end{equation}
with a suboptimal exponent $\alpha' < \alpha$. Then, at this point of the proof, using \eqref{eq:c1:bcestrough}, \eqref{eq:BC:quantitative:1} and Lemma \ref{L:normal}, that is, $\|N\|_{C^{0,1}(\partial\Sigma_{\e_k}^A\cap B_1)}\le c\e_k^{-1}$ and $\|\tilde{\nu}\|_{C^{0,\alpha}(\partial\Sigma_{\e_k}^A\cap B_1)}\le c$, we would obtain 
\begin{equation}\label{eq:badest}
\begin{aligned}
&|( \eta A \nabla u_k+ \eta F)(z_k^0)\cdot\nu(z_k^0 + r_k z) |\leq cM_k\frac{r_k}{\eps_k^{1-\alpha}} + cM_k r_k^{\alpha}\,, 
\end{aligned}
\end{equation}
which implies that 
\[
\begin{aligned}
\Big|\frac{r_k^{-\alpha'}}{M_k} \int_{\partial\Omega_k}\rho_k^a(y)
(\eta A\nabla u_k+ \eta F)(z_k^0)\cdot\nu(z_k^0 + r_k z)\phi(z)\,dz\Big| \leq c\|\phi\|_{L^\infty(\R^d)} \Big(\Big(\frac{r_k}{\eps_k}\Big)^{1-\alpha'} + r_k^{\alpha-\alpha'}\Big) \to 0\,.
\end{aligned}
\]
Therefore, we conclude that $|I_4| \to 0$, as required. 
From this point on, we can complete the proof of the theorem and obtain stable $C^{1,\alpha'}$ estimates. This, in turn, implies a uniform $L^\infty$ bound for $\D u_k$ in $B_{3/4} \setminus\Sigma_{\e_k}$.

We can now return to our previous computation with $\alpha' = \alpha$. Thanks to \eqref{eq:suboptimal} and \eqref{eq:alpha:RR:contradiction2}, we get 
\[
\|\eta A\D u_k + \eta F\|_{L^\infty(B_{1})}\le c k^{-1}M_k\,.
\]
Thus we can improve \eqref{eq:badest} and get 
\[
\begin{aligned}
&|( \eta A \nabla u_k+ \eta F)(z_k^0)\cdot\nu(z_k^0 + r_k z) |\leq cM_k\frac{r_k}{\eps_k^{1-\alpha}} + cM_k k^{-1} r_k^{\alpha}\,. 
\end{aligned}
\]
As a consequence, 
\[
\begin{aligned}
\Big|\frac{r_k^{-\alpha}}{M_k} \int_{\partial\Omega_k}\rho_k^a(y)
(\eta A\nabla u_k+ \eta F)(z_k^0)\cdot\nu(z_k^0 + r_k z)\phi(z)\,dz\Big| \leq c\|\phi\|_{L^\infty(\R^d)} \Big(\Big(\frac{r_k}{\eps_k}\Big)^{1-\alpha} + k^{-1}\Big) \to 0\,,
\end{aligned}
\]
and we can conclude that $|I_4|\to0$.

\smallskip

Finally, we address {\bf Case 3}. Recall that $\tilde z_k = (x_k, 0)$, $\rho_k = |y|$, and $\eps_k \leq c|y_k| \leq c r_k$.  Let us call
\[
I_4' = \frac{r_k^{-\alpha}}{M_k} \int_{\Omega_k}|y|^a
(A(\tilde z_k)P_k + (\eta F)(z_k^0))\cdot\nabla\phi(z)\,dz\,.
\]
We have 
\[
\begin{aligned}
\Big| I_4 - I_4'\Big| &=  \Big|\frac{r_k^{-\alpha}}{M_k} \int_{\Omega_k}|y|^a
((\eta F)(\tilde z_k) -  (\eta F)(z_k^0))\cdot\nabla\phi(z)\,dz\Big|\\
&\leq c\|\nabla\phi\|_{L^1(\R^d, |y|^a dz)}\frac{\|F\|_{C^{0,\alpha}(B_1)}}{M_k}|\tilde z_k - z_k^0|^\alpha r_k^{-\alpha} \leq ck^{-1} \to 0\,,
\end{aligned}
\]
where we used that in {\bf Case 3} it holds
\[
|\tilde z_k - z_k^0| \leq |\tilde z_k - z_k| + |z_k - z_k^0| =  |y_k| + d_k \leq cr_k\,.
\]
Due to our choice of $P_k$, an integration by parts yields
\begin{align*}
    I'_4 = \frac{r_k^{-\alpha}}{M_k} \int_{\Omega_k}|y|^a ( \eta A \nabla u_k + \eta F)_1(z_k^0)\cdot\nabla\phi(z)\,dz 
    = \frac{r_k^{-\alpha}}{M_k} \int_{\partial \Omega_k}|y|^a (\eta A \nabla u_k + \eta F)_1(z_k^0)\cdot \nu(\tilde z_k + r_k z)\phi(z)\,dz\,.
\end{align*}
Recall that at a point $z \in \partial \Sigma_{\eps_k}^A$, the normal vector is given by $\nu(z) = (0, N(z)) + \tilde \nu(z)$, see \eqref{eq:N+tilde:nu}, where $|\tilde \nu(z)| \leq c \eps_k$. Therefore, by \eqref{eq:c1:bcestrough}, we get
\[
\begin{aligned}
|I'_4| \leq \frac{r_k^{-\alpha}}{M_k} \int_{\partial \Omega_k}|y|^a |(\eta A \nabla u_k + \eta F)_1(z_k^0)||\tilde \nu(\tilde z_k + r_k z)||\phi(z)|\,dz
\leq c\Big(\frac{\eps_k}{r_k}\Big) r_k^{1-\alpha} \|\phi\|_{L^\infty(\R^d)}\int_{\supp(\phi) \cap \partial {\Omega_k}}|y|^a \,dz\,.
\end{aligned}
\]
Notice that there exists $\bar{R} >0$ such that 
\[
\supp(\phi) \cap \partial {\Omega_k} = \supp(\phi) \cap \Big\{\sqrt{A_3^{-1}(\tilde z_k + r_kz) y\cdot y} = \frac{\e_k}{r_k}\Big\}\subset (B_{\bar{R}}^{d-n}\times \R^n) \cap \frac{\partial \Sigma_{\eps_k}^A- \tilde z_k}{r_k}\,.
\]
Thus, denoting $E_k =B^{d-n}_{r_k\bar R}(-r_k^{-1}\tilde z_k) \times \R^n$, we have
\[
\int_{\supp(\phi) \cap \partial {\Omega_k}}|y|^a \,dz \leq \int_{(B_{\bar{R}}^{d-n}\times \R^n) \cap \frac{\partial \Sigma_{\eps_k}^A- \tilde z_k}{r_k}}|y|^a \,dz = r_k^{1-d-a}\int_{E_k \cap \partial \Sigma_{\eps_k}^A} |y|^a dz\,.
\]
On the other hand, using Lemma \ref{L:appendix3}, $iv)$ and performing the change of variables $(x, \tau) = \Phi(z)$ (for instance see \cite[\S 11]{Maggi}), we obtain 
\[
\begin{aligned}
\int_{E_k\cap \partial \Sigma_{\eps_k}}|\tau|^a d\sigma(x, \tau) = \int_{E_k\cap \partial \Sigma_{\eps_k}^A}(A_3^{-1}(z)y \cdot y)^{\frac{a}{2}}|\det  (\Pi_{\eps_k} \circ J_\Phi^\top J_\Phi\circ\Pi_{\eps_k})| d\sigma(z) \geq c\int_{E_k\cap \partial \Sigma_{\eps_k}^A}|y|^a d\sigma(z)\,. 
\end{aligned}
\]
Summing up, we have that
\[
\int_{\supp(\phi) \cap \partial {\Omega_k}}|y|^a \,dz \leq cr_k^{1-d-a}\int_{E_k\cap \partial \Sigma_{\eps_k}}|\tau|^a d\sigma(x, \tau) = c r_k^{1-n-a}\int_{B^n_{\eps_k}}|\tau|^a d\sigma(\tau) =  c\Big(\frac{\eps_k}{r_k}\Big)^{a+n-1}\,.
\]
Therefore 
\[
|I'_4| \leq c\Big(\frac{\eps_k}{r_k}\Big)^{a+n} r_k^{1-\alpha} \leq c r_k^{1-\alpha} \to 0 \,,
\]
which implies that $|I_4| \to 0$ also in {\bf Case 3}. Hence, the right hand side of \eqref{eq:sol:wk} vanishes as $k\to\infty$.

\smallskip

From this point on, arguing as in Theorem \ref{T:0:1:alpha:eps}, \emph{i)}, \emph{Step 5}, we obtain that the left hand side of \eqref{eq:sol:vk} satisfies
\[
\int_{\Omega_k\cap \supp(\phi)}\rho_k^a A_k\nabla w_k\cdot \nabla {\phi} \, dz \to \int_{\Omega_\infty\cap \supp(\phi)}\bar{\rho}^a\, \bar {A}\nabla \bar{w}\cdot \nabla \phi\, dz\,,
\]
where
\[
\bar{\rho}(y) = \begin{cases}
    1 &\text{ in }\textbf{Case 1} \text{ and }\textbf{Case 2}\,,\\
    |y| &\text{ in }\textbf{Case 3}\,,
\end{cases}
\]
and $\bar{w}$ belongs to $H^{1}({\Omega}_\infty\cap B_R,\bar{\rho}^adz)$ for every $R>0$. Hence, recalling the definition of the limit domain $\Omega_\infty$, see \eqref{eq:limit:domain*}, we infer that $\bar{w}$ is an entire solution to:
\begin{itemize}
\item[\,] {\bf Case 1}
\[
-\dive(\bar{A}\nabla\bar{w})=0,\quad\text{in }\R^d\,;
\]
\item[\,] \textbf{Case 2}
\[
\begin{cases}
-\dive(\bar{A}\nabla\bar{w})=0,&\text{ in }\Pi_{\bar{e}}\,\\
\bar A\nabla \bar w \cdot \nu = 0 & \text{ on }\partial \Pi_{\bar{e}};
\end{cases}
\]
\item[\,] \textbf{Case 3}, if $\bar{\e}=0$
\[
-\dive(|y|^a\bar{A}\nabla\bar{w})=0, \quad\text{ in }\R^d\,,
\]
\item[\,] \textbf{Case 3}, if $\bar \eps >0$
\[
\begin{cases}
-\dive(|y|^a\bar{A}\nabla\bar{w})=0 & \text{in }\R^d\setminus\Sigma_{\bar{\e}}^{\bar A}\,,\\
\bar{A}\nabla\bar{w}\cdot \nu =0  & \text{on } \Sigma_{\bar{\e}}^{\bar A}\,.
\end{cases}
\]
\end{itemize}

\smallskip

\noindent\emph{Step 7. Liouville theorems and contradiction.}

\smallskip

Since $\bar{w}$ satisfies the growth condition \eqref{eq:growth:c1:w}, with $1+\alpha<\min\{2,\alpha_*\}$, by the assumption \eqref{eq:1:alpha:reg:1}, by invoking the classical Liouville Theorem in \textbf{Case 1}, the Liouville Theorem in an half space with an homogeneous conormal boundary condition (for instance, see \cite[Theorem 1.6]{TerTorVit24a}) in \textbf{Case 2}, and the Liouville Theorem \ref{L:Liouville} in \textbf{Case 3}, we get that $\bar{w}$ must be a linear function: this is a contradiction with the \emph{Step 3}, since $\nabla\bar{w}$ is not constant. Thus, \eqref{eq:c1:alpha:RR} holds with a constant $C$ uniformly bounded in $\eps$. Moreover, recalling \eqref{eq:BC:quantitative:1}, we get that \eqref{eq:BC:stable:1} holds true. The proof is complete.\end{proof}

\begin{remark}\label{R:C2:fails}
In certain ranges of the parameters $a$ and $n$, the exponent $\alpha_*$ given by \eqref{alphastar2} can be greater than $2$ (see Remark \ref{rem:alphastar}). Hence, the regularity of solutions to \eqref{generalPDE} is expected to be higher than $C^{1,\alpha}$ ($C^{2,\alpha}$ or even more).
In this remark, we give a simple example which shows that our method fails to prove higher order Schauder estimates which are stable in $\e$ with respect to domain perforation. 

For $a+n\not=2$ such that $\alpha_*>2$, let us consider the family of functions 
\[
u_\e(x,y):=(a+n)x_1^2 - {|y|^2} + \frac{2}{2-a-n}\e^{a+n}{|y|^{2-a-n}},
\]
which are solutions to 
\[
\begin{cases}
    -\dive(|y|^a \D u_\e)=0, & \text{in }B_1\setminus\Sigma_\e,\\
    \D u_\e \cdot \nu =0, & \text{on }\partial \Sigma_\e \cap B_1,
\end{cases}
\]
and satisfy $\|u_\e\|_{L^\infty(B_1\setminus\Sigma_\e)}\le c$.
Fixed $i,j \in \{1,\dots,n\}$ such that $i\not =j$, we compute 
$$\partial_{y_i y_j}^2 u_\e (x,y) = -2(a+n)\e^{a+n}|y|^{-2-a-n}y_i y_j\,.$$
Let us fix the points
$$z_1 = \frac{\e}{\sqrt{2}} (e_{y_i} + e_{y_j}), \quad 
 z_2 = \frac{\e^\beta}{\sqrt{2}} (e_{y_i} + e_{y_j}),$$ 
where $\beta \in (0,1)$. One has that $z_1,z_2 \in \overline{ B_{1/2}\setminus \Sigma_\e}$ and, for every $\alpha \in (0,1)$, we have
\begin{align*}
    \frac{|\partial_{y_i y_j}^2 u_\e (z_1)-\partial_{y_i y_j}^2 u_\e (z_2)|}{|z_1-z_2|^\alpha} = (a+n)\frac{1-\e^{(a+n)(1-\beta)}}{\eps^{\alpha \beta}|1-\eps^{1-\beta}|^{\alpha}}
    \to \infty\,,\quad\text{as }\e \to 0\,,
\end{align*}
since $1-\beta>0$ and $a+n>0$.
Then, we have proved that $\e$-uniform $C^{2,\alpha}$ estimates fail.

\end{remark}

\section{A priori estimates and proofs of the main results}\label{sec:7}

In this section we prove a priori regularity estimates, which imply stable estimates with respect to standard smoothing of data. This leads to the proof of the main Theorems \ref{T:0:alpha} and \ref{T:1:alpha}.

\subsection{A priori estimates}

\begin{Proposition}[A priori estimates]\label{P:a:priori}
    Let $a+n>0$. Let $A$ be a uniformly elliptic matrix satisfying \eqref{eq:unif:ell} and \eqref{eq:unif:ell:loc} with constants $0<\lambda\leq\lambda_*\leq\Lambda_*\leq\Lambda$.
    Let $\alpha_*=\alpha_*(n,a,\lambda_*/\Lambda_*)$ be defined as in \eqref{alphastar2}.   
    Then the following points hold true:

    \begin{itemize}
        \item[i)] Let $p>(d+a_+)/2$, $q>d+a_+$. Let $\alpha\in(0,1)$ satisfying \eqref{eq:0:alpha:reg1}. Let $A$ be $C^{0}$ with $\|A\|_{C^{0,\omega}(B_1)}\le L$ for some given modulus of continuity $\omega$, $f\in L^{p,a}(B_1)$ and $F\in L^{q,a}(B_1)^d$. Let $u\in C^{0,\alpha}(B_{1})$ be a weak solution to \eqref{generalPDE} in $B_1$.
        Then, there exist a constant $C>0$ depending only on  $d$, $n$, $a$, $\lambda$, $\Lambda$, $p$, $q$ $\alpha$ and $L$ such that
        \begin{equation*}
        \|u\|_{C^{0,\alpha}(B_{1/2})}\le C\big(
        \|u\|_{L^{2,a}(B_{1})}
        +\|f\|_{L^{p,a}(B_{1})}
        +\|F\|_{L^{q,a}(B_{1})}
        \big).
        \end{equation*}             
        \item[ii)] Let $p>d+a_+$. Let assume that $\alpha_*>1$. Let $\alpha\in(0,1)$ satisfying \eqref{eq:1:alpha:reg:1}. Let $A$ be $C^{0,\alpha}$ with $\|A\|_{C^{0,\alpha}(B_1)}\le L$, $f\in L^{p,a}(B_1)$ and $F\in C^{0,\alpha}(B_1)$. Let $u\in C^{1,\alpha}(B_{1})$ be a weak solution to \eqref{generalPDE} in $B_1$ such that 
    \begin{equation}\label{eq:c1:boundary:condition*}
    (A\D u + F)\cdot e_{y_i}=0,\quad\text{ on } \Sigma_0\cap B_{1},\quad\text{for every } i=1,\dots,n.
    \end{equation}
    Then, there exist a constant $C>0$ depending only on  $d$, $n$, $a$, $\lambda$, $\Lambda$, $p$, $\alpha$ and $L$ such that
    \begin{equation*}
    \|u\|_{C^{1,\alpha}(B_{1/2})}\le C\big(
    \|u\|_{L^{2,a}(B_{1})}
    +\|f\|_{L^{p,a}(B_{1})}
    +\|F\|_{C^{0,\alpha}(B_1)}
    \big).
\end{equation*}
\end{itemize}

\end{Proposition}

We provide the proof only for the estimates in $ C^{1,\alpha} $-spaces, as the proof for the other case follows the same argument and is easier to establish.

\begin{proof}[Proof of Proposition \ref{P:a:priori}, ii)]
The proof is similar to that of Theorem \ref{T:0:1:alpha:eps}, \emph{ii)}, so we omit some details.

By contradiction, let us suppose that there exist sequences $\{u_k\}_k$, $\{f_k\}_k$, $\{F_k\}_k$ and $\{A_k\}_k$, such that $A_k \in C^{0,\alpha}$ with $\|A_k\|_{C^{0,\alpha}(B_1)}\le L$, $f_k\in L^{p,a}(B_1)$, $F_k\in C^{0,\alpha}(B_1)$, and every $u_k$ is a weak solution to 
\[
-\dive(|y|^a A_k\D u_k) = |y|^a f_k +\dive(|y|^aF_k), \quad \text{ in } B_1 ,
\]
which satisfies the boundary condition \eqref{eq:c1:boundary:condition*} and 
\begin{equation*}
\|u_k\|_{C^{1,\alpha}(B_{1/2})}> k\big(
\|u_k\|_{L^{2,a}(B_{1})}
+\|f_k\|_{L^{p,a}(B_{1})}
+\|F_k\|_{C^{0,\alpha}(B_1)}
\big).
\end{equation*}
Let us fix a smooth cut-off function $\eta\in C_c^\infty(B_{3/4})$ such that $\eta=1$ in $B_{1/2}$ and $0\le \eta \le 1$. Then, arguing as in \emph{Step 1} in Theorem \ref{T:0:1:alpha:eps}, \emph{ii)}, we have 
\begin{equation*}
M_k=[\nabla (\eta u_k)]_{C^{0,\alpha}(B_{1})}\geq [\nabla u_k]_{C^{0,\alpha}(B_{1/2})} \geq ck\big(
\|u_k\|_{L^{2,a}(B_{1})}
+\|f_k\|_{L^{p,a}(B_{1})}
+\|F_k\|_{C^{0,\alpha}(B_1)})\,.
\end{equation*}
Take two sequences of points $z_k=(x_k,y_k),\hat{z}_k=(\hat{x}_k,\hat{y}_k) \in B_1$ such that 
\[
\frac{|\D (\eta u_k)(z_k)-\D (\eta u_k)(\hat z_k)|}{|z_k-\hat{z}_k|^\alpha}\ge \frac12 M_k,
\]
and define $r_k=|z_k-\hat{z}_k|$. Without loss of generality, we can assume that $z_k \in B_{3/4}$.
From now on, we distinguish two cases:
\smallskip
\begin{itemize}
    \item[$\,$] \textbf{Case 1:} 
    $\quad\displaystyle \frac{|y_k|}{r_k}\to \infty \text{ as } k\to\infty\,,$

\item[$\,$] \textbf{Case 2:}  $\quad\displaystyle \frac{|y_k|}{r_k}\le c\, \text{ uniformly in } k\,.$
\end{itemize}
\smallskip
Unlike the proofs of the stable estimates in Section \ref{sec:6}, here we do not introduce any perforation around $ \Sigma_0 $. Consequently, we have only two blow-up regimes: in {\bf Case 1}, the blow-up scale does not capture $\Sigma_0$, whereas in {\bf Case 2}, it does.

Let $\zeta_k=(x_k,0)$ be the projection of $z_k$ onto $\Sigma_0$. Let us define
\[
\tilde{z}_k=(\tilde{x}_k,\tilde{y}_k)=\begin{cases}
z_k, & \text{ in }  \textbf{Case 1},  \\
\zeta_k, & \text{ in } \textbf{Case 2},
\end{cases}\qquad \xi_k=\frac{z_k-\tilde{z}_k}{r_k}\,,
\]
the sequence of rescaled domains
\[
\Omega_k=\frac{B_1-\tilde{z}_k}{r_k},
\]
and the sequences of functions  
\[
\begin{gathered}
v_k(z)= \frac{
\eta ( \tilde{z}_k+r_kz) \big(u_k( \tilde{z}_k+r_kz) - u_k (\tilde{z}_k)\big) - \eta(z_k) \D u_k({z}_k)\cdot r_k (z-\xi_k)
}{ r_k^{1+\alpha} M_k}\,, \quad \text{ for }z\in \Omega_k\,,\\
w_k(z)=\frac{\eta(\tilde{z}_k) (u_k( \tilde{z}_k+r_kz) -  u_k (\tilde{z}_k)) - (\eta \D  u_k) (\tilde{z}_k)\cdot r_k z}{ r_k^{1+\alpha}M_k}\,, \quad \text{ for }z\in \Omega_k\,,
\end{gathered}
\]
and set the limit blow-up domain $\Omega_\infty$ as in \eqref{eq:limit:domain:def}. 

\smallskip

By arguing as in Theorem \ref{T:0:1:alpha:eps}, \emph{ii), Step 2} one has that $ [v_k]_{C^{1,\alpha}(K\cap \Omega_k)}\le2 $ and
$\|v_k\|_{C^{1,\alpha}(K)}\le c(K)$ for every compact set $K\subset \Omega_\infty$. Therefore, by applying the Arzel\'a-Ascoli theorem, we infer that $v_k \to \bar{v}$ in $C^{1,\gamma}_{\rm loc}({\Omega}_\infty)$ for every $\gamma\in (0,\alpha)$ and
$$|\bar{v}| \le c (1+|z|^{1+\alpha}).$$
Moreover, employing a similar argument as in Theorem \ref{T:0:1:alpha:eps}, \emph{ii), Step 3}, it follows that also $w_k \to \bar{v}$ uniformly on compact set of $\Omega_\infty$.
Next, repeating the argument of Theorem \ref{T:0:1:alpha:eps}, \emph{ii), Step 4, Step 5}, we conclude that $\D \bar{v}$ is not constant and $r_k\to0$, which immediately implies that $\Omega_\infty=\R^d$ in both \textbf{Case 1} and \textbf{Case 2}.

\smallskip

Finally, we prove that $\bar{v}$ is an entire solution to a homogeneous equation with constant coefficients.
Since $\|A_k\|_{C^{0,\alpha}(B_1)}\le L$, we have that $A_k(\tilde{z}_k+r_k z)\to A(\bar{z})=\bar{A}$, which is a constant matrix satisfying \eqref{eq:unif:ell} and $\bar{z}=\lim_{k\to\infty}\tilde{z}_k$.
Let us define
\[
\rho_k^a(y)=\begin{cases}
\displaystyle{\frac{|{y}_k+r_k y|^a}{|y_k|^a}}, & \text{ in } \textbf{Case 1},\\
|y|^a, &  \text{ in } \textbf{Case 2}.
\end{cases}
\]
Fixed $\phi\in C_c^\infty(\R^d)$, we have that
\[
\int_{\Omega_k}\rho_k^a(y)A_k(\tilde{z}_k+r_k z)\D w_k(z)\cdot \D \phi (z)dz = I_1-I_2-I_3-I_4\,,
\]
where
\[
\begin{aligned}
&I_1 =  \frac{r_k^{1-\alpha} }{M_k} \int_{\Omega_k}\rho_k^a(y) \eta(\tilde{z}_k)f_k(\tilde{z}_k+r_k z)\phi(z)dz\\
&I_2 =  \frac{r_k^{-\alpha}}{M_k} \int_{\Omega_k}\rho_k^a(y) \eta(\tilde{z}_k)\big(
F_k(\tilde{z}_k+r_k z)-F_k(\tilde{z}_k) \big) \cdot\D \phi (z)dz\\
&I_3 =  \frac{r_k^{-\alpha}}{M_k} \int_{\Omega_k}\rho_k^a(y)
\eta(\tilde{z}_k)\big(A_k(\tilde{z}_k+r_k z)-A_k(\tilde{z}_k)\big)\D u_k(\tilde{z}_k)\cdot\D \phi(z)  dz  \\
&I_4 = \frac{r_k^{-\alpha}}{M_k} \int_{\Omega_k}\rho_k^a(y)\eta(\tilde{z}_k)
\big(
A_k\D u_k + F_k
\big) (\tilde{z}_k)\cdot\D\phi(z)dz\,.
\end{aligned}
\]
The terms $I_1$, $I_2$ vanishes as $k\to\infty$ exactly as in Theorem \ref{T:0:1:alpha:eps}, \emph{Step 6}.

To show that the term $I_3$ vanishes, we proceed in two steps, as in the proof of Theorem \ref{T:0:1:alpha:eps}, \emph{ii), Step 6}. First, we establish local $C^{1,\alpha'}$ estimates for some suboptimal $\alpha'\in(0,\alpha)$. Once this is done, we can then obtain the regularity with the optimal exponent $\alpha$, using as additional information a $L^\infty$ bound for the gradient of the solutions.  
Hence, in the first step, we have 
\begin{align*}
|I_3|&= \Big|
\frac{r_k^{-\alpha'}}{M_k} \int_{\Omega_k}\rho_k^a(y)
\eta(\tilde{z}_k)(A_k(\tilde{z}_k+r_k z)-A_k(\tilde{z}_k))\D u_k(\tilde{z}_k)\cdot\D \phi(z)  dz  
 \Big|\\
 &\le  \frac{ c L r_k^{\alpha-\alpha' }( \|\eta u_k\|_{L^\infty(B_1)} + \| \D (\eta u_k)\|_{L^\infty(B_1)}) }{M_k}  \le c L r_k^{\alpha-\alpha'} \to 0, \quad \text{as }k \to \infty.
\end{align*}
Next, restarting the proof with the optimal exponent $\alpha$, a priori estimates with a suboptimal exponent imply that $\|\D (\eta u_k)\|_{L^\infty(B_1)} \le M_k/k$ and we can conclude in the following way
\[
|I_3| \le \frac{ c L( \|\eta u_k\|_{L^\infty(B_1)} + \| \D (\eta u_k)\|_{L^\infty(B_1)}) }{M_k} \le cL k^{-1} \to 0, \quad \text{as }k \to \infty.
\]

Finally, we show that the term $I_4$ goes to zero. 
Let us consider the \textbf{Case 1}, recalling that $\tilde{z}_k=z_k$, $r_k/|y_k|\to0$, $\rho_k^a\to 1$ and $\zeta_k=(x_k,0) \in \Sigma_0 \cap B_1$.
Arguing as in the proof of Theorem \ref{T:0:1:alpha:eps}, \emph{ii), Step 6}, and using the boundary condition \eqref{eq:c1:boundary:condition*}, we have that
\[
|(\eta A_k\D u_k+\eta F_k)_2({z}_k)|= |(\eta A_k\D u_k+\eta F_k)_2({z}_k)-(\eta A_k\D u_k+\eta F_k)_2({\zeta}_k)|\le c M_k|y_k|^\alpha,
\]
and thus
\[
 |\D \rho_k^a(y) \cdot (\eta A_k\D u_k+\eta F_k)({z}_k)
| \le c \frac{r_k}{|y_k|^{1-\alpha}}M_k.
\]
Hence, by using the divergence theorem, one has that 
\begin{align*}
|I_4|= \frac{r_k^{-\alpha}}{M_k} 
\Big| \int_{\Omega_k}\D\rho_k^a(y)\cdot (\eta A_k\D u_k + \eta F_k
) ({z}_k) \phi(z)  dz  \Big|\le c\Big(\frac{r_k}{|y_k|}\Big)^{1-\alpha}\to0,\quad \text{as }k\to\infty.
\end{align*}

In \textbf{Case 2}, 
since $u_k$ satisfies \eqref{eq:c1:boundary:condition*} and $\tilde{z}_k=(x_k,0)$,
the divergence theorem allows us to conclude directly that $I_4=0$. In fact,
\[
\int_{\Omega_k}\rho_k^a(y)
(
A_k\D u_k + F_k
) (\tilde{z}_k)\cdot\D\phi(z)dz = - 
\int_{\Omega_k}\D \rho_k^a(y) \cdot 
(
A_k\D u_k + F_k ) (\tilde{z}_k)\phi(z)dz=0.
\]
From this point on, arguing as in Theorem \ref{T:0:1:alpha:eps} \emph{i), Step 6}, we obtain that 
\[
\int_{\Omega_k}\rho_k^a(y) A_k(\tilde{z}_k+r_k z)\D w_k (z)\cdot \D {\phi}(z)dz\to \int_{\R^d}\bar{\rho}^a(y) \bar{A}\D \bar{v}(z)\cdot \D \phi(z)dz,
\]
where 
\[
\bar{\rho}(y)=\begin{cases}
1, & \text{ in }\textbf{Case 1},\\
|y|, & \text{ in }\textbf{Case 2}.
\end{cases}
\]
Then, we have proved that $\bar{v}$ is an entire solution
to
$$-\dive(\bar{A}\D\bar{v})=0,\quad\text{in }\R^d,\quad \text{in } \textbf{Case 1}$$
and $\bar{v}$ is an entire solution to 
$$-\dive(|y|^a\bar{A}\D\bar{v})=0, \quad\text{in }\R^d,\quad \text{in } \textbf{Case 2}.$$
Finally, as in Theorem \ref{T:0:1:alpha:eps}, \emph{ii), Step 7}, by invoking appropriate Liouville type theorems we get a contradiction and the thesis follows.
\end{proof}

\subsection{Proof of Theorems \ref{T:0:alpha} and \ref{T:1:alpha}}

We will prove Theorem \ref{T:1:alpha} only, as the proof of Theorem \ref{T:0:alpha} follows by a similar argument.

\begin{proof}[Proof of Theorem \ref{T:1:alpha}]

We divide the proof into two steps.

\smallskip

\noindent
\emph{Step 1.} First, we prove that if the matrix $A\in C^{1,\alpha}(B_1)$, then $u\in C^{1,\alpha}(B_{3/4})$ and satisfies \eqref{eq:c1:boundary:condition} in $B_{3/4}$.

Let $u$ be a weak solution to \eqref{generalPDE}. By applying Lemma \ref{L:approximation} there exists a family of functions $\{u_{\e}\}_{0 < \e\ll1}$, such that every $u_\e$ is a weak solution to
     \begin{equation*}
    \begin{cases}
        -\dive(|y|^a A\D u_\e)=|y|^a f +\dive(|y|^a F), & \text{ in } B_{4/5}\setminus\Sigma_\e^A\\
        (A\D u_\e + F)\cdot \nu = 0, & \text{ on } \partial \Sigma_\e^A \cap B_{4/5},
    \end{cases}
    \end{equation*}
and a sequence $\e_k\to0$ such that $u_{\e_k} \to u$ in $H^{1}_{\rm loc}(B_{4/5}\setminus\Sigma_0)$.

By applying Theorem \ref{T:0:1:alpha:eps}, \emph{ii)} to the sequence $\{u_{\e_k}\}$, combined with \eqref{eq:stima:u:e} and \eqref{eq:caccioppoli}, we get
\begin{equation*}
\|u_{\e_k}\|_{C^{1,\alpha}(B_{3/4}\setminus\Sigma_{\e_k}^A)}\le c \big(
\|u\|_{L^{2,a}(B_{1})}
+\|f\|_{L^{p,a}(B_{1})}
+\|F\|_{C^{0,\alpha}(B_1)}
\big),
\end{equation*}
for some $c>0$ depending only on $d$, $n$, $a$, $\lambda$, $\Lambda$, $p$, $\alpha$ and $L$. By using the Arzel\'a-Ascoli Theorem and the a.e. convergence $u_{\e_k}\to u$, we get that $u_{\e_k}\to u$ in $C_{\rm loc}^{1,\gamma}(B_{3/4}\setminus\Sigma_0)$, for every $\gamma\in(0,\alpha)$. Moreover, by taking $z,z'\in B_{3/4}\setminus\Sigma_0$ such that $z\not=z'$, we have that
\begin{align*}
&\frac{|\D u(z)-\D u(z')|}{|z-z'|^\alpha}=\lim_{\e_k\to0^+}\frac{|\D u_{\e_k}(z)-\D u_{\e_k}(z')|}{|z-z'|^\alpha}\le c \big(
\|u\|_{L^{2,a}(B_{1})}
+\|f\|_{L^{p,a}(B_{1})}
+\|F\|_{C^{0,\alpha}(B_1)}
\big).
\end{align*}
Hence, by taking the supremum, we get
\begin{align*}
    [\D v]_{C^{0,\alpha}(B_{3/4}\setminus\Sigma_0)}=\sup_{z,z' \in B_{3/4}\setminus\Sigma_0}\frac{|\D u(z)-\D u(z')|}{|z-z'|^\alpha}
    \le c \big(
\|u\|_{L^{2,a}(B_{1})}
+\|f\|_{L^{p,a}(B_{1})}
+\|F\|_{C^{0,\alpha}(B_1)}
\big).
\end{align*}
By continuity, we can extend $v$ in the whole $B_{3/4}$ in such a way that $[\D v]_{C^{0,\alpha}(B_{3/4})}=[\D v]_{C^{0,\alpha}(B_{3/4}\setminus\Sigma_0)}$, hence, we have that $u\in C^{1,\alpha}(B_{3/4})$.

\smallskip

Further, let us prove that $u$ satisfies the boundary condition \eqref{eq:c1:boundary:condition}.
Fix $z=(x,y)\in B_{3/4} \setminus \Sigma_0$ and let $z^\e$ be a chosen projection of $z$ onto $\partial\Sigma_\e^A\cap B_{3/4}$. For every $i=1,\dots,n$, by using \eqref{eq:BC:stable:1}, we get
\[
\begin{aligned}
    \big|(A\D u+F)(z) &\cdot  e_{y_i}\big|\le \big|(A\D u)(z)-(A\D u_{\e})(z)\big| + \big|(A\D u_{\e} + F)(z)-(A\D u_{\e}+ F)(z^{\e})\big|\\
   &  +\big|(A\D u_{\e}+F)(z^{\e})\cdot e_{y_i} \big|\\
    &\le \Lambda \big|\D u(z)-\D u_\e(z)\big|+
    c\big(
\|u\|_{L^{2,a}(B_{1})}
+\|f\|_{L^{p,a}(B_{1})}
+\|F\|_{C^{0,\alpha}(B_1)}
\big)\big(|z-z^\e|^\alpha + \e^\alpha\big).
\end{aligned}
\]
By taking the limit as $\e \to 0^+$ in the previous inequality, we get
\[
\big|(A\D u+F)(z) \cdot  e_{y_i}\big| \le c\big(
\|u\|_{L^{2,a}(B_{1})}
+\|f\|_{L^{p,a}(B_{1})}
+\|F\|_{C^{0,\alpha}(B_1)}\big)|y|^\alpha,
\]
therefore, by taking $|y|\to 0$ it follows that $(A\D u+F)(x,0) \cdot  e_{y_i} = 0$, that is, \eqref{eq:c1:boundary:condition} holds true.

\smallskip
\noindent
\emph{Step 2.} Finally, we prove that if $A\in C^{0,\alpha}(B_1)$, then  $u\in C^{1,\alpha}(B_{1/2})$, satisfies \eqref{eq:stima:c1} and \eqref{eq:c1:boundary:condition}.

Let $u$ be a weak solution to \eqref{generalPDE}. 
For $\delta>0$, let $\{\rho_\delta\}_{\delta>0}$ be a family of smooth mollifiers and let us define $A_\delta=A*\rho_\delta$. Choosing $\delta$ small enough, we have that $A_\delta \in C^\infty(B_{4/5})$, satisfies \eqref{eq:unif:ell} and $\|A_\delta\|_{C^{0,\alpha}(B_{4/5})}\le \|A\|_{C^{0,\alpha}(B_{1})}$. By using the approximation Lemma \ref{L:approximation:smooth}, we have that there exists a family $\{u_\delta\}$ such that every $u_\delta$ is a weak solution to 
\[
-\dive(|y|^a A_\delta \D u_\delta)=|y|^a {f}+ \dive(|y|^a {F}), \text{ in }B_{4/5},
\]
satisfies
\[
\|u_\delta\|_{H^{1,a}(B_{4/5})}\le c(\|u\|_{H^{1,a}(B_1)}+\|{f}\|_{L^{2,a}(B_1)}+\|{F}\|_{L^{2,a}(B_1)}),
\]
and, up to consider a subsequence, $u_\delta\to u$ in $H^{1,a}(B_{4/5})$.

By applying \emph{Step 1}, we have that $u_\delta\in C^{1,\alpha}(B_{3/4})$ and satisfies \eqref{eq:c1:boundary:condition}. Hence, by using Proposition \ref{P:a:priori}, \emph{ii)},
it follows that
\begin{equation}\label{eq:v_delta:schauder}
\|u_\delta\|_{C^{1,\alpha}(B_{1/2})}\le c\big(
\|u\|_{L^{2,a}(B_{1})}
+\|f\|_{L^{p,a}(B_{1})}
+\|F\|_{C^{0,\alpha}(B_1)}
\big),
\end{equation} 
for some $c>0$, depending only on $d$, $n$, $a$, $\lambda$, $\Lambda$, $p$, $\alpha$ and $L$.
Hence, we can apply the Arzel\'a-Ascoli theorem to get that $u_\delta\to u$ in $C^{1,\gamma}(B_{1/2})$ (for every $\gamma\in(0,\alpha)$) and passing to the limit as $\delta\to 0 $ in \eqref{eq:v_delta:schauder} we obtain that $u$ satisfies \eqref{eq:stima:c1}. Moreover, by \emph{Step 1} every $u_\delta$ satisfies \eqref{eq:c1:boundary:condition}, so 
$u$ also satisfies \eqref{eq:c1:boundary:condition}, and our statement follows.
\end{proof}

\section{Equations degenerating on curved manifolds}\label{sec:8}

This section is devoted to the extension of Theorems \ref{T:0:alpha} and \ref{T:1:alpha} to a class of equations whose weights are degenerate/singular on lower-dimensional curved manifolds.

Let $2\le n<d$ and consider a $(d-n)$-dimensional $C^1$-manifold $\Gamma$ which is locally embedded in $\R^d$ and parametrized by $\varphi\in C^1(\Sigma_0\cap B_1;\R^n)$, in the sense that, up to perform a dilation, a translation and a rotation, one has
\begin{equation}\label{eq:parametrization}
\Gamma\cap B_1=\{(x,y) \in \R^d \mid y =\varphi(x)\}\cap B_1,\quad 0\in \Gamma, \quad\varphi(0)=0,\quad J_\varphi(0)=0.
\end{equation}
Let us define the diffeomorphism 
\begin{equation}\label{eq:diffeomorphism}
    \Phi(x,y):=(x,y+\varphi(x)),
\end{equation}
which satisfies $\Phi(\Sigma_0\cap B_R)\subset\Gamma\cap B_1$ for some $R>0$, $\Phi(0)=\Phi^{-1}(0)=0$ and the Jacobian associated to $\Phi$ is
\[
J_\Phi(x,y)=\left(\begin{array}{cc}
       \mathbb{I}_{d-n} & {\bf 0}  \\ 
       J_\varphi(x) & \mathbb{I}_{n}\\ 
  \end{array}\right), \quad  |\det J_\Phi|\equiv1, \quad J_\Phi(0)=\mathbb{I}_d.
\]
Next, we define the class of admissible weights with respect to the parametrization $\varphi$.

\begin{Definition}[$\alpha$-defining function]\label{D:admissible:weight}
    Let $2\le n <d$, $\alpha\in[0,1)$ and $\varphi\in C^{1,\alpha}(\Sigma_0\cap B_1;\R^n)$ be a parametrization in the sense of \eqref{eq:parametrization}.
We say that $\delta$ is an $\alpha$-defining function with respect to the parametrization $\varphi$ if $\delta \in C^{0,1}(B_1)$ and the two following condition holds true:
\begin{itemize}
\item[i)] there exist two constants $0<c_0\le c_1$ such that
\[
c_0\le \frac{ \delta}{\dist_\Gamma}  \le c_1;
\]
\item[ii)] 
\begin{equation}\label{eq:rho:delta}
\tilde \delta (x,y):=\frac{\delta(\Phi(x,y))}{|y|}\in C^{0,\alpha}(B_1).
\end{equation}
\end{itemize}
\end{Definition}

\begin{remark}\label{R:tilde:A}
Let us consider a variable coefficient matrix $A\in C^0(B_1)$ satisfying the global uniform ellipticity condition \eqref{eq:unif:ell} with ellipticity constants $0<\lambda\leq\Lambda$ and satisfying the restricted-to-$\Gamma$ uniform ellipticity condition
\begin{equation}\label{eq:unif:ell:Gamma}
    \lambda_*|\xi|^2\le A(z)\xi\cdot\xi\le\Lambda_*|\xi|^2,\qquad\text{for a.e. } z\in B_1\cap\Gamma \text{ and for all } \xi\in \R^d,
\end{equation}
for some constants $0<\lambda\leq\lambda_*\leq\Lambda_*\leq\Lambda$. Given a $0$-defining function $\delta$, let $\tilde{\delta}$ as in \eqref{eq:rho:delta} and define the matrix
\begin{equation}\label{eq:tilde:A}
    \Tilde{A}:=\tilde{\delta}^a (J_\Phi^{-1})(A\circ \Phi) (J_\Phi^{-1})^\top,
\end{equation}
which is symmetric, continuous and satisfies the global uniform ellipticity condition \eqref{eq:unif:ell}. The latter condition follows by $i)$ and $ii)$ in Definition \ref{D:admissible:weight}, which implies $\tilde{\delta}\ge c >0$. In addition, since $J_\Phi(0)=\mathbb{I}$, one has that $\tilde{A}(0) = \tilde{\delta}^a(0) A(0)$ has ellipticity constants restricted to $\Sigma_0$ given by
\[
0<\tilde{\delta}^a(0)\lambda_*\leq\tilde{\delta}^a(0)\Lambda_*.
\]
By using the continuity of $\tilde A$, for every $\sigma\in(0,1)$ there exists a small radius $\bar{r}\in(0,1)$ such that $\tilde A$ has ellipticity constants restricted to $B_{\bar{r}}\cap \Sigma_0$ given by 
\[
 \tilde\delta^a(0)\lambda_* - \sigma \le  \tilde\delta^a(0)\Lambda_* + \sigma.
\]
That is, in a local scale, $\tilde{A}$ and $A$ have the same ellipticity ratio restricted to the thin space, up to an error
\[
\frac{\tilde\delta^a(0)\lambda_* - \sigma}{\tilde\delta^a(0)\Lambda_* + \sigma} = \frac{\lambda_*}{\Lambda_*} + \tilde \sigma,
\]
where $\tilde \sigma>0$ is arbitrarily small.

\end{remark}

As done in Section \ref{S:sobolev:spaces}, we can define the spaces
$L^p(B_1,\delta^a)$, $L^p(B_1,\delta^a)^d$ and $H^1(B_1,\delta^a)$. Then, we provide the definition of weak solution to \eqref{eq:weak:solution:curve}.

\begin{Definition}\label{D:weak:solution:curve}
    Let $a+n>0$, $A$ satisfying \eqref{eq:unif:ell} and $\delta$ be a $0$-defining function in the sense of Definition \ref{D:admissible:weight}. Let $f\in L^{(2^*_a)'}(B_1,\delta^a)$ and $F\in L^{2}(B_1,\delta^a)^d$. We say that $u$ is a weak solution to \eqref{eq:weak:solution:curve}
if $u\in H^{1}(B_1,\delta^a)$ and satisfies
\begin{equation*}
\int_{B_1}\delta^a A\D u\cdot \D\phi dz= \int_{B_1}\delta^a (f\phi-F\cdot\D\phi) dz,
\end{equation*}
for every $\phi\in C_c^\infty(B_1).$
\end{Definition}

The main results of this section are the following corollaries, which establish local $C^{0,\alpha}$ and $C^{1,\alpha}$ estimates for weak solutions to \eqref{eq:weak:solution:curve}.

\begin{Corollary}\label{C:0:alpha:curve}
    Let $a+n>0$, $p>{(d+a_+)}/{2}$ and $q>d+a_+$. Let $f\in{L^{p}(B_{1},\delta^a)} $, $F\in{L^{q}(B_{1},\delta^a)}^d$, $A \in C^{0}(B_1)$ be a continuous matrix satisfying \eqref{eq:unif:ell} and \eqref{eq:unif:ell:Gamma} with constants $0<\lambda\leq\lambda_*\leq\Lambda_*\leq\Lambda$, $\alpha_*=\alpha_*(n,a,\lambda_*/\Lambda_*)$ defined in \eqref{alphastar2} and $\alpha$ satisfying \eqref{eq:0:alpha:reg1}.  
    Let $\varphi\in C^1(\Sigma_0\cap B_1;\R^n)$ be a parametrization in the sense of \eqref{eq:parametrization}. Let $\delta$ be a $0$-defining function with respect to $\varphi$ in the sense of Definition \ref{D:admissible:weight} and set $\tilde{\delta} \in C^{0}(B_1)$ as in \eqref{eq:rho:delta}.
    
Let $u$ be a weak solution to \eqref{eq:weak:solution:curve} in the sense of Definition \ref{D:weak:solution:curve} and
let us suppose that there exists a modulus of continuity $\omega$ such that
$$\|A\|_{C^{0,\omega} (B_1)}+\|\tilde \delta\|_{C^{0,\omega} (B_1)}+\|\varphi\|_{C^{1,\omega} (\Sigma_0\cap B_1;\R^n)}\le L.$$
Then, there exists $\bar{r}\in (0,1/2)$, depending only on $\alpha$ and $L$, such that $u\in C^{0,\alpha}(B_{\bar{r}})$ and there exists a constant $c>0$, depending only on $d$, $n$, $a$, $ \lambda$, $\Lambda$, $p$, $q$, $\alpha$ and $L$ such that
\begin{equation*}
\|u\|_{C^{0,\alpha}(B_{\bar{r}})}  \le c\big(
\|u\|_{L^{2}(B_{1},\delta^a)}
+\|f\|_{L^{p}(B_{1},\delta^a)}
+\|F\|_{L^{q}(B_{1},\delta^a)}
\big).
\end{equation*}
\end{Corollary}
    
\begin{Corollary}\label{C:1:alpha:curve}
Let $a+n>0$ and $p>{d+a_+}$. Let $A$ be a matrix satisfying \eqref{eq:unif:ell} and \eqref{eq:unif:ell:Gamma} with constants $0<\lambda\leq\lambda_*\leq\Lambda_*\leq\Lambda$, $\alpha_*=\alpha_*(n,a,\lambda_*/\Lambda_*)$ defined in \eqref{alphastar2} and $\alpha$ satisfying \eqref{eq:1:alpha:reg:1}.    
    Let $\varphi\in C^{1,\alpha}(\Sigma_0\cap B_1;\R^n)$ be a parametrization in the sense of \eqref{eq:parametrization}. Let $\delta$ be a $\alpha$-defining function with respect to $\varphi$ in the sense of Definition \ref{D:admissible:weight} and set $\tilde{\delta} \in C^{0,\alpha}(B_1)$ as in \eqref{eq:rho:delta}.
Let us suppose that $A$ is $C^{0,\alpha}$-continuous, $f\in{L^{p}(B_{1},\delta^a)} $, $F\in C^{0,\alpha}(B_1) $. 

Let $u$ be a weak solution to \eqref{eq:weak:solution:curve} in the sense of Definition \ref{D:weak:solution:curve} and let us suppose that 
$$\|A\|_{C^{0,\alpha} (B_1)}+\|\tilde \delta\|_{C^{0,\alpha} (B_1)}+\|\varphi\|_{C^{1,\alpha} (\Sigma_0\cap B_1;\R^n)}\le L.$$
Then, there exists $\bar{r}\in(0,1/2)$, depending only on $\alpha$ and $L$, such that $u\in C^{1,\alpha}(B_{\bar{r}})$ and there exists a constant $C>0$, depending only on $d$, $n$, $a$, $ \lambda$, $ \Lambda$, $p$, $\alpha$ and $L$ such that
\begin{equation}\label{eq:cor:c1}
\|u\|_{C^{1,\alpha}(B_{\bar{r}})}  \le C\big(
\|u\|_{L^{2}(B_{1},\delta^a)}
+\|f\|_{L^{p}(B_{1},\delta^a)}
+\|F\|_{C^{0,\alpha}(B_1)}
\big).
\end{equation}
Moreover, denoting by $T_z\Gamma$ the tangent space to $\Gamma$ at the point $z\in\Gamma$, we have that $u$ satisfies the following boundary condition for every $z\in \Gamma\cap B_{\bar{r}}$,
\begin{equation}\label{eq:BC:curve}
(A\D u +F) (z) \cdot \nu(z)= 0,\quad \text{ for every } \nu(z)\perp T_z{\Gamma}.
\end{equation}
\end{Corollary}

Since the proofs of Corollaries \ref{C:0:alpha:curve} and \ref{C:1:alpha:curve} are quite similar, we merely provide the proof of the Corollary \ref{C:1:alpha:curve}. Additionally, we refer to \cite{Fio24}, where the same results are proved in a closely related context.
\begin{proof}
The proof relies on making changes of variables to reduce the problem to the \emph{flat} case, where the weight becomes $|y|^a$. The desired result then follows by applying Theorem \ref{T:1:alpha}.

Let $\Phi$ be as in \eqref{eq:diffeomorphism}, and define $\tilde{u}=u\circ\Phi$. Let $\bar{r} \in (0,1)$ to be chosen later, and let $\phi \in C_c^\infty(B_{2\bar{r}})$. Using the change of variable $z=\Phi(z)$, we obtain that 
\begin{equation*}
0=\int_{B_1}\delta^a \big(A\D u\cdot\D\phi-f\phi+F\cdot\D\phi\big) dz
=\int_{B_1}|y|^a\big(\tilde A\D\tilde u\cdot\D\tilde\phi-\tilde f\tilde\phi+\tilde F\cdot\D\tilde \phi\big) dz,
\end{equation*}
where $\tilde{A}$ is defined in \eqref{eq:tilde:A} and 
\[\tilde{f}:= \tilde{\delta}^a f\circ\Phi \in L^{p,a}(B_{2\bar{r}}) ,\quad \tilde{F}:=  \tilde{\delta}^a (J_\Phi^{-1}) F\circ\Phi \in C^{0,\alpha}(B_{2\bar{r}}),  \quad \tilde{\phi}:= \phi\circ\Phi.
\]
By using the properties of the defining function (see Definition \ref{D:admissible:weight}), one has that
$\tilde{A}\in C^{0,\alpha}(B_{2\bar{r}})$. Moreover, by Remark \ref{R:tilde:A}, for every $\sigma>0$ there exists $\bar{r}>0$ such that  $\tilde{A}$ has ellipticity constant restricted to $B_{2\bar{r}}\cap \Sigma_0$ given by 
\[0<\tilde{\delta}^a(0)\lambda_*-\sigma\le\tilde{\delta}^a(0)\Lambda_*+\sigma.\] 
Let us consider $\bar{r}$ small enough such that
\[
\alpha< \alpha_*\Big(n,a,\frac{\tilde{\delta}^a(0)\lambda_*-\sigma}{\tilde{\delta}^a(0)\Lambda_*+\sigma}\Big)-1,
\]
where $\alpha_*$ is defined in \eqref{alphastar2}. 
This choice is always possible since $\tilde{A}(0)$ has ellipticity constant $\tilde{\delta}^a(0)\lambda_*\le\tilde{\delta}^a(0)\Lambda_*$ on $\Sigma_0$ and $\alpha <\alpha_*(n,a,\lambda_*/\Lambda_*)-1$ (see Remark \ref{R:tilde:A}).

Resuming, we have shown that $\tilde{u}$ is a weak solution to 
\[
-\dive(|y|^a\tilde{A}\D \tilde{u})= |y|^a \tilde f+\dive(|y|^a\tilde{F}),\quad \text{ in }B_{2\bar{r}}.
\]
By the choice of $\bar{r}$, after rescaling the domain, the assumptions of Theorem \ref{T:1:alpha} are satisfied. Hence, we have that $\tilde{u}$ satisfies \eqref{eq:stima:c1} and \eqref{eq:c1:boundary:condition} and composing back with $\Phi^{-1}$ we get that $u$ satisfies \eqref{eq:cor:c1} and \eqref{eq:BC:curve} in $B_{\bar{r}}$. 
\end{proof}

\appendix
\section{Geometry of perforated domains}\label{sec:A}

In this section, we work in the unit ball $B_1$ to simplify the notation, but the same results hold at any scale $R > 0$, in which case the constants involved also depend on $R$.

Let $ A \in C^1(B_1; \R^{d,d}) $ (see Remark \ref{R:A3regularity} for the sharp requirement) be a symmetric matrix satisfying the uniform ellipticity condition \eqref{eq:unif:ell}, and recall the notation \eqref{Blocks}.
The block $A_3$ is still symmetric and satisfies
\[
 \lambda|\zeta|^2 \le A_3(z)\zeta \cdot \zeta \le \Lambda|\zeta|^2\,, \quad \text{ for a.e. } z \in B_1 \quad \text{ and every }\zeta \in \R^n \,,
\]
where $\lambda$ and $\Lambda$ are the uniform ellipticity constants of $A$.
By the ellipticity condition, $ A_3 $ is invertible, and $ A_3^{-1} \in C^1(B_1; \R^{n,n}) $. The matrix $ A_3^{-1} $ is also symmetric and satisfies
\begin{equation}\label{eq:ueQ}
\Lambda^{-1}|\zeta|^2 \le A_3^{-1}(z)\zeta \cdot \zeta \le \lambda^{-1}|\zeta|^2\,, \quad \text{ for a.e. } z \in B_1 \quad \text{ and every }\zeta \in \R^n \,.
\end{equation}
The square root of $ A_3^{-1} $, denoted by $ A_3^{-\frac12} $, is well-defined and belongs to $ C^1(B_1; \R^{n,n}) $ (see for instance \cite[Theorem 2.5]{DiPaPu22}). It is also symmetric and satisfies
\begin{equation}\label{eq:ueP}
 \Lambda^{-\frac12}|\zeta|^2\le A_3^{-\frac12}(z)\zeta\cdot\zeta\le\lambda^{-\frac12}|\zeta|^2\,, \text{ for a.e. } z \in B_1 \quad \text{ and every }\zeta \in \R^n \,.
\end{equation}
We will frequently use the following properties, which follow directly from \eqref{eq:ueQ} and  \eqref{eq:ueP}.
\begin{equation}\label{eq:app:prop}
\begin{gathered}
|A_3^{-\frac12}(z) y| = \sqrt{A_3^{-1}(z) y \cdot y}\,,\quad
\Lambda^{-\frac12}|y| \le |A_3^{-\frac12}(z)y|\le \lambda^{-\frac12}|y|\,, \quad
\Lambda^{-1}|y| \le |A_3^{-1}(z)y|\le \lambda^{-1}|y|\,.
 \end{gathered}
\end{equation}
In the following lemma, we describe some properties of the $(\eps, A)$-neighborhood of $\Sigma_0$
\[
\Sigma^A_\e=\{z \mid A_3^{-1}(z)y \cdot y\le\e^2\}\,,
\]
and its boundary 
\[
\partial \Sigma^A_\e=\{z \mid A_3^{-1}(z)y \cdot y=\e^2\}\,.
\]

\begin{Lemma}\label{L:normal}
Let $ A \in C^1(B_1; \R^{d,d}) $ and $L:=\|A_3^{-1}\|_{C^1(B_1)}$. There exists $\eps_0 >0$, depending only on $d$, $n$, $L$, $\lambda$ and $\Lambda$, such that for every $\eps \in (0, \eps_0]$ the following properties hold:
\begin{enumerate}
\item[$i)$] the normal vector $\nu(z)$ to $\partial \Sigma^A_\e$ at $z\in B_1 \cap \partial \Sigma_\eps^A$ is well-defined. In particular,
\begin{equation}\label{eq:N+tilde:nu}
    \nu(z) = (0, N(z)) + \tilde \nu(z)\,,
\end{equation}
where 
\begin{equation}\label{eq:N}
N(z) = \frac{A_3^{-1}(z)y}{|A_3^{-1}(z)y|} \in \mathbb{S}^{n-1}\,, \quad \text{and}\quad |\tilde \nu(z)| \leq c\eps\,,
\end{equation}
for some constant $c = c(d, n, \Lambda, L, \eps_0)>0 $;
\item[$ii)$] it holds
\begin{equation}\label{eq:nuyest}
\nu(z) \cdot \frac{y}{|y|} \geq \frac{\lambda}{2\Lambda}\,;
\end{equation}
\item[$iii)$] there exists a constant $c>0$ depending only on $L$, $d$, $n$, $\lambda$ and $\Lambda$ such that
\[
[N]_{C^{0,1}(B_1 \cap \partial \Sigma^A_\e)} \leq c \eps^{-1}\,;
\]
\item[$iv)$] if in addition $ A \in C^{1,\beta}(B_1; \R^{d,d}) $ for some $\beta \in(0,1)$, then 
\[
[\tilde \nu]_{C^{0,\beta}(B_1 \cap \partial \Sigma^A_\e)}\leq c\,,
\]
where $c>0$ depends only on $\|A_3^{-1}\|_{C^{1,\beta}(B_1)}$, $d$, $n$, $\lambda$ and $\Lambda$.
\end{enumerate}
\end{Lemma}
\begin{proof}
Define the function  
\begin{equation}\label{PsiA1}
\Psi(z) = \sqrt{A_3^{-1}(z)y\cdot y}\,.
\end{equation}
From \eqref{eq:app:prop}, we have $\Psi \in C^{0,1}(B_1) \cap C^1(B_1 \setminus \Sigma_0)$, and $\Psi(x,y) = 0$ if and only if $y=0$. 

Since $\partial \Sigma^A_\eps = \{\Psi(z) = \eps\}$, the normal vector $\nu(z)$ at a point $z \in \partial \Sigma_\eps^A$ is given by 
\[
\nu(z) = \frac{\nabla \Psi(z)}{|\nabla \Psi(z)|} = \frac{\nabla (A_3^{-1}(z) y \cdot y)}{|\nabla (A_3^{-1}(z) y \cdot y)|}\,.
\]
Let $A_3^{-1}(z) = (b_{ij}(z) )_{i,j = 1, \ldots, n}$. We define the functions $G: B_1 \to \R^d$ and $H: B_1 \to \R$ as follows
\[
G(z) = \frac12\big(\sum_{i,j=1}^n \partial_{z_\ell}b_{ij}(z)y_iy_j\big)_{\ell=1, \ldots, d}\,,\qquad
H(z) = |A_3^{-1}(z)y|- |(0, A_3^{-1}(z)y) + G(z)| \,.
\]
Using this notation, the normal vector can be expressed as
\[
\nu(z)= \frac{(0, A_3^{-1}(z)y) + G(z)}{|A_3^{-1}(z)y| - H(z)}\,.
\]
There exists a constant $c_1 = c_1(d, n)>0$ such that 
\begin{equation}\label{eq:app:pass44}
|H(z)| \leq |G(z)| \leq c_1 L |y|^2\,, \quad \text{ for every $z \in B_1$}\,.
\end{equation}
Fix $\eps_0$ such that $\Lambda^{-1}\lambda^{\frac12} - c_1L\Lambda\eps_0 >0$. Thus, if $\eps \leq \eps_0$ and $z \in B_1 \cap \partial \Sigma_\eps^A$, thanks to \eqref{eq:app:prop} and \eqref{eq:app:pass44} we have
\begin{equation}\label{eq:app:pass55}
|A_3^{-1}(z)y| - H(z)\geq \Lambda^{-1}|y| - cL|y|^2 \geq \eps(\Lambda^{-1}\lambda^{\frac12} - cL\Lambda\eps) \geq c\eps\,.
\end{equation}
Therefore, $\nu$ is well-defined.

Define $\tilde \nu(z)  := \nu(z)  - (0, N(z))$, where $N(z) $ is as in \eqref{eq:N}. Then
\[
\tilde \nu(z) = \frac{H(z)N(z) }{|A_3^{-1}(z)y| - H(z)} + \frac{G(z)}{|A_3^{-1}(z)y| - H(z)} \,.
\]
It follows that 
\[
|\tilde \nu(z)| \leq c\eps\,,
\]
where $c = c(d, n, \Lambda,\lambda, L, \eps_0) >0$. Thus $i)$ is proved. 

\smallskip

To prove $ii)$, we compute 
\[
\nu \cdot \frac{y}{|y|} = N(z) \cdot \frac{y}{|y|} + \tilde \nu \cdot \frac{y}{|y|} \geq \frac{\lambda}{\Lambda} - |\tilde \nu|\,.
\]
Taking if necessary a smaller $\eps_0$, so that $|\tilde \nu| \leq {\lambda}/{2\Lambda}$, we obtain \eqref{eq:nuyest}.

\smallskip

For $iii)$, note that for every $z=(x, y), z'= (x', y') \in B_1$ it holds
\begin{equation}\label{eq:app:pass11}
|A_3^{-1}(z) y - A_3^{-1}(z') y' | \leq |A_3^{-1}(z) - A_3^{-1}(z')||y| + |A_3^{-1}(z')| |y-y'| \leq 2L|z-z'|\,.
\end{equation}
Therefore, for all $z, z' \in B_1\cap \partial \Sigma^A_\e$, we have
\[
|N(z) - N(z')| \leq 2 \frac{|A_3^{-1}(z) y - A_3^{-1}(z') y' |}{|A_3^{-1}(z') y' |} \leq c \eps^{-1} |z-z'|\,,
\]
where $c >0$ depends only on $L$, $\Lambda$, $\lambda$.

\smallskip

Finally, we prove $iv)$. Assume that $A_3^{-1}\in C^{1,\beta}(B_1; \R^{n,n})$. Using similar computations as before, we obtain 
\begin{equation}\label{eq:app:pass33}
|G(z) - G(z') | \leq c\eps|z-z'|^\beta\,, \quad \text{ for every }\quad z, z'\in B_1\cap \partial \Sigma_\eps^A\,,
\end{equation}
where $c$ depends on $d$, $n$, $\lambda$, $\Lambda$, $\|A^{-1}_3\|_{C^{1,\beta}(B_1)}, \eps_0$.

Define the function
\[
h(z) = 2 \frac{N(z) \cdot G(z)}{|A_3^{-1}(z)y|} + \frac{|G(z)|^2}{|A_3^{-1}(z)y|^2}\,.
\]
By combining \eqref{eq:app:pass44}, \eqref{eq:app:pass11}, \eqref{eq:app:pass33}, and $iii)$, we deduce 
\begin{equation}\label{eq:app:pass77}
|h(z)| \leq c|y|\, \text{ for all } z \in B_1\,, \quad \text{and }\quad |h(z)- h(z')| \leq c|z - z'|^\beta\, \text{ for all } z,z' \in B_1\,.
\end{equation}
Thus, when $\eps$ is sufficiently small, we can apply the Taylor expansion for $(1 + t)^{\frac12}$ to $|(0, A_3^{-1}(z)y) + G(z)|$, yielding for every $z \in B_1\cap \partial \Sigma_\eps^A$,
\[
H(z) = |A_3^{-1}(z)y| \sum_{j=1}^\infty c_j h^j(z)\,,
\]
where $c_j$ are the coefficients of the expansion. Using \eqref{eq:app:pass44} and \eqref{eq:app:pass77}, we infer that  
\begin{equation}\label{eq:app:pass22}
|H(z) - H(z')| \leq c \eps |z-z'|^\beta\,.
\end{equation}
Finally, combining \eqref{eq:app:pass44}, \eqref{eq:app:pass55}, \eqref{eq:app:pass11}, \eqref{eq:app:pass33} and \eqref{eq:app:pass22}, we obtain 
\[
|\tilde \nu (z) - \tilde \nu(z)| \leq c(|z-z'| + \eps^{-1}|H(z) - H(z')| + \eps^{-1}|G(z) - G(z')|) \leq c |z-z'|^\beta\,.
\] 
\end{proof}

\begin{Lemma} \label{L:BC}
Let $G: B_1 \setminus \Sigma_{\eps}^A \to \R^d$ be such that $G \in C^{0,\alpha}(B_1 \setminus \Sigma_{\eps}^A)$ for some $\alpha\in(0,1)$, and $G(z) \cdot \nu(z) = 0$ for every $z \in B_1 \cap \partial \Sigma_\e^A $. Then, there exist a constant $c>0$ depending only on $d$, $n$, $\lambda$, $\Lambda$, $\alpha$ and $L:=\|A_3^{-1}\|_{C^1(B_1)}$ such that for every $0<\eps \ll 1$  and $i=1,\dots,n$, it holds
\begin{equation*}
    \|G\cdot e_{y_i}\|_{L^{\infty}(B_{1/2}\cap \partial\Sigma_{\eps}^A)}\le c \e^\alpha \|G\|_{C^{0,\alpha}(B_{1/2}\setminus\Sigma_{\eps}^A)}\,.
\end{equation*}
\end{Lemma}
\begin{proof}
We recall that by Lemma \ref{L:normal}, for $\eps \leq \eps_0$, the normal vector to the surface $\partial \Sigma_\e^A$ at $z = (x, y) \in B_1 \cap \partial \Sigma_\e^A$ is given by $\nu(z) = (0, N(z)) + \tilde \nu(z)$, where $ \tilde \nu(z) \leq c\eps$ and
\[
N(x, y) =  \frac{A_3^{-1}(z) y}{|A_3^{-1}(z)y|} \in \R^n\,.
\] 
In what follows, we show that there exist $\delta_1 = \delta_1(\eps_0, \Lambda)$, $\delta_2 = \delta_2 (L, \Lambda)$, such that for every $x \in B_{1/2}^{d-n}$ and for $\eps \leq \eps_0$ there exist $\bar x \in B_{1/2}^{d-n}$ and $n$ vectors $\{y_{j}\}_{j=1,\dots,n}\subset\R^n$ (depending on $\bar x$) such that, defined $\tau_j = N(\bar x, y_j)$, it holds
\[
|x - \bar x| \leq \delta_1\eps\,,\qquad 
(\bar x,y_{j})\in B_{1/2}\cap \partial\Sigma_{\e}^A\,,\qquad
|\tau_i\cdot \tau_j| \leq \delta_2 \eps\,, \text{ for every } i\neq j\,.
\]
Note that this implies that ${\rm span}\{\tau_j\} = \R^n$. 
Take $\delta_1 >0$ such that $\Lambda^{\frac12} < \delta_1 < (2\eps_0)^{-1}$ and define 
\[
\bar x = (1 - 2\delta_1 \eps) x. 
\]
By construction $1 -2\delta_1 \eps >0$ and $|x-\bar x| \leq \delta_1 \eps$. Moreover, let $y \in \R^d$ be such that $(\bar x, y) \in \partial \Sigma_\eps^A$. It holds
\[
|(\bar x, y)| \leq |\bar x| + |y| = (1-2\delta_1 \eps)|x| + |y| \leq \frac{1}{2} - \big( \delta_1 - \Lambda^{\frac12})\eps < \frac12\,.
\]
Hence $(\bar x, y) \in B_{1/2}\cap\partial \Sigma_\eps^A$, as needed. 

Now, let us fix any point $(\bar x, y_1)  \in \partial\Sigma_{\e}^A$, and define $\tau_1 = N(\bar x, y_1)$. One can choose  $\sigma_{2}\in \mathbb{S}^{n-1}$ such that $\sigma_2 \cdot A_3^{-2}(\bar x, y_1)y_1 = 0$ (where $A^{-2}_3 = (A^{-1}_3)^2 = (A_3^2)^{-1}$), and define $y_2 = r_2 \sigma_2$, where $r_2$ is chosen in such a way that $(\bar x, y_2) \in \partial \Sigma_\eps^A$.

To show that such $r_2$ exists, consider the function $f(r) = \Psi(x, r\sigma)$, where $\Psi$ is defined in \eqref{PsiA1}. Since $f$ is continuous, $f(0) = 0$, and $\lim_{r \to \infty}f(r) = \infty$, there exists $r_2$ such that $f(r_2) = \eps$, that is, $y_2 = r_2 \sigma$. 

Let us define $\tau_2 = N(\bar x, y_2)$. Note that, by construction, 
\[
A_3^{-1}(\bar x, y_1) y_1 \cdot A_3^{-1}(\bar x, y_1) y_2 = r_2 A_3^{-2}(\bar x, y_1) y_1 \cdot \sigma_2 = 0\,.
\]
Therefore
\[
\begin{aligned}
    |\tau_1 \cdot \tau_2| = |N(\bar x,y_1)\cdot N(\bar x,y_2)| &= \frac{|A_3^{-1}(\bar x,y_1) y_1\cdot A_3^{-1}(\bar x, y_2) y_2|}{|A_3^{-1}(\bar x,y_1) y_1||A_3^{-1}(\bar x, y_2) y_2|}\\
    &\le\frac{|A_3^{-1}(\bar x, y_2) - A_3^{-1}(\bar x, y_1)| |y_2|}{|A_3^{-1}(\bar x, y_2) y_2|} \leq L\Lambda|y_2 - y_1| \leq \delta_2\eps\,,
\end{aligned}
\]
where $\delta_2$ depends only on $L$, $\Lambda$. 

Next, fix ${y}_{3}$ such that $(x,y_{3})\in\partial\Sigma_{\e}^A$ and ${y}_{3}$ is orthogonal to both $A_3^{-2}(\bar x, y_1)y_1$ and $A_3^{-2}(\bar x, y_2)y_2$, and define $\tau_3 = N(\bar x, y_3)$. Performing the same computation as before, we find that $\tau_3\cdot \tau_1 \leq \delta_2\eps$, $\tau_3\cdot \tau_2 \leq \delta_2\eps$. To conclude, it suffices to iterate this argument a finite number of times.

We are now in position to prove the lemma. Fix $z = (x, y) \in B_{1/2}\cap\partial \Sigma_{\eps}^A$. Recall that $\tau_j = N(\bar x, y_j)$, where $(\bar x, y_j) \in B_{1/2}\cap\partial \Sigma_{\eps}^A$. Let $T = (t_{ij})$ denote the matrix with entries $t_{ij} = \tau_i \cdot \tau_j$. Let $\|\cdot \|_{\R^{n,n}}$ be a chosen matrix norm. By construction, we have $\|T - \mathbb{I}_n\|_{\R^{n,n}} \leq c\eps$, where $c$ depends only on $n$, $\delta_2$. Thus, for $\eps$ sufficiently small, $\det T >0$ and $T$ is invertible. 

Fix $i \in \{1, \ldots, n\}$. We can decompose $e_{y_i} = \sum_{j=1}^n \alpha_j \tau_j$ with respect to the basis ${\tau_j}$, where the coefficients $\alpha = (\alpha_j)_{j=1, \ldots n}$ satisfy the linear system $T\alpha = \beta$, with $\beta = (e_{y_i}\cdot \tau_j)_{j=1, \ldots, n}$. Since $T$ is invertible, it follows that $|\alpha_i| \leq c$, where $c$ does not depend on $\eps$. Therefore, we have
\[
\begin{aligned}
|G(z) \cdot e_{y_i}| &= |G(z) \cdot \sum_{j=1}^n\alpha_j\tau_j| \leq c\sum_{j=1}^n |G(z) \cdot (0,  N(\bar x, y_j))| \\
&\leq c \Big(\sum_{j=1}^n |G(z) \cdot \nu(\bar x, y_j)| + \sum_{j=1}^n |G(z) \cdot \tilde \nu(\bar x, y_j)|\Big)\,.
\end{aligned}
\]
Next, using that $G(z) \cdot \nu(z) = 0$ for every $z \in B_{1/2} \cap \partial \Sigma_{\eps}^A $, and that $ \tilde \nu(z) \leq c\eps$, we get 
\[
\begin{aligned}
|G(z) \cdot e_{y_i}|& \leq c\Big( \sum_{j=1}^n |(G(z) - G(\bar x, y_j))  \cdot \nu(\bar x, y_j)| +\eps \|G\|_{L^\infty(B_{1/2}\setminus\Sigma_{\eps}^A)}\Big)\\
& \leq c\Big([G]_{C^{0,\alpha}(B_{1/2}\setminus\Sigma_{\eps}^A)} \sum_{j=1}^n(\delta_1 \e+|y - y_j|)^\alpha + \eps \|G\|_{L^\infty(B_{1/2}\setminus\Sigma_{\eps}^A)}\Big)\,,
\end{aligned}
\]
where we used that, by construction, $(\bar x, y_i) \in B_{1/2}\cap \partial \Sigma_\eps^A$.
Finally, since $|y|, |y_j| \leq c \eps$ by the uniform ellipticity condition, we obtain 
\[
|G(z) \cdot e_{y_i}| \leq c \eps^\alpha \|G\|_{C^{0,\alpha}(B_{1/2}\setminus\Sigma_{\eps}^A)}\,.
\]
\end{proof}

\begin{Lemma}\label{L:appendix3}
Let $ A \in C^1(B_1; \R^{d,d}) $, and let $\Phi \in C^1(B_1; \R^d)$ given by 
\[
\Phi(z) = (x, A_3^{-\frac12}(z)y)\,.
\]
Then, there exists $0 < R\leq 1 $  which depends on $L:=\|A\|_{C^1(B_1)}$, $\lambda$, $\Lambda$, $d$, $n$ such that the following holds: 
\begin{enumerate}
\item[$i)$] the function $\Phi: B_{ R} \to \R^d$ is injective. Moreover, there exists $c>0$ such that
\begin{equation}\label{eq:app:inj}
c^{-1}|z_1 - z_2| \leq |\Phi(z_1) - \Phi(z_2)| \leq c |z_1 - z_2|\,, \quad \text{ for every }z_1, z_2 \in B_{ R};
\end{equation}
\item[$ii)$] $\Phi: B_{ R} \to \Phi(B_{ R})$ is a $C^1$-diffeomorphism and it holds 
\[
\frac12 \Lambda^{-\frac{n}{2}} \leq \det  J_\Phi(z) \leq \frac32\lambda^{-\frac{n}{2}}\,;
\]
\item[$iii)$] the matrix $J_\Phi^\top J_\Phi$ is symmetric and  uniformly elliptic. In particular
\[
\frac{1}{2}\min\{1, \Lambda^{-1}\}|\zeta|^2 \leq J_\Phi^\top J_\Phi(z) \zeta \cdot \zeta \leq 2\max\{1, \lambda^{-1}\} |\zeta|^2\quad \text{ for every } \zeta \in \R^d\,;
\] 
\item[$iv)$] Let $\eps < \eps_0$, where $\eps_0$ is as in Lemma \ref{L:normal}. Given $z \in B_R \cap\partial \Sigma_\eps^A$, consider the projection map $ \Pi_\eps(z): \R^d \to T_z\partial \Sigma^A_\eps \simeq \R^{d-1}$ given by
\[
\Pi_\eps(z) \xi = \xi - (\xi \cdot \nu(z)) \nu(z) \,,\quad \xi \in \R^d\,.
\]
Then 
\[
\frac{(2^{-1}\min\{1, \Lambda^{-1}\})^d}{2\max\{1, \lambda^{-1}\}} \leq \det  (\Pi_\eps \circ J_\Phi^\top J_\Phi\circ\Pi_\eps) \leq \frac{(2\max\{1, \lambda^{-1}\})^d}{2^{-1}\min\{1, \Lambda^{-1}\}}\,.
\]

\end{enumerate}
\end{Lemma}
\begin{proof}
Let us prove $i)$. To ease the notations, we call $P = A_3^{-\frac12}$. Denote $P=(p_{ij})_{i,j=1,\ldots n}$ and $L = \| P\|_{C^{1}(B_1)}$. Let $z_1, z_2 \in B_{R_1}$, where $0 <R_1 \leq 1$ will be specified later. Obviously 
\[
\begin{aligned}
|\Phi(z_1) - \Phi(z_2)| &\leq |x_1 - x_2| + |(P(z_1) - P(z_2))y_1| + |P(z_2)(y_1 - y_2)| \\
&\leq |x_1 - x_2| + L|z_1 - z_2| +  \lambda^{-\frac12}|y_1 - y_2| \leq c|z_1 - z_2|\,,
\end{aligned}
\]
where $c >0$ depends only on $L$, $\lambda$. 

By similar computations,
\[
|P(z_1)y_1-P(z_2)y_2|\ge |P(z_1)(y_1-y_2)|-|(P(z_1)-P(z_2))y_2| \ge \Lambda^{-\frac{1}{2}}|y_1-y_2|-L|z_1-z_2||y_2|,
\]
hence, 
\[
|\Phi(z_1) - \Phi(z_2)|\ge \min\{1,\Lambda^{-\frac{1}{2}}\}|z_1-z_2|-LR_1|z_1-z_2|\ge c^{-1}|z_1-z_2|,
\]
by choosing $R_1>0$ small enough, such that $\min\{1,\Lambda^{-\frac{1}{2}}\}-LR_1>c^{-1}$. Then,  \eqref{eq:app:inj} follows and the map $\Phi$ is injective.

\smallskip

Let us prove $ii)$. The Jacobian of $\Phi$ is given by
\[
J_\Phi  = 
\begin{pmatrix}
\mathbb{I}_{d-n} & 0 \\
P_x & P + P_y
\end{pmatrix},
\]
where
\[
\begin{gathered}
P_x  = (p^x_{ij})_{i= 1,\ldots n, j = 1 \ldots, d-n}\, \qquad p^x_{ij} = \sum_{k=1}^n \partial_{x_j}p_{i,k}(z)y_k\,,\\
P_y  = (p^y_{ij})_{i,j = 1, \ldots, n}\, \qquad p^y_{ij} = \sum_{k=1}^n \partial_{y_j}p_{i,k}(z)y_k\,.
\end{gathered}
\]
We compute (recall that $P = A_3^{-\frac12}$ is invertible)
\begin{equation}\label{eq:detPhifor}
\det  J_\Phi = \det (P + P_y) = \det (P)\det (\mathbb{I}_n + P^{-1}P_y)\,.
\end{equation}
First we notice that, thanks to \eqref{eq:ueP}, it holds
\begin{equation}\label{eq:detPest}
 \Lambda^{-\frac{n}{2}}\le \det P(z)\le\lambda^{-\frac{n}{2}}\,.
\end{equation}
Let now $\|\cdot \|_{\R^{n,n}}$ be a chosen matrix norm. We have 
\[
\|P^{-1}P_y\|_{\R^{n,n}} \leq \|P^{-1}\|_{\R^{n,n}} \|P_y\|_{\R^{n,n}}  \leq  cL^2 |y|\,.
\]
Since the determinant is continuous with respect to any matrix norm, we infer that there exists $R_2 \leq R_1$ such that for every $z \in B_{R_2}$ it holds 
\begin{equation}\label{eq:detPvar}
\frac12 \leq \det (\mathbb{I}_n + P^{-1}P_y) \leq \frac32\,.
\end{equation}
By \eqref{eq:detPhifor}, \eqref{eq:detPest} and \eqref{eq:detPvar} we readily get that 
\begin{equation}\label{eq:detPhi}
\frac12 \Lambda^{-\frac{n}{2}} \leq \det  J_\Phi(z) \leq \frac32\lambda^{-\frac{n}{2}}
\end{equation}
holds for every $z \in B_{R_2}$, thus proving $ii)$.

\smallskip

Let now prove $iii)$. Call $M = J_\Phi^\top J_\Phi$. Obviously, $M$ is symmetric and thanks to \eqref{eq:detPhi} is invertible in $B_{R_2}$. Moreover, via a straightforward computation (recall that $P$ is symmetric and $P^2 = A_3^{-1}$) we find that $M = M_1 + M_2$ where 
\[
M_1 = \begin{pmatrix}
\mathbb{I}_{d-n} & 0 \\
0 & A_3^{-1}
\end{pmatrix}
\quad \text{ and } \quad M_2 = \begin{pmatrix}
P_x^\top P_x & (P + P_y)^\top P_x \\
P_x^\top(P + P_y) & P^\top P_y + P_y^\top P + P_y^\top P_y 
\end{pmatrix}\,.
\]
We immediately see that there exists a constant $c = c(L) >0$ such that 
\(
\|M_2\|_{\R^{d,d}} \leq c|y|\,.
\)
Moreover, thanks to \eqref{eq:ueQ} we have
\[
\min\{1, \Lambda^{-1}\}|\zeta|^2 \leq M_1(z) \zeta \cdot \zeta \leq \max\{1, \lambda^{-1}\}|\zeta|^2\,.
\]
Thus
\[
(\min\{1, \Lambda^{-1}\} - c|y|)|\zeta|^2 \leq M(z) \zeta \cdot \zeta \leq (\max\{1, \lambda^{-1}\} + c|y|)|\zeta|^2\,.
\]
Hence, there exists $R \leq R_2$ such that for every $z \in B_{R}$ the matrix $M$ satisfies 
\[
\frac{1}{2}\min\{1, \Lambda^{-1}\}|\zeta|^2 \leq M(z) \zeta \cdot \zeta \leq 2\max\{1, \lambda^{-1}\} |\zeta|^2\quad \text{ for every } \zeta \in \R^d\,.
\] 

\smallskip

As for $iv)$, let us call $N = \Pi_\eps \circ M \circ \Pi_\eps$. We can represent $\Pi_\eps$ via a $\R^{d-1, d}$ matrix $Q$. Thus, $N = QMQ^\top \in \R^{d-1, d-1}$. By the Cauchy interlacing theorem (also known as Poincar\'e separation theorem), the eigenvalues $\{\mu_j\}_{j=1,\ldots, d}$ of $M$ are related to the eignevalues $\{\tilde\mu_i\}_{i = 1, \ldots, d-1}$ of $N$ via the formula
\[
\mu_{i+1} \leq  \tilde \mu_i \leq \mu_{i}\,.
\]
Therefore, using $iii)$, we infer
\[
\frac{(2^{-1}\min\{1, \Lambda^{-1}\})^d}{2\max\{1, \lambda^{-1}\}} \leq \frac{\det M}{\mu_1} \leq \det  N \leq \frac{\det M}{\mu_d} \leq \frac{(2\max\{1, \lambda^{-1}\})^d}{2^{-1}\min\{1, \Lambda^{-1}\}}\,.
\]
This completes the proof of $iv)$ and of the lemma.
\end{proof}


\section*{Acknowledgment}
We would like to thank Susanna Terracini for many fruitful discussions on the topic.
The authors are research fellows of INDAM group GNAMPA and supported by the GNAMPA project E5324001950001. G.C. is supported by the PRIN project 20227HX33Z. S.V. is supported by the PRIN project 2022R537CS.

\end{document}